\documentclass[review]{elsarticle}

\usepackage{amssymb,amsmath,comment}
\usepackage{amsthm}
\usepackage[usenames,dvipsnames]{xcolor}
\usepackage{subfigure}
\usepackage{caption}
\usepackage{algorithm, algorithmicx, algcompatible}
\usepackage{algpseudocode}
\usepackage{graphicx}
\usepackage{tikz}
\usepackage[export]{adjustbox}
\usepackage{epstopdf}

\definecolor{cadmiumgreen}{rgb}{0.0, 0.42, 0.24}

\usepackage{lineno}
\modulolinenumbers[5]
\newdefinition{rmk}{Remark }
\newdefinition{prop}{Proposition }

\makeatletter
\def\@author#1{\g@addto@macro\elsauthors{\normalsize%
    \def\baselinestretch{1}%
    \upshape\authorsep#1\unskip\textsuperscript{%
      \ifx\@fnmark\@empty\else\unskip\sep\@fnmark\let\sep=,\fi
      \ifx\@corref\@empty\else\unskip\sep\@corref\let\sep=,\fi
      }%
    \def\authorsep{\unskip,\space}%
    \global\let\@fnmark\@empty
    \global\let\@corref\@empty  %% Added
    \global\let\sep\@empty}%
    \@eadauthor={#1}
}
\makeatother

\begin{document}
\begin{frontmatter}

\title{Multiscale Mixed Methods with Improved Accuracy: The Role
of Oversampling and Smoothing}

\author[]{Dilong Zhou $^1$}
\cortext[cor1]{Corresponding author}
\ead{dxz200000@utdallas.edu}
\author[]{Rafael T. Guiraldello $^2$}
\ead{rafaeltrevisanuto@gmail.com}
\author[]{Felipe Pereira $^1$ \corref{cor1}}
\ead{luisfelipe.pereira@utdallas.edu}

\address{$^1$ Department of Mathematical Sciences, The University of Texas at Dallas, \\
800 W. Campbell Road, Richardson, Texas 75080-3021, USA \\
$^2$ Piri Technologies, LLC \\
 1000 E. University Ave., Dept. 4311, Laramie, WY 82071-2000, USA}


\begin{abstract}
 
Multiscale mixed methods based on non-overlapping domain decompositions can efficiently handle the solution of significant subsurface flow problems in very heterogeneous formations of interest to the industry, especially when implemented on multi-core supercomputers. 
Efficiency in obtaining numerical solutions is dictated by the choice of interface spaces that are selected: the smaller the dimension of these spaces, the better, in the sense that fewer multiscale basis functions need to be computed, and smaller interface linear systems need to be solved. Thus, in solving large computational problems, it is desirable to work with piecewise constant or linear polynomials for interface spaces. However, for these choices of interface spaces, it is well known that the flux accuracy is of the order of $10^{-1}$.

This study is dedicated to advancing an efficient and accurate multiscale mixed method aimed at addressing industry-relevant problems. A distinctive feature of our approach involves subdomains with overlapping regions, a departure from conventional methods. We take advantage of the overlapping decomposition to introduce a computationally highly efficient smoothing step designed to rectify small-scale errors inherent in the multiscale solution. The effectiveness of the proposed solver, which maintains a computational cost very close to 
its predecessors, is demonstrated through a series of numerical studies. Notably, for scenarios involving modestly sized overlapping regions and employing just a few smoothing 
steps, a substantial enhancement of two orders of magnitude in flux accuracy is achieved with the new approach.

\end{abstract}

\end{frontmatter}
%%\linenumbers

\section{Introduction}\label{intro}

Multiphase flows in subsurface formations are governed by coupled systems of partial differential equations. In the case of two and three-phase flows 
of interest to the oil industry frequently these equations are decoupled by operator splitting methods (see \cite{douglas1997numerical, operator2009} for such methods developed, respectively, for two and three-phase flows) and a convection-diffusion equation needs to be solved 
along with a second-order elliptic equation. The focus of this work is on the numerical solution of these elliptic problems that are challenging because the permeability coefficient may exhibit very large contrast and classical iterative schemes are known to converge slowly.

In this paper, we direct our attention to multiscale methods, which have received significant attention due to their natural parallelizability on multi-core computers. Over the past two decades, various research groups have pioneered specific methodologies integrating domain decomposition with multiscale techniques. These methods generally fall into three distinct categories: Multiscale Finite Volume, Multiscale Finite Element, and Multiscale Mixed Finite Element approaches.

The Multiscale Finite Volume Method was introduced in 1997 to solve elliptic problems encountered in two-phase flows \cite{HouMulti, HouMultivol}. Following this, the Multiscale Finite Element Method was developed in 2003 \cite{jenny2003multi}. In this study, we focus on Multiscale Mixed Finite Element Methods, which encompass several variants. Here we mention
the ones that are directly relevant to our work. The Multiscale Mixed Finite Element Method (MsMFEM) was first formulated in 2003 \cite{chen_hou}, followed by the development of the Multiscale Mortar Mixed Finite Element Method (MMMFEM) in 2007 \cite{arbogast}, and the Multiscale Hybrid-Mixed Finite Element Method (MHM) in 2013 \cite{araya2013multiscale}.
Our research focuses particularly on the Multiscale Robin Coupled Method (MRCM), introduced in 2018 \cite{guiraldello2018multiscale}, which can be considered a generalization of the Multiscale Mixed Method (MuMM) introduced in 2014 \cite{pereira}. This generalization allows for arbitrary and independent interface spaces for the pressure and flux variables.
 
The exceptional scalability of MRCM for solving industry-relevant subsurface flow problems on multi-core supercomputers has been demonstrated in previous studies \cite{Recursive2023, HPC2022}. Recent advancements have introduced new interface spaces utilizing the Singular Value Decomposition \cite{interface2019}, as well as approaches that take advantage of the underlying physics of problems \cite{interface2021}. While the pressure and normal flux components in MRCM may have discontinuities in the resulting fine grid solution, a post-processing method has been devised to address this issue, ensuring continuous normal flux components \cite{velocityMM2020}. With this enhancement, MRCM becomes applicable for approximating multiphase flows in porous media. Additionally, the approximation of two-phase flows with MRCM has been explored in recent publications, which have introduced accurate and computationally efficient methods based on operator splitting and implicit time-stepping methodologies \cite{MMtwo2020, MRCM2022,MPM1,MPM2}.

Robin-type boundary conditions are imposed by MRCM on all local problems defined on a non-overlapping partition of the domain. Based in a domain decomposition methodology \cite{Douglas1993}, MRCM ensures compatibility across the interfaces of adjacent subdomains through weak continuity of pressure and normal components of the flux. At the heart of the MRCM framework lies a critical parameter, denoted as $\alpha$, which governs the relative importance of the pressure and the normal component of the flux in the specification of Robin boundary conditions. This parameterization enables MRCM to replicate outcomes from both the Multiscale Mortar Mixed Finite Element Method (MMMFEM) and the Multiscale Hybrid-Mixed Finite Element Method (MHM). In the case of MMMFEM, pressure continuity along the interface is preserved at the fine-grid scale, while flux continuity is weakly satisfied on a larger scale, typically corresponding to the size of each subdomain. By assigning a higher weight to pressure continuity via $\alpha$, MRCM converges towards MMMFEM. Conversely, in MHM, flux continuity along the interface is achieved at the fine-grid scale. Adjusting $\alpha$ to prioritize flux continuity leads MRCM to converge towards MHM
\cite{guiraldello2018multiscale}.

However, even with the latest version of MRCM utilizing non-overlapping partitions, a notable resonance error persists along the interface of solutions obtained with
multiscale mixed methods. To mitigate this error, the oversampling method has been introduced. Initially employed in the Multiscale Finite Volume Method in 1997 \cite{first-over} and subsequently substantiated in 1999 \cite{oversampleprove1,oversampleprove2}, the oversampling method has undergone further development in recent years.
Additional references concerning recent advancements in Multiscale Finite Volume methods with oversampling are available in \cite{2004overfv1, 2007overfv}. Numerous references exist for Multiscale Finite Element methods incorporating oversampling. References dedicated to elliptic equations can be located in \cite{2004overfe1, 2004overfe2, 2005overfe1, 2008over1, 2008overfe2, 2011over, 2013overfe2, 2017over1} and papers focusing on two-phase flow can be found in \cite{2005over, 2006overfe1, 2008overfe1}. 
In the area of Multiscale Finite Element methods with oversampling, specialized methodologies such as Extended Multiscale Finite Element \cite{2010overfe1, 2012over, 2013overfe1}, Generalized Multiscale Finite Element \cite{2021over}, and Multiscale DG methods \cite{2005book} have been also explored.

In the realm of multiscale mixed methods, the first application of the oversampling technique in MsMFEM dates back to 2003 \cite{chen_hou}. In this work the authors considered rapidly oscillating coefficients to address significant resonance errors along the interface. Subsequent advancements have further refined oversampling techniques to enhance accuracy and performance. In 2012, an innovative approach introduced a special term into the Multiscale Basis Function, using Green's function to mitigate resonance errors  \cite{ExpandedMixed}.
Expanding on these efforts, a paper in 2017 integrated upscaling techniques into oversampling methodologies, to approximate two-phase flow problems rather than solely focusing on elliptic equations \cite{over2017}. This approach shows the importance of considering the entire domain in addressing multiscale challenges. Meanwhile, oversampling for the Mixed Generalized Multiscale Finite Element Method was introduced in 2015 
\cite{GMFE2015}, with further advancements in 2017 \cite{GMFE2017}. These works utilized both offline and online stages in the numerical solution to enhance accuracy. The offline stage involves compiling boundary conditions for each multiscale basis function using preset spaces such as piecewise constant or piecewise linear functions. In contrast, the online stage utilizes results from the offline stage to update boundary values.

%%%DILONG%%%%
%\textcolor{red}{(RAFAEL: please rewrite this paragraph as you see fit!) }
%
%The method we introduce is nonconforming in the sense that the test space used to impose weakly the continuity of the pressure and normal 
%component of the flux is different from the space used to specify boundary conditions for multiscale basis functions. Nonconforming multiscale
%methods can be found in \cite{DILONG}. Some of them, such as \cite{DILONG} also take advantage of oversampling techniques.
%
% (DILONG WILL LOOK FOR ADDITIONAL REFS FOR THIS PARAGRAPH).
 %%%DILONG%%%%
%% Our contributions
In our current study, our primary focus lies on enhancing the Multiscale Robin Coupled Method (MRCM) by integrating two distinct strategies aimed at improving the accuracy of numerical solutions while preserving the computational efficiency inherent to MRCM \cite{Recursive2023, HPC2022}.
Our first significant contribution involves the integration of the oversampling technique into MRCM. This novel approach necessitates non-trivial modifications to the original method, as multiscale basis functions now need to be computed in terms of novel informed spaces. Despite these modifications, compatibility conditions are maintained and enforced on the skeleton of an underlying non-overlapping partition of the domain within a nonconforming
multiscale approach.
Our second contribution involves the introduction of a smoothing step in the framework of multiscale mixed methods, a concept commonly employed in the context of overlapping Schwarz domain decomposition methods \cite{DDintro, DomainD}. This step serves as a tool to rectify small-scale errors inherent in the multiscale solution, further enhancing the accuracy of our numerical approach. The computational cost of these steps is very small because a factorization computed in the first step
can be reused.

Following the description of the novel procedure, we conduct testing in two distinct examples. The first example has an analytical solution, while the second example addresses problems relevant to the oil industry. Through comprehensive studies, we evaluate various aspects including convergence rates, the significance of oversampling, the impact of smoothing steps, the combined effect of oversampling and smoothing, and the influence of the $\alpha$ parameter on numerical solutions.
Our findings reveal that in scenarios characterized by modestly sized overlapping regions and utilizing only a few smoothing steps, our new approach yields a remarkable enhancement in flux accuracy, with improvements reaching two orders of magnitude compared to previous methods. Moreover,
results produced with the new approach with piecewise-constant interface spaces are comparable to MRCM results obtained with piecewise linear spaces.

%% organization of the paper
The paper is structured into several sections as follows. The first section focuses on the formulation of the new MRCM with oversampling with subsections dedicated to detailing the formulation of the method along with its well-posedness, and the definition of multiscale basis functions. Following this, the next section describes the concept of a smoothing step, outlining its significance and implementation within the context of the MRCM methodology. 
  In the subsequent section, we present the numerical strategy employed to maximize the computational efficiency of the previously introduced methods.
    The subsequent section is dedicated to numerical studies and is subdivided into subsections that present the results of our numerical experiments. We consider a problem with an analytical solution and another problem involving a permeability field derived from the SPE 10 project. Through these numerical studies, we aim to assess the effectiveness and performance of the proposed enhancements to the MRCM in various scenarios.

\section{The MRCM with Oversampling (MRCM-O)}

We consider single-phase flow in porous media. The governing equations for pressure $p$ and Darcy velocity ${\bf u}$ are given by
\begin{eqnarray}
{\bf u} & = & -\, K\,\nabla p \qquad \mbox{in}~\Omega \label{eq1a}\\
\nabla \cdot {\bf u} & = & f \qquad \mbox{in}~\Omega \label{eq1b}\\
p & = & g \qquad \mbox{on}~\partial \Omega_p \\
{\bf u}\cdot {\bf n} & = & z \qquad \mbox{on}~\partial \Omega_u \label{eq1d}
\end{eqnarray}
where $\Omega\subset \mathbb{R}^d$, $d=2$ or $3$ is the
domain of the problem, $K$ is a symmetric, uniformly positive definite tensor
with components in $L^\infty(\Omega)$,
$f\,\in\,L^2(\Omega)$ is the source term, $g\,\in\,H^{\frac12}(\partial\Omega_p)$
is the pressure condition on the boundary, $z\,\in\,H^{-\frac12}(\partial\Omega_u)$
is the normal velocity condition on the boundary and ${\bf n}$ is the outer normal to $\partial{\Omega}$.

%%%%HERE %%%%%
\subsection{Formulation}
Consider $\mathcal{T}_h$ to be a subdivision of $\Omega\subset \mathbb{R}^d$ into
a Cartesian mesh of $d$-dimensional rectangles.
From this mesh, define partitions of $\Omega$ into non-overlapping subdomains $\{\Omega_i\}_{i=1,\ldots,m}$, and define $\Gamma$, the skeleton of the domain decomposition, to be the union of all interfaces 
$\Gamma_{i,j}=\overline{\Omega}_i\cap\overline{\Omega}_j, \forall i,j=1\dots\,m$. We also define $\Gamma_i = \partial\Omega_i\setminus\partial\Omega$. The restriction of the computational mesh to a subdomain $\omega \subset \Omega$  is denoted by $\mathcal{T}_h^{\omega}$.
For each subdomain $\Omega_i$, define $\hat{\Omega}_i$ to be the augmented subdomain comprising $\Omega_i$ along with an adjoining region.
Figure \ref{overlapjpg} illustrates the two partitions in the case of $d = 2$. Notice that $\bigcup_{i=1}^m \hat{\Omega}_i$ defines an overlapping domain decomposition of $\Omega$. This decomposition is of particular interest to the proposed formulation since it will be used to build the Lagrange multiplier \emph{informed spaces} \cite{guiraldello2018interface}.
 Three length scales appear
in the formulation of the MRCM with oversampling: $\hat H > H \geq h$ that are also illustrated 
in Figure \ref{overlapjpg}. $\hat H$ refers to the size of the oversampling regions, $H$ indicates the
size of non-overlapping subdomains and $h$ is the size of the finest grid used in the numerical 
approximation of Eqs. (\ref{eq1a}) - (\ref{eq1d}). 
We now define the local and interface discrete spaces, followed by a detailed description of the constructing of the Lagrange multiplier spaces based on the overlapping domain decomposition.

For each $\Omega_i$ we define the lowest-order Raviart-Thomas \cite{RaviartThomas::1977}
spaces for velocity and pressure, 
\begin{eqnarray}
{\bf V}_{h}^i& = & \{{\bf v}\,\in\,H(\mbox{div},\Omega_i)~,
~v_j({\bf x})|_K=p_{j1}(x_1)p_{j2}(x_2)\ldots~,\forall\,K\,\in\,\mathcal{T}_h^{\Omega_i},\nonumber \\
& &~~~\mbox{with}~
p_{jk}\,\in\,\mathbb{P}_1~\mbox{if}~j=k~,
p_{jk}\,\in\,\mathbb{P}_0~\mbox{if}~j\neq k~\} \\
Q_h^i& = & \{q\,\in\,L^2(\Omega_i)~,~q({\bf x})|_K\,\in\,\mathbb{P}_0\}~,
\end{eqnarray}
where the space of polynomials of degree $\leq k$ is written as $\mathbb{P}_k$.
%Because the submeshes are a partition of a unique div-conformal mesh,
%the space of normal traces of $\{ {\bf V}_{h}^i\}$ onto the skeleton $\Gamma$ is
%uniquely defined. 
%CORREGIMOS ALGUMAS INCONSISTENCIAS DE NOTACAO POR CAUSA DE QUE A BORDA FISICA
%NAO FAZ MAIS PARTE DO ESKELETO.

Consider $S_h$ to be any subset of edges/faces of $\mathcal{T}_h$ (e.g., $\Gamma$),
then define
\begin{equation}
F_h(S_h) = \{ f:{S}_h\to \mathbb{R}~|~f|_e\,\in\,\mathbb{P}_0~,
~\forall\,e\,\in\,{S}_h \}~.
\end{equation}
%{\color{magenta} Notice that $F_h(\partial \Omega_i)$ is the space of normal traces of ${\bf V}_h^i$.}
Denote by $\mathcal{E}_h$ the set of all faces/edges of $\mathcal{T}_h$ contained 
in $\Gamma$. 
On each edge/face $e\,\in\,\mathcal{T}_h$ we introduce
a unique normal $\check{\bf n}$, which is the exterior normal
to $\partial \Omega$ if $e\,\in\,\partial \Omega$, and if $e\,\in\,\mathcal{E}_h$, 
 then it is defined
as the unit normal exterior to the adjacent subdomain with smallest index, 
$\min \{i,j\}$.
We assume that $\partial \Omega_u$ and $\partial \Omega_p$ 
coincide with subsets of
$\mathcal{T}_h\cap \partial \Omega$ and introduce
\begin{eqnarray}
{\bf V}_{hy}^i&=&\{{\bf v}\,\in\,{\bf V}_{h}^i~,~
{\bf v}\cdot \check{\bf n}=y~\mbox{on}\,\partial\Omega_i\cap\partial\Omega_u
\}~, \label{eq:Vhy}
\end{eqnarray}
where we have assumed that $y$ belongs to $F_h(\partial \Omega_u)$.
 
Before defining the new method we need to introduce coarse (global) 
subspaces of  $F_h(\mathcal{E}_h)$ that will be used in setting consistency
conditions between adjacent subdomains, namely, weak continuity of the pressure and
the normal component of Darcy's velocity. We refer to them as $M_H$ and $V_H$ for the 
imposition of weak continuity of normal component of the flux and pressure, respectively.
In this work we take $M_H$ and $V_H$ to be either piecewise constant or piecewise linear polynomials.
Another important space in our formulation is the space of Lagrange multipliers $\Lambda^i_{H} \subset F_h(\mathcal{E}_h\cap\Gamma_i)$, which is build based on $\hat\Omega^i$, i.e., an oversampling region of $\Omega_i$.  
In order to build the Lagrange multiplier spaces $\Lambda^i_H$, consider the following set
of $N$ Darcy problems posed on $\hat\Omega_i$ given by  

\begin{equation}
  \begin{array}{rclll}
    {\bf u}_h^{k} &=& -K\nabla_h\,p_h^{k} &&\mbox{in} \ \hat\Omega_i \\
    \nabla_h\cdot{\bf u}_h^{k} &=& 0 &&\mbox{in}  \ \hat\Omega_i \\
    p_h^{k} & = & 0 \qquad &&\mbox{on}~\partial\hat\Omega_i\cap\partial\Omega_p \\
    {\bf u}_h^{k}\cdot {\bf n}^i & = & 0 \qquad &&\mbox{on}~\partial\hat\Omega_i\cap\partial\Omega_u \\
    -\beta_i\,{\bf u}_h^{k}\cdot {\bf n}^i + p_h^{k} &=& \lambda^k &&\mbox{on}~\partial\hat\Omega_i\setminus\partial\Omega
  \end{array}, \label{eq:oversampling_problem}
\end{equation}
in which $\nabla_h\cdot$ and $\nabla_h$ denotes the discrete divergence and gradient operators, respectively,
and $\lambda^k$ is a piecewise polynomial functions defined on $F_h(\partial\hat\Omega_i\setminus\partial\Omega)$ for $k=1,\dots,N$. 
The discrete spaces used to solve problems (\ref{eq:oversampling_problem}) are the lowest-order Raviart-Thomas
spaces for subdomain $\hat\Omega_i$ with mesh $\mathcal{T}_h^{\hat\Omega_i}$.
After solving the $N$ problems for each subdomain $\hat\Omega_i$, retrieve the normal component of the velocity ${\bf u}_h^{k}\cdot {\bf n}^i$ and the pressure $\pi^k$ on the interface $\Gamma_i$ and write the following functions

\begin{equation}
  \phi_i^k = -\beta_i\,{\bf u}_h^{k}\cdot {\bf \check{n}}^i|_{\Gamma_i} + \pi^k|_{\Gamma_i}, \ \forall\,k=1,\dots,N,  
  \label{informed_space}  
\end{equation}
in which notation $|_{\Gamma_i}$ is used to reinforce the restriction of the solution to the interface $\Gamma_i$. Finally we define $\Lambda^i_H = \mbox{span}\left\{\phi_i^1,\phi_i^2,..,\phi_i^N\right\}$ for each $\Omega_i$. Having defined all necessary spaces, we are now prepared to introduce the proposed method.
%\begin{rmk}
%  Notice that $\beta_i$ is a function of space, then $\beta_i$ of Eq. (\ref{informed_space}) is defined over interface $\Gamma_i$. On the other hand, $\beta_i$ of Eq. (\ref{eq:local_interface_problem}) is defined over $\partial\hat\Omega_i\setminus\partial\Omega$.`
%\end{rmk}

The discrete variational formulation of the Multiscale Robin Coupled Method with Oversampling (MRCM-O) reads as: 
Find $({\bf u}_h^i,p_h^i,\lambda_h^i)\,\in\,{\bf V}_{hz}^i\times Q_h^i
\times \Lambda^i_{H}$, for
$i=1,\ldots,m$, such that
\begin{eqnarray}
(K^{-1}{\bf u}_h^i,{\bf v})_{\Omega_i}-(p_h^i,\nabla\cdot {\bf v})_{\Omega_i} 
+(\beta_i\,{\bf u}_h^i\cdot\check{\bf n}^i,{\bf v}\cdot\check{\bf n}^i)_{\Gamma_i}
+(\lambda_h^i,{\bf v}\cdot\check{\bf n}^i)_{\Gamma_i} \nonumber \\ = -(g,{\bf v}\cdot\check{\bf n}^i)_{\partial\Omega_i\cap\partial\Omega_p},  \label{eq11d}\\
(q,\nabla\cdot {\bf u}_h^i)_{\Omega_i}  =   (f,q)_{\Omega_i}, \label{eq12d}\\
\sum_{i=1}^m ({\bf u}_h^i\cdot \check{\bf n}^i,M)_{\Gamma_i} = 0,  \label{eq13d}\\
\sum_{i=1}^m (\beta_i\,{\bf u}_h^i\cdot \check{\bf n}^i+\lambda_h^i,V \,\check{\bf n}^i\cdot\check{\bf n})_{\Gamma_i} = 0,  \label{eq14d}
\end{eqnarray}
hold for all $({\bf v},q)\,\in\,{\bf V}_{h0}^i$ and for all $(M,V)\,\in\, M_H \times V_H \subset F_h(\mathcal{E}_h)\times\,F_h(\mathcal{E}_h)$.
In the above equations $\beta_i > 0$ are the Robin condition parameters. We remark that
in line with \cite{guiraldello2018multiscale} when specifying Robin boundary conditions 
\begin{equation}
 -\beta_i \, {\bf u}^i \cdot \check{\bf n}^i + p^i = g_R \label{eq:genRobin}
\end{equation}
where $g_R$ is a prescribed value, we write
\begin{equation}
 \beta_i\left({\bf x}\right) = \dfrac{\alpha H}{K_{H}\left({\bf x}\right)}
 \end{equation}
where $K_{H}$ refers to the harmonic average of adjacent $K$ values and $\alpha$ is a dimensionless parameter that determines the relative importance of the normal component of the flux and the pressure. The importance of the parameter $\alpha$ has been discussed in Section \ref{intro}. The MRCM-O formulation should be compared with the Two-Lagrange-Multiplier formulation introduced in \cite{guiraldello2018multiscale}. 
Notice that the functions of space $\Lambda^i_{H}$ are the Robin boundary conditions to be imposed to $\Omega_i$, such that
its solution is the restriction of solutions (multiscale basis functions) computed by solving \eqref{eq:oversampling_problem}. This integration enables us to seamlessly incorporate the oversampling strategy into our non-overlapping, well-defined methodology, as outlined below.

\begin{prop}
Given spaces $M_H\times\,V_H \subset F_h(\mathcal{E}_h)\times\,F_h(\mathcal{E}_h)$, the solution $({\bf u}_h^i,p_h^i,\lambda_h^i)$ in  $\Pi_{i=1}^m{\bf V}_{hz}^i\times Q_h^i
\times \Lambda^i_{H}$ to the discrete formulation (\ref{eq11d})\,-\,(\ref{eq14d}) is unique  if the following restrictions holds:
\begin{enumerate}
  \item $\mbox{dim}\left(M_H\right) = \mbox{dim}\left(V_H\right)$,
  \item $\sum_{i=1}^{m} \mbox{dim}\left(\Lambda^i_H\right) = \mbox{dim}\left(M_H\right) + \mbox{dim}\left(V_H\right)$. 
\end{enumerate}
\end{prop}

\begin{proof}
Assuming the above restrictions holds will lead (\ref{eq11d})\,-\,(\ref{eq14d}) into a square linear system. Set $g=z=f=0$ and 
$\lambda_h^i = 0$ for $i=1,\dots,m$., and take ${\bf v}={\bf u}_h^i$, $q=p_h^i$, 
$M=\tilde{M}:=\Pi_{0,M_H}({\bf u}_h^i\cdot\,{\bf \check{n}}^i+{\bf u}_h^j\cdot\,{\bf \check{n}}^j)|_{\Gamma_{ij}},\ \forall\,i<j$
and 
$V=\tilde{V}:=\Pi_{0,V_H}(\beta^i{\bf u}_h^i\cdot\,{\bf \check{n}}^i - \beta^j{\bf u}_h^j\cdot\,{\bf \check{n}}^j)|_{\Gamma_{ij}},\ \forall\,i<j$, in which $\Pi_{0,M_H}$ and $\Pi_{0,V_H}$ are the L$^2$-projection operators
onto $M_H$ and $V_H$, respectively. Adding over all subdomains one gets
\begin{eqnarray}
\sum_{i=1}^m \left [ (K^{-1}{\bf u}_h^i,{\bf u}_h^i)_{\Omega_i} + \|\sqrt{\beta_i}\,{\bf u}_h^i\cdot\check{\bf n}^i\|^2_{\Gamma_i}\right ] &=& 0, \label{eq1_proof}\\
\sum_{i<j} (\tilde{M},\tilde{M})_{\Gamma_{ij}} &=& 0,  \label{eq2_proof}\\
\sum_{i<j}(\tilde{V},\tilde{V})_{\Gamma_{ij}} &=& 0.  \label{eq3_proof}
\end{eqnarray}
The derivation involves rewriting equations \eqref{eq13d}-\eqref{eq14d} as a sum of jumps over interfaces $\Gamma_{ij}$ for all $i<j$, and incorporating the orthogonality of the projections, resulting in equations \eqref{eq2_proof}-\eqref{eq3_proof}. Adding up all the resulting equations one concludes that ${\bf u}_h^i = 0$ for $i=1,\dots,m$.
Thus, we have $\sum_{i=1}^m (p_h^i,\nabla\cdot {\bf v})_{\Omega_i} = 0,\,\forall\,{\bf v} \in {\bf V}_{h0}^i$. 
From the stability of the local Raviart-Thomas spaces, we conclude that $p_h^i = 0$ for $i=1,\dots,m$. 
 
\end{proof}

\begin{rmk}
Notice that $\tilde{M}$ is the projection of velocity jumps on $\Gamma$ onto $M_H$ 
and $\tilde{V}$ is the projection of pressure jumps (with $\lambda_h^i = 0$) on $\Gamma$ onto $V_H$.
\end{rmk}

Indeed, formulation \eqref{eq11d}-\eqref{eq14d} can be characterized as a domain decomposition method, as outlined in the subsequent proposition.

\begin{prop}
Let $({\bf u}_h^i,p_h^i,\lambda_h^i)$ be the solution in  $\Pi_{i=1}^m{\bf V}_{hz}^i\times Q_h^i
\times \Lambda^i_{H}$ of the discrete formulation (\ref{eq11d})--(\ref{eq14d}) with $M_H\times\,V_H = F_h(\mathcal{E}_h)\times\,F_h(\mathcal{E}_h)$ and $\Lambda^i_H = \mbox{span}\left\{\phi_i^1,\phi_i^2,..,\phi_i^N\right\} = F_h(\mathcal{E}_h\cap\Gamma_i)$, for $i=1,\dots,m$, and let 
$({\bf u}_h,p_h)\,\in\,({\bf V}_{hz}\cap H(\mbox{div},\Omega))\times Q_h$ be the solution of the non-decomposed
discrete problem which satisfies
\begin{eqnarray}
  (K^{-1}{\bf u}_h,{\bf v})_{\Omega}-(p_h,\nabla\cdot {\bf v})_{\Omega} 
  &=&-(g,{\bf v}\cdot\check{\bf n})_{\partial\Omega_p}~, \label{eqglo1}\\
  (q,\nabla\cdot {\bf u}_h)_{\Omega}&=& (f,q)_{\Omega}~, \label{eqglo2}
\end{eqnarray}
for all ${\bf v}\,\in\,{\bf V}_{h0}\cap H(\mbox{div},\Omega)$ and all
$q\,\in\,Q_h$.

Then, assuming $\beta_i$ to be
constant on each edge of $\mathcal{E}_h\cap\partial\Omega_i$,
\begin{eqnarray}
  {\bf u}_h^i&=&{\bf u}_h|_{\Omega_i},\label{eqprop1a}\\
  p_h^i&=&p_h|_{\Omega_i},\label{eqprop1b}
\end{eqnarray}
for each $i=1,\dots,m.$
\label{prop2}
\end{prop}

\begin{proof}
The proof is given in \cite{guiraldello2018multiscale} (see {\bf Proposition 2.} together {\bf Remark 2.}) 
with the fact that $\Lambda^i_H = \mbox{span}\left\{\phi_i^1,\phi_i^2,..,\phi_i^N\right\} = F_h(\mathcal{E}_h\cap\Gamma_i)$.   
\end{proof}

\begin{rmk}
For \textbf{Proposition 2} to hold, the number $N$ of basis functions should match the count of edges/faces on $\mathcal{E}_h \cap \Gamma_i$, for $i=1,\dots,m$. 
\end{rmk}

\begin{figure}[H]
    \centering
    \includegraphics[width = 0.4\textwidth]{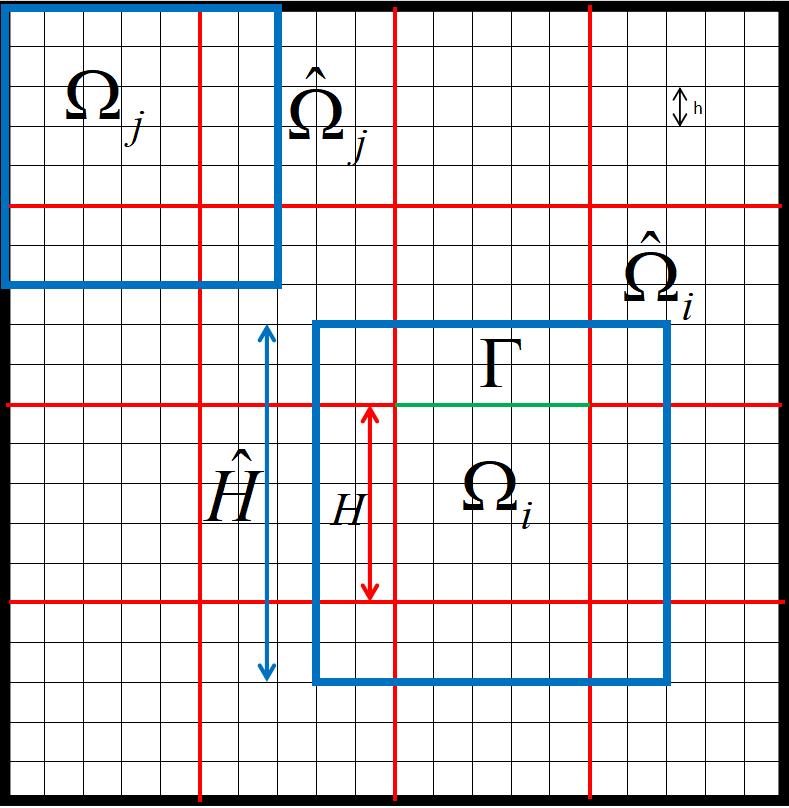}
    \caption{Decompositions of the computational domain $\Omega$: The non-overlapping subdomains are denoted by $\Omega_i$, and the corresponding oversampling regions are written as $\hat{\Omega}_i$. The three length scales that enter in the formulation of the MRCM with oversampling are also illustrated:  $\hat H > H \geq h$.}
    \label{overlapjpg}
\end{figure}

\subsection{Multiscale basis functions (MBFs)}
%%%%%%%%%%%%%%%%%%%%%%%%%%%%%%%%%%%%%%%%%
To describe an efficient method to solve the system (\ref{eq11d})--(\ref{eq14d}) possibly taking advantage of multiple cores we will introduce the notion of multiscale basis functions.
Assume at this stage that the Robin boundary conditions $\lambda_h^i\,\in\,{\Lambda^i_{\hat H}}$, for $i=1,\ldots,m$, are known.
This way we can identify a family of local boundary value problems on $\Omega_i$, $i=1,\ldots,m$, given by:
Find $({\bf u}_h^i,p_h^i)\,\in\,({\bf V}_{hz}^i\times Q_h^i)$, such that
\begin{eqnarray}
(K^{-1}{\bf u}_h^i,{\bf v})_{\Omega_i}-(p_h^i,\nabla\cdot {\bf v})_{\Omega_i} 
+(\beta_i\,{\bf u}_h^i\cdot\check{\bf n}^i,{\bf v}\cdot\check{\bf n}^i)_{\Gamma_i} \nonumber \\ = 
- (\lambda_h^i,{\bf v}\cdot\check{\bf n}^i)_{\Gamma_i}
-(g,{\bf v}\cdot\check{\bf n}^i)_{\partial\Omega_i\cap\partial\Omega_p},  \label{eq111d}\\
(q,\nabla\cdot {\bf u}_h^i)_{\Omega_i}  =   (f,q)_{\Omega_i}, \label{eq122d}
\end{eqnarray}
hold for all $({\bf v},q)\,\in\,{\bf V}_{h0}^i$.

%KEEP%
Next we split the solution $({\bf u}_h^i,p_h^i)$ as
\begin{eqnarray}
 {\bf u}^i_h & = & \tilde{\bf u}^i_h + \bar{\bf u}^i_h, \label{eq:addu}\\
 p^i_h & = & \tilde{p}^i_h + \bar{p}^i_h, \label{eq:addp}
\end{eqnarray}
such that 
$(\bar{\bf u}_h,\bar{p}_h) \in {\bf V}_{hz}^i\times Q_h^i$ satisfies the subdomain problems with 
trivial Robin boundary conditions and nonzero forcing terms, that is
\begin{eqnarray}
(K^{-1}\bar{\bf u}_h^i,{\bf v})_{\Omega_i}-(\bar{p}_h^i,\nabla\cdot {\bf v})_{\Omega_i} 
+(\beta_i\,\bar{\bf u}_h^i\cdot\check{\bf n}^i,{\bf v}\cdot\check{\bf n}^i)_{\Gamma_i} \nonumber \\ = 
-(g,{\bf v}\cdot\check{\bf n}^i)_{\partial\Omega_i\cap\partial\Omega_p},  \label{eq111dbar}\\
(q,\nabla\cdot \bar {\bf u}_h^i)_{\Omega_i}  =   (f,q)_{\Omega_i}, \label{eq122dbar}
\end{eqnarray}
hold for all $({\bf v},q)\,\in\,{\bf V}_{h0}^i$.
Moreover,
$(\tilde{\bf u}^i_h,\tilde{p}^i_h) \in {\bf V}_{h0}^i\times Q_h^i$ satisfies local problems
with forcing terms $f$ and $g$ set to zero
\begin{eqnarray}
(K^{-1}\tilde{\bf u}_h^i,{\bf v})_{\Omega_i}-(\tilde{p}_h^i,\nabla\cdot {\bf v})_{\Omega_i} 
+(\beta_i\,\tilde {\bf u}_h^i\cdot\check{\bf n}^i,{\bf v}\cdot\check{\bf n}^i)_{\Gamma_i} \nonumber \\ = 
- (\lambda_h^i,{\bf v}\cdot\check{\bf n}^i)_{\Gamma_i},  \label{eq111dtilde}\\
(q,\nabla\cdot \tilde{\bf u}_h^i)_{\Omega_i}  =   0, \label{eq122dtilde}
\end{eqnarray}
hold for all $({\bf v},q)\,\in\,{\bf V}_{h0}^i$.

%HERE% 
The local boundary value problems above, like Eqs. (\ref{eq111dbar})--(\ref{eq122dbar}) 
or Eqs. (\ref{eq111dtilde})--(\ref{eq122dtilde}),
can be solved using standard discrete spaces for
${\bf V}_{h0}^i$ and $Q_h^i$.
In this work, we consider two-dimensional problems
along with the lowest order Raviart-Thomas space $\mbox{RT}_0$ on 
quadrilateral cartesian grids of uniform cell size $h$ (see Fig. \ref{overlapjpg}).
The choice of this space and a combination of the midpoint and the 
trapezoidal rules for numerical integration produce a
discrete linear system which only involves pressure unknowns and
that is equivalent to a cell--centered finite difference method
\cite{douglas1997numerical}.
%HERE%

Considering Eqs. (\ref{eq13d})-(\ref{eq14d}) that impose weak
flux and pressure continuity on $\Gamma$ on the coarse scale $H$, in principle
we may simultaneously solve for all unknowns and subdomains in the system
(\ref{eq11d})-(\ref{eq14d}). Our goal next is to eliminate the internal subdomain degrees of freedom,
thus solving a relatively small linear system associated with the interface $\Gamma$. 
Multiscale basis functions that are constructed independently on
subdomains play an essential role in such reduction of the problem size.

First of all let us denote by $\{ \phi^j \}_{1 \le j \le n_i}$ a basis for the coarse space  $\Lambda^i_{\hat H}$, $i=1,\ldots,m$. 
For instance, in the case of piecewise constant space $\Lambda^i_{\hat H}$, for subdomains not touching the exterior boundary $n = 4$
(see Fig \ref{overlapjpg}). 
%%%%%%%$M^i_H$, $V^i_H$, that, for simplicity, we take to be the same. Also, let us denote 
The multiscale basis functions are defined to be solutions to local problems in $\hat\Omega_i$ as follows. For
each $\Lambda^i_{\hat H}$, $i=1,\ldots,m$ let $(\boldsymbol{\Phi}^i_{j}, \Psi^i_{j})$ be the solution to Eqs.
(\ref{eq111dtilde})-(\ref{eq122dtilde}) with Robin boundary data set to be $\phi^j$, $j=1,\ldots,n_i$. 
Next, express the local solution on each subdomain as:
\begin{equation}
\tilde{\bf u}_h^i = \sum_{j=1}^{n_i}{X_{j}^i \boldsymbol{\Phi}^i_{j}},~~
\tilde{p}_h^i = \sum_{j=1}^{n_i}{X_{j}^i \Psi^i_{j}}.  \label{eq:uhatphat}
\end{equation}
%%%%
Taking into account Eqs. (\ref{eq:addu})-(\ref{eq:addp})
along with Eqs.  (\ref{eq13d})-(\ref{eq14d}) on the $H$ scale, we have
\begin{eqnarray}
\sum_{i=1}^m{\left ( \tilde{\bf u}_h^i \cdot \check{\bf n}^i , M \right )}_{\Gamma_i} & = &
- \sum_{i=1}^m{\left ( \bar{\bf u}_h^i \cdot \check{\bf n}^i , M \right )}_{\Gamma_i}\,,  \label{eq:int_ujump} \\
 \sum_{i=1}^m (\beta_i\,\tilde{\bf u}_h^i\cdot \check{\bf n}^i +\lambda_h^i\,,
 V \check{\bf n}^i\cdot\check{\bf n})_{\Gamma_i} 
  & = & -\sum_{i=1}^m (\beta_i\,\bar{\bf u}_h^i\cdot \check{\bf n}^i,
 V \check{\bf n}^i\cdot\check{\bf n})_{\Gamma_i}.\label{eq:int_pjump}
\end{eqnarray}
Next, substitute Eq. (\ref{eq:uhatphat}) on the above equations and test with all basis
functions $(M,V)\,\in\, M_H \times V_H$
to construct
a global linear system for the unknowns $X_{j}^i$, $i=1,\ldots,m$, and $j=1,\ldots n_1.$ 
For additional details concerning the construction of this linear system, we refer the reader
to \cite{guiraldello2018multiscale}.

\begin{rmk}
In practice, solutions to Eqs. (\ref{eq111dtilde})-(\ref{eq122dtilde}) with $\phi^j$, $j=1,\ldots,n_i$, have already been computed during the construction of functions $\phi^j$ for the informed Lagrange space in Eq. (\ref{informed_space}), thereby leaving only Eqs. (\ref{eq111dbar})-(\ref{eq122dbar}) to be solved.   
\end{rmk}

%%%%HERE%%%%%

\section{The Smoothing Steps}

We will define the smoothing steps through local updates of a solution for Eqs. (\ref{eq1a} - \ref{eq1d}) that has been obtained by the MRCM with oversampling. Let us refer to it as $({\bf u}_h^{i,0},p_h^{i,0})$, $i=1,\ldots,m$, where the superscript $0$ indicates that this solution is the first one of a sequence that will follow.
The steps that produces the solution $({\bf u}_h^{i,k},p_h^{i,k})$, $i=1,\ldots,m$, $k=1,\ldots,N_s$ can be described as follows:
\begin{enumerate}
\item Set a coloring scheme for oversampling regions within the subdomain partitioning process. The aim is to assign a distinct color to each subdomain to ensure that subdomains sharing the same color do not have any common boundary points. This approach is illustrated in Fig.(\ref{colored}) for a two-dimensional partition (see \cite{DomainD}).
%%\item Define a coloring scheme for the oversampling regions such that for each subdomain, we associate a color so that the subdomains with the same color don't share any common boundary point(this is illustrated in Fig.(\ref{colored}) for a two-dimensional partition) and loop over all colors.
\item For each color, loop of all corresponding oversampling regions.
\item For all oversampling regions with the same color set the Robin boundary conditions as $\lambda_h^{i, k-1} = -\beta_i\,{\bf u}_h^{i,k-1}\cdot {\bf \check{n}}^i|_{\partial\hat\Omega_i\setminus\partial\Omega} + \pi^{i,k-1}|_{\partial\hat\Omega_i\setminus\partial\Omega}$ on $\partial\hat\Omega_i\setminus\partial\Omega$ (see \cite{douglas}) in terms of 
existing $({\bf u}_h^{i,k-1},p_h^{i,k-1})$, $i=1,\ldots,m$ with $\alpha = 1$ (other values for $\alpha$ can be used but our numerical studies indicate that this is a good choice).
\item For all oversampling regions with the same color simultaneously find $(\hat {\bf u}_h^{i},\hat p_h^{i})$ by solving
\begin{eqnarray}
(K^{-1}\hat {\bf u}_h^{i},{\bf v})_{\hat\Omega_i}-(\hat p_h^i,\nabla\cdot {\bf v})_{\hat\Omega_i} 
+(\beta_i\,\hat {\bf u}_h^i\cdot\check{\bf n}^i,{\bf v}\cdot\check{\bf n}^i)_{\partial\hat\Omega_i\setminus\partial\Omega} \nonumber \\ = 
- (\lambda_h^{i,k-1},{\bf v}\cdot\check{\bf n}^i)_{\partial\hat\Omega_i\setminus\partial\Omega}
-(g,{\bf v}\cdot\check{\bf n}^i)_{\partial\hat\Omega_i\cap\partial\Omega_p},  \label{eq:sm_1}\\
(q,\nabla\cdot \hat {\bf u}_h^i)_{\hat\Omega_i}  =   (f,q)_{\hat\Omega_i}, \label{eq:sm_2}
\end{eqnarray}
hold for all $({\bf v},q)\,\in\,{\bf V}_{h0}^i$.
\item For all oversampling regions with the same color update $({\bf u}_h^{i,k-1},p_h^{i,k-1})$ in $\Omega_i$ to be the restriction of 
$(\hat {\bf u}_h^{i},\hat p_h^{i})$ computed above to $\Omega_i$. 
\item After the loop over all colors is completed set $({\bf u}_h^{i,k},p_h^{i,k})$ to be the existing $({\bf u}_h^{i,k-1},p_h^{i,k-1})$, $i=1,\ldots,m$.
\item Repeat the above steps $N_s$ times.
\end{enumerate}

%We refer to the new method introduced in this work, consisting of MRCM-O followed by smoothing steps, as the Multiscale Robin Coupled Method with Oversampling and Smoothing (MRCM-OS).

The novel method introduced in this work, comprising MRCM-O followed by smoothing steps, is referred to as the Multiscale Robin Coupled Method with Oversampling and Smoothing (MRCM-OS).

\begin{figure}[H]
	%\begin{minipage}{0.49\textwidth}
	    %\centering
	   % \includegraphics[width = 1\textwidth]{color.jpg}
	    %%\caption{Domain decomposition: Non-overlapping partition $\Omega_i$ and oversampling domain $\hat{\Omega}_i$. }
   		%%\label{overlapjpg}
	%\end{minipage}
	\begin{minipage}{1.0\textwidth}
		\centering
		\includegraphics[width = 0.243\textwidth]{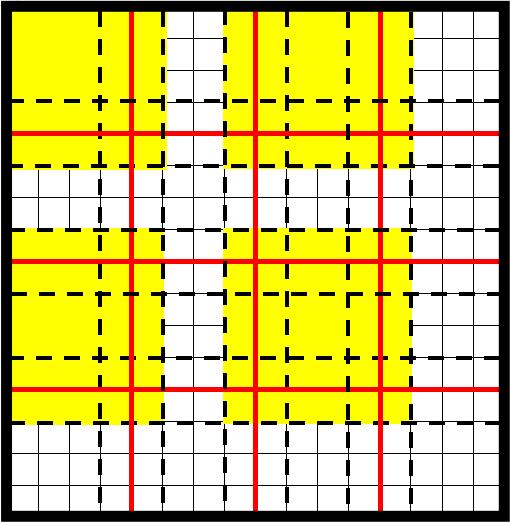}
		\includegraphics[width = 0.243\textwidth]{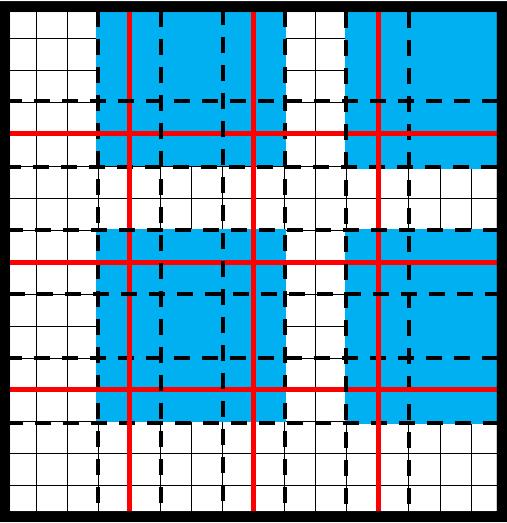}
		\includegraphics[width = 0.243\textwidth]{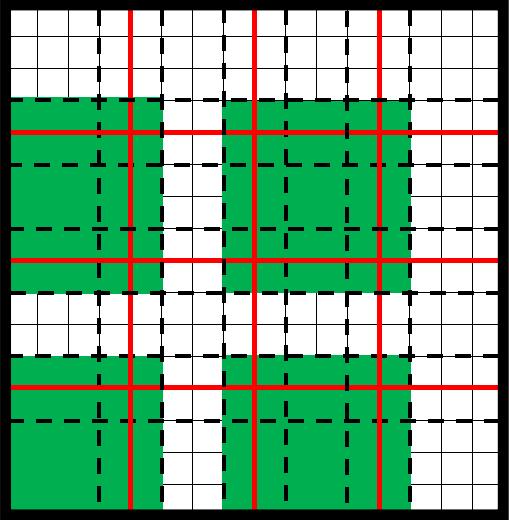}
		\includegraphics[width = 0.243\textwidth]{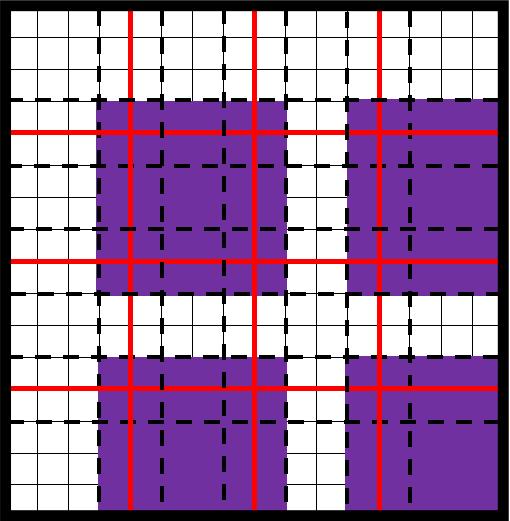}
	\end{minipage}
	\caption {Partitions: non-overlapping (solid line) and oversampling (dotted line). In the coloring scheme oversampling subdomains  sharing the same color do not have any common boundary points.}
	\label{colored}
\end{figure}

\section{Computational Strategy}
This section provides the numerical strategy employed within our in-house parallel code developed in C/MPI. Algorithm \eqref{alg:numerical_strategy} provides a pseudo-code describing the adopted strategy.

\begin{algorithm}
  \caption{Algorithmic representation of the adopted numerical strategy.}
  \label{alg:numerical_strategy}
  \begin{algorithmic}[1]
    \State \textbf{Input: }$\text{Domain\_decomposition}, \text{oversampling\_size}$
    \State \textbf{Step 1} Compute and store LDL$^T$-factorization for problem \eqref{eq:oversampling_problem}
    \State \textbf{Step 2} Solve N problems of \eqref{eq:oversampling_problem} using Step 1's factorization
    \State \textbf{Step 3} Assemble lhs of eqs. \eqref{eq:int_ujump}-\eqref{eq:int_pjump} using Step 2's solutions
    \State \textbf{Step 4} Solve problems \eqref{eq111dbar}-\eqref{eq122dbar} and assemble rhs of eqs. \eqref{eq:int_ujump}-\eqref{eq:int_pjump}
    \State \textbf{Step 5} Solve interface problem and assemble MRCM-OS solution
    \State \textbf{Step 6} Compute and store LDL$^T$-factorization for problem \eqref{eq:sm_1}-\eqref{eq:sm_2}
    with $\alpha=1$ for the Robin parameter $\beta_i$
    \State \textbf{Step 7} Compute smoothing using Step 6's factorization
  \end{algorithmic}
\end{algorithm}

Efficiency within our strategy is primarily achieved through the reuse of LDL$^T$ factorizations. Here, we compute the LDL$^T$-factorization to solve linear problems of the form $Ax=b$ posed on subdomain $\hat\Omega_i$. It is worth noting that while solving local problems within each subdomain, the matrix $A$ remains constant, with variations occurring solely in vector $b$. As a result, we can compute the LDL$^T$-factorization of $A$ once and reuse it across all $Ax=b$ calculations.
This approach dramatically reduces the computational cost associated with computing the MBFs in \textbf{Step 2} taking advantage of \textbf{Step 1} and the subsequent smoothing step in \textbf{Step 7} taking advantage of \textbf{Step 6}.
As previously mentioned, our numerical experiments (not explored in this manuscript) indicate that the choice of $\alpha = 1$ for the smoothing
leads to improved accuracy with a reduced number of iterations. Additionally, the subsequent numerical investigations in the following section indicate that the smallest errors occur when $\alpha$ is approximately 1. 
In practice, choosing $\alpha \sim 1$ for MBFs and smoothing can eliminate the need to compute \textbf{Step 6}, allowing the use of the factorization obtained in \textbf{Step 1} to compute \textbf{Step 7}. This approach enhances accuracy and reduces computational cost.
 It is noteworthy that the majority of steps are computed concurrently within each MPI-process, underscoring the parallel nature of our approach. The sole exception arises during the computation of the solution to the interface problem in \textbf{Step 5}.

\section{Numerical Studies}

The simulations discussed were executed in an HPC cluster. We discuss the numerical solutions of two model problems: one having a constant permeability coefficient $K\left(\bf x\right)$  with an analytical solution and the other one encompassing a very heterogeneous permeability field from the SPE10 project (http://www.spe.org/web/csp).

For the problem with constant $K\left(\bf x\right)$, we analyze the system of Eq. (\ref{eq1a})-(\ref{eq1d}) subject to trivial Newmann boundary conditions. 
The permeability field is set to $K\left(\bf x\right) = 1$. The source term $q$ is written as  $q=8\pi^2 cos(2\pi x)cos(2\pi y)$. The analytical solution for 
this problem is given by $p=cos(2\pi x)cos(2\pi y)$. We are going to discuss several partitions of the domain $[0,1]\times[0,1]$. 
In all these partitions, each subdomain will be discretized by a computational local grid of size $20 \times 20$.

For the problem with a heterogeneous permeability field, we consider a rectangular domain and prescribe Dirichlet-type boundary conditions by fixing the pressure to 1 (0) 
on the left (right) boundary. We set Neumann-type conditions to zero at the top and bottom boundaries. The source term is set as $q(x)=0$. Additionally, 
the permeability field $K\left(\bf x\right)$ is set to be the 40th (two-dimensional) layer of the SPE10 project, as depicted in Fig.(\ref{PERM}). 
This particular field poses a considerable challenge for numerical solvers due to its variability and the presence of a highly permeable channel.
For all reported results in heterogeneous problems, the domain is decomposed into $11 \times 3$ subdomains with $H=1/3$, each containing a local $20 \times 20$ grid.

We are going to refer to the two problems described above as homogeneous and heterogenous, respectively.
%We are going to refer to the two problems described above as homogeneous and heterogenous, respectively. We remark that the original MRCM is equipped with linear polynomial interface spaces in all studies reported here.
\begin{figure}[h!]
\vspace{1cm}
    \centering
    \includegraphics[width = 1.0\textwidth]{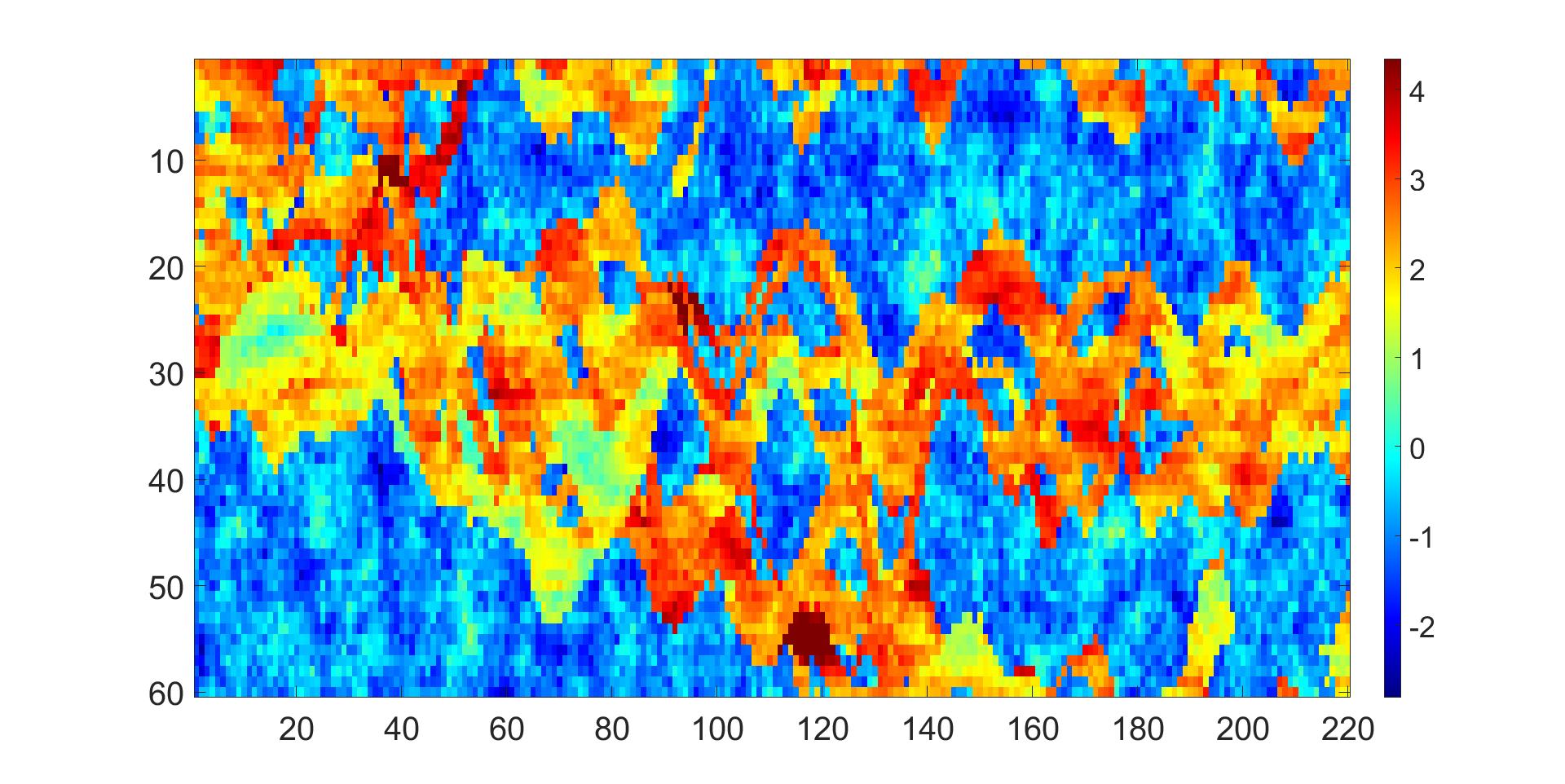}
    \caption{Permeability field: Slice 40 of the SPE 10 project.}
    \label{PERM}
\end{figure}
%Within the figures discussed below we adopt the following notation: $M \times M$ represents the count of subdomains in the non-overlapping partition. The term "MRCM" denotes the original Multiscale Robin Coupled Method (MRCM). Additionally, $OC$ ($OL$) signifies the utilization of piecewise constant (linear polynomial spaces) space for the interface and the
%setting of Robin-type boundary conditions in the construction of Multiscale Basis Functions. The notation $OL-l$ indicates a solution derived with an oversampling size of $lh$, whereas $OL-l,kS$ refers to a solution obtained with an oversampling size of $lh$, followed by $k$ smoothing steps.
Table \ref{tab:notation} summarizes the notation to be considered in the numerical experiments. We remark that the original MRCM is equipped with linear polynomial interface spaces in all studies reported here. 

\begin{table}[htbp]
  \centering
  \begin{tabular}{|l|l|}
    \hline
    \textbf{Notation} & \textbf{Description} \\
    \hline
    $M \times M$ & Count of subdomains in non-overlapping partition \\
    \textit{MRCM} & Original Multiscale Robin Coupled Method \\
    \textit{OC} & MRCM-OS with piecewise constant interface spaces for MBFs \\
    \textit{OL} & MRCM-OS with piecewise linear interface spaces for MBFs \\
    \textit{$O\_\,-l$} & Solution with oversampling size $lh$ \\
    \textit{$O\_\,-l,kS$} & Solution with oversampling size $lh$, followed by $k$ smoothing steps \\
    \hline
  \end{tabular}
  \caption{Notation used in numerical experiments.}
  \label{tab:notation}
\end{table}

Section 4.1 provides numerical results for both the heterogeneous and homogeneous problems, illustrating some of the advantageous aspects 
of the proposed MRCM-OS method. First we illustrate
our method's improved accuracy followed by a study showing that MRCM-OS employing only piecewise constant spaces is comparable to the original MRCM with linear polynomial spaces, although with lower computational cost. We emphasize that, in some results of the homogeneous problem, our method utilizing linear polynomial spaces demonstrates the capability to match or even surpass the accuracy of the fine grid solution. 
Section 4.2 explores the advantages brought by oversampling alone. In Section 4.3, we focus in a detailed analysis of the error reduction achieved by 
our method employing both oversampling and smoothing techniques.

\subsection{The role of oversampling and smoothing}

Our initial findings are presented in Fig. (\ref{hete-OL4-norm}), which pertains to the heterogeneous problem. The relative error is computed in the $L^2(\Omega)$ norm, comparing pressure and flux variables to the fine grid solution for $\alpha = 10^{-8}, 10^{-7}, ...,10^7, 10^8$. This study illustrates the improvements in accuracy  achieved through our proposed method. As illustrated in Fig. (\ref{hete-OL4-norm}), a remarkable enhancement by two orders of magnitude is observed in the flux error when comparing MRCM-OS with an oversampling size of $4h$ and 4 smoothing steps to MRCM. 
%As mentioned in the Introduction, the extra computational cost associated with the smoothing steps is minimal because we reuse factorizations previously computed in the construction of Multiscale Basis Functions.

\begin{figure}[H]
	\centering
	\begin{minipage}{0.49\textwidth}
		\includegraphics[width = 1.0\textwidth]{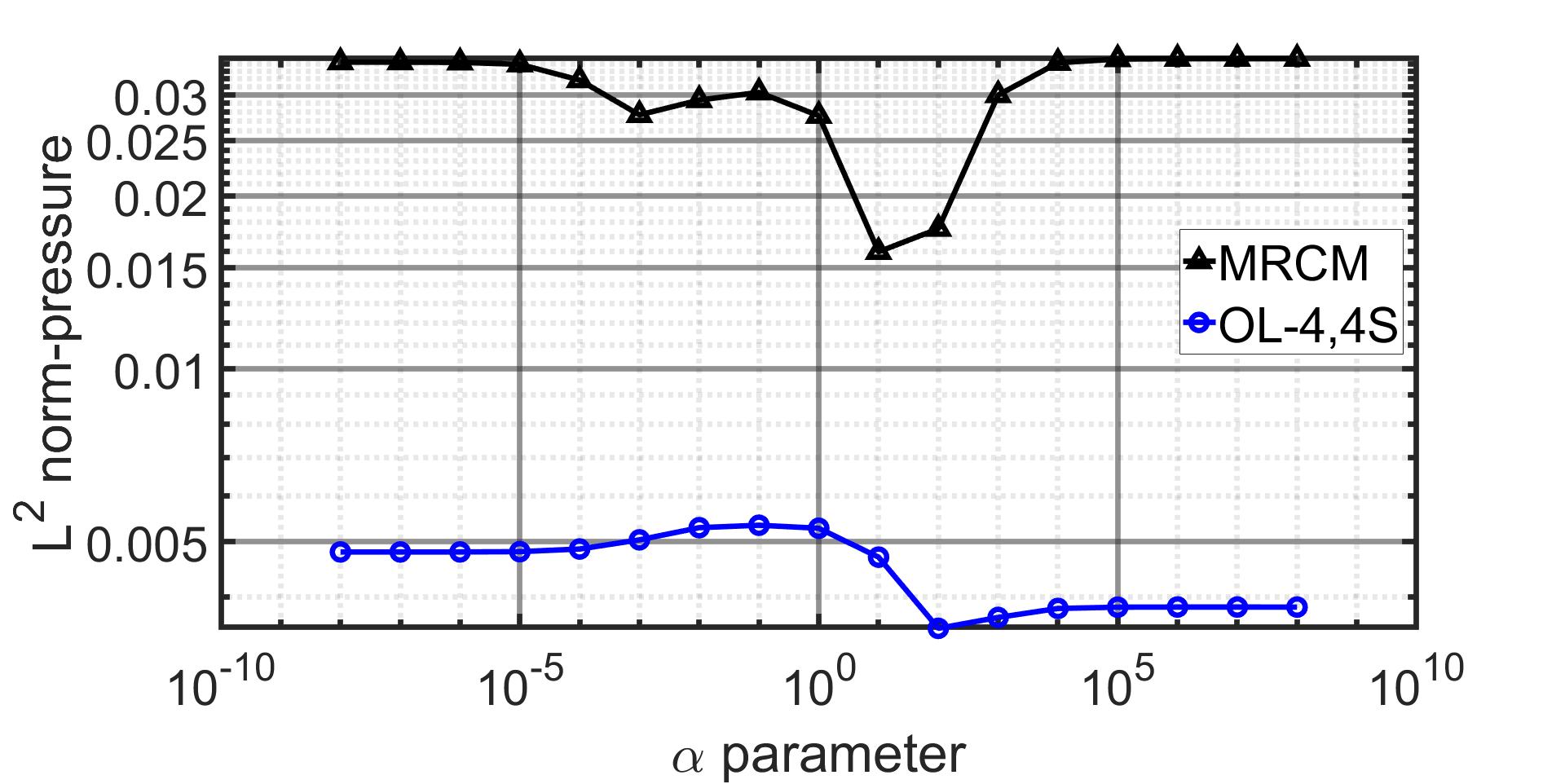}
	\end{minipage}
	\begin{minipage}{0.49\textwidth}
		\includegraphics[width = 1.0\textwidth]{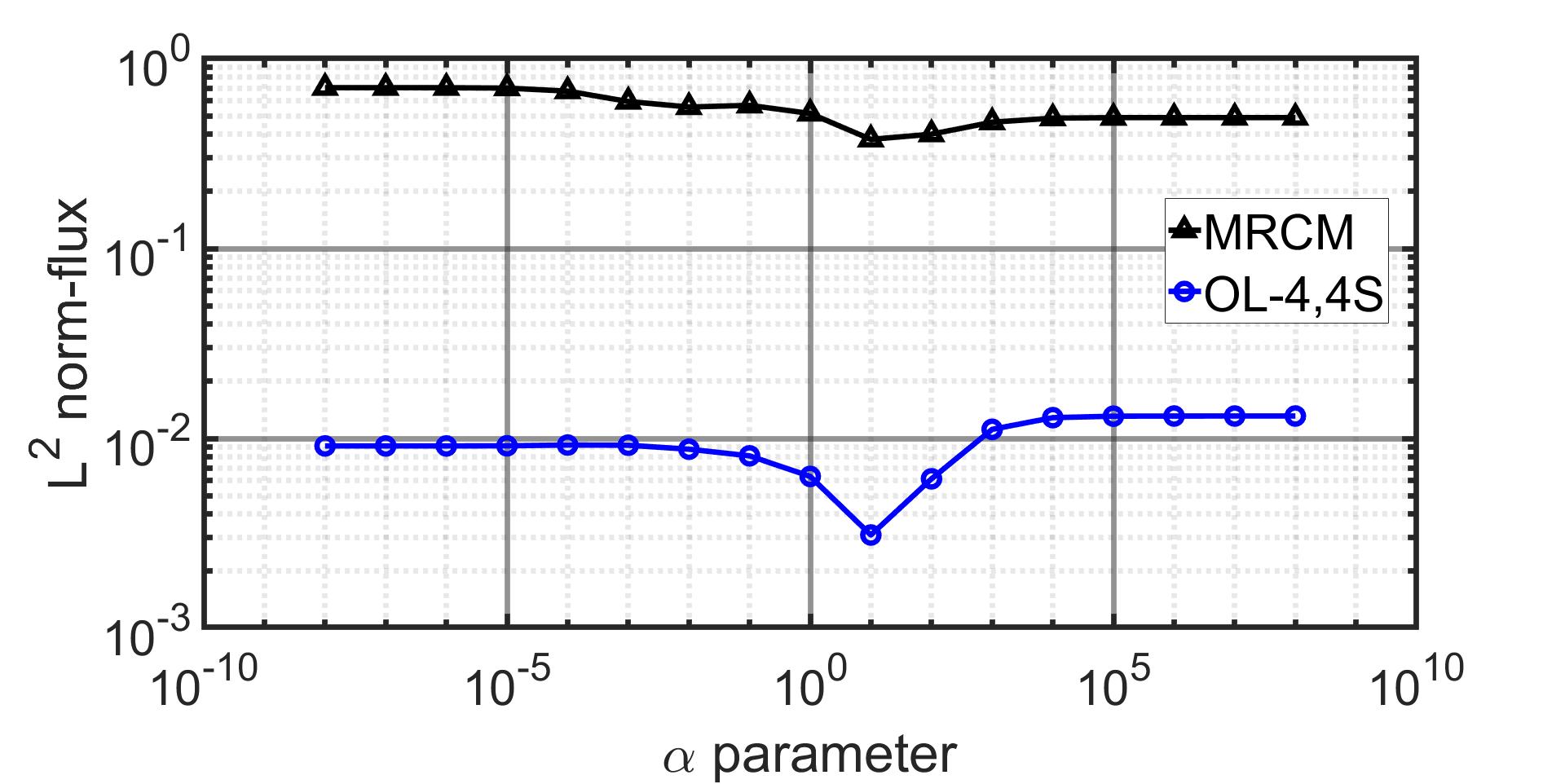}
	\end{minipage}
	\caption{$\alpha$ parameter study for the heterogeneous problem demonstrating significant improvement: Pressure relative error (left) and flux relative error (right).}
	\label{hete-OL4-norm}
\end{figure}

To better display the differences between the solutions generated by MRCM and MRCM-OS we provide additional results for the case $\alpha=1$.  
In Fig. (\ref{hete-OL4-pic}) we plot the fluxes
computed with both methods in the left column where colors refer to the magnitude of the velocity field.
The right column contains two results. The top images display the pressure fields, illustrating that our method yields a more continuous pressure field. The bottom images show the flux jump along the highlighted line in the pictures on the left column. Clearly, the jump is significantly reduced in the case of MRCM-OS, indicating its ability to mitigate flux discontinuity along the interface.

\begin{figure}[H]
	\centering
	\begin{minipage}{1\textwidth}

		\includegraphics[width = 1.0\textwidth]{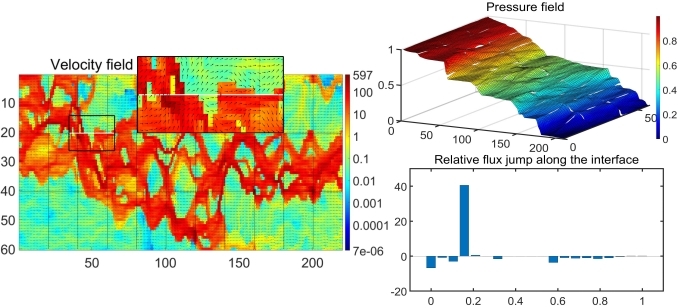}
	\end{minipage}
	
	\begin{minipage}{1\textwidth}\vspace{1cm}
		\includegraphics[width = 1.0\textwidth]{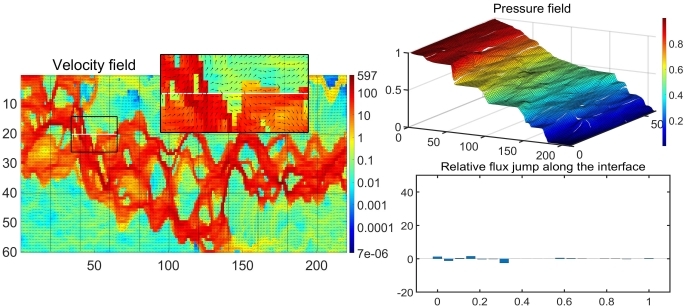}
	\end{minipage}

	\caption{Multiscale solution for the heterogeneous problem presented in colored images for the Robin condition parameter $\alpha=1$: Original MRCM method (top) and 
	MRCM-OS (bottom).}
	\label{hete-OL4-pic}
\end{figure}

Next, we focus on a comparison of MRCM-OS with piecewise constant spaces ($M_H,V_H$) and $\Lambda_{H}^i, i=1,...,m$ built with piecewise constant functions on the oversampling and with
the original MRCM with linear polynomial spaces for pressure and flux variables. 
We aim to illustrate that through the integration of oversampling and smoothing techniques, the approximation for the flux derived from piecewise constant spaces can not only match but potentially outperform those obtained by the original MRCM employing linear polynomial spaces.
In Fig. (\ref{het-OC-norm}) we present relative errors computed with respect to the fine grid solution in the $L^2(\Omega)$ norm considering both pressure and flux variables for different values of parameter $\alpha$. We remark that the computational cost of MRCM is, in fact, larger than that of MRCM-OS in computing the solutions: MRCM
requires twice the number of Multiscale Basis Functions (solutions of Eqs. (\ref{eq111dtilde}) - (\ref{eq122dtilde})) and the interface problem is two times larger that the one in MRCM-OS.

\begin{figure}[H]
	\centering
	\begin{minipage}{0.49\textwidth}
		\includegraphics[width = 1.0\textwidth]{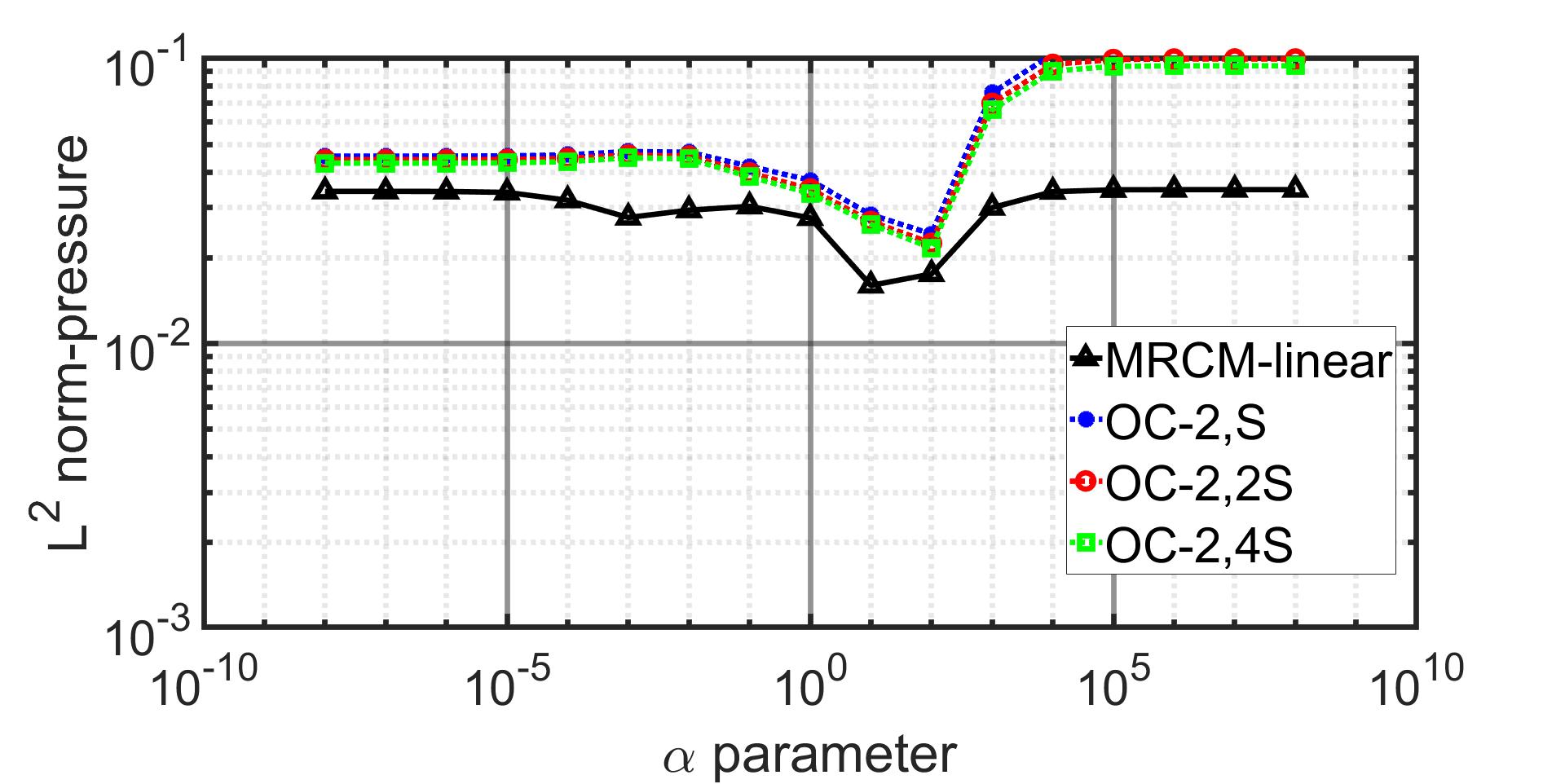}
	\end{minipage}
	\begin{minipage}{0.49\textwidth}
		\includegraphics[width = 1.0\textwidth]{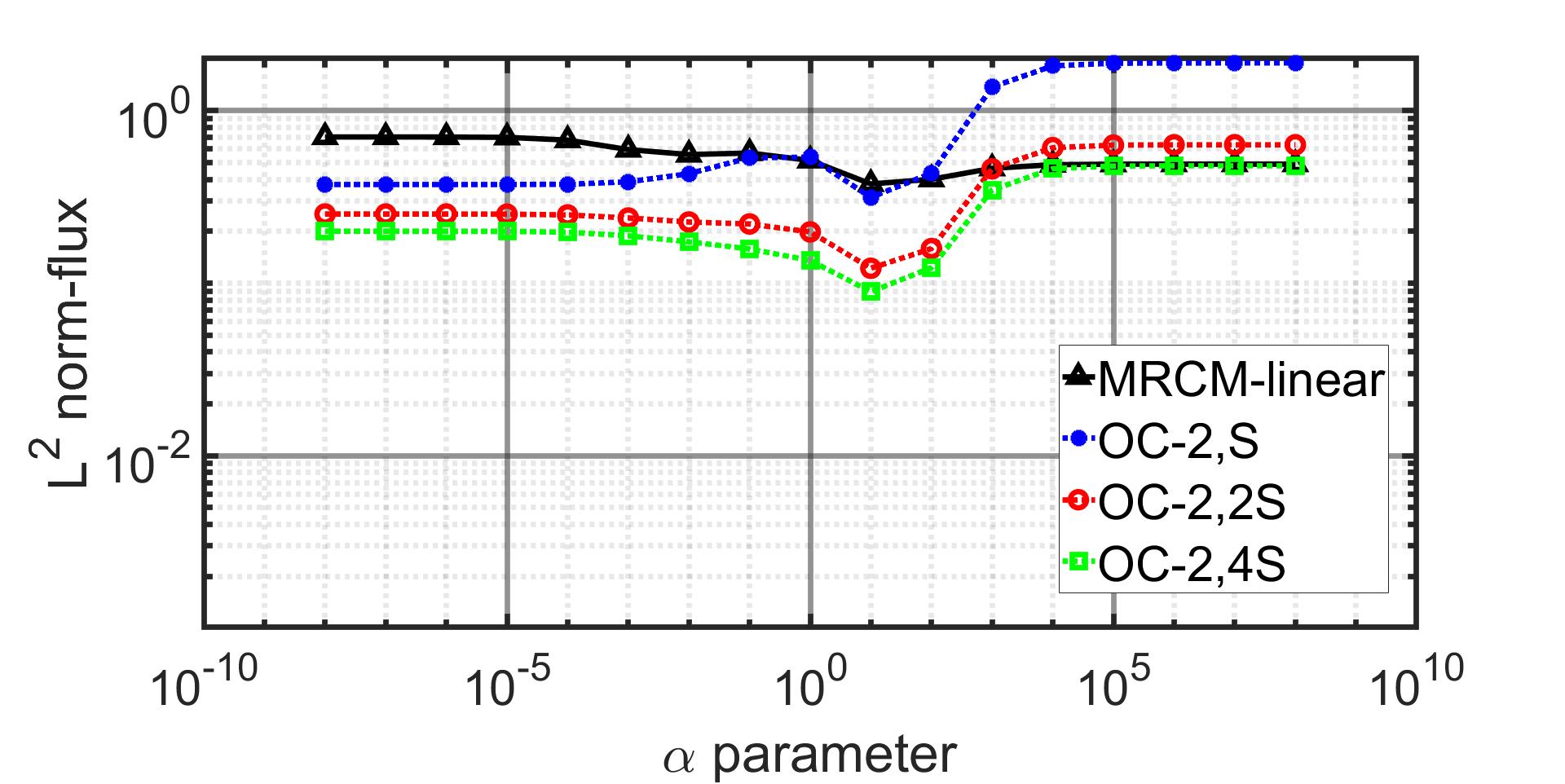}
	\end{minipage}
	\caption{Comparison of solutions of the heterogeneous problem obtained by MRCM with linear polynomial spaces and our method with piecewise constant spaces: Pressure relative error (left) and flux relative error (right).}
	\label{het-OC-norm}
\end{figure}

An important observation emerges from our analysis: solutions obtained with smaller values of $\alpha$ (resembling MMMFEM) exhibit slightly superior accuracy compared to those with larger $\alpha$ (akin to MHM). Interestingly, when $\alpha$ approximates 1, optimal outcomes are attained. Furthermore, from Fig. (\ref{het-OC-norm}), it becomes evident that within the proposed method, the pressure closely approximates the pressure obtained with the original MRCM. Moreover, the flux accuracy, achieved with a mere oversampling size of $2h$ and two smoothing steps, surpasses the results obtained from the original MRCM.

Next we report the $L^2(\Omega)$ norm (absolute error) of MRCM-OS results relative to the homogeneous problem. Fig. (\ref{ana-best}) 
contains the results of a mesh refinement study and illustrates that, in some cases, one may find that MRCM-OS exhibits comparable or even smaller errors than the fine grid solution in pressure (when $\alpha=100$) and flux (when $\alpha=10^8$) variables.

\begin{figure}[H]
	\centering
	\begin{minipage}{0.49\textwidth}
		\includegraphics[width = 1.0\textwidth]{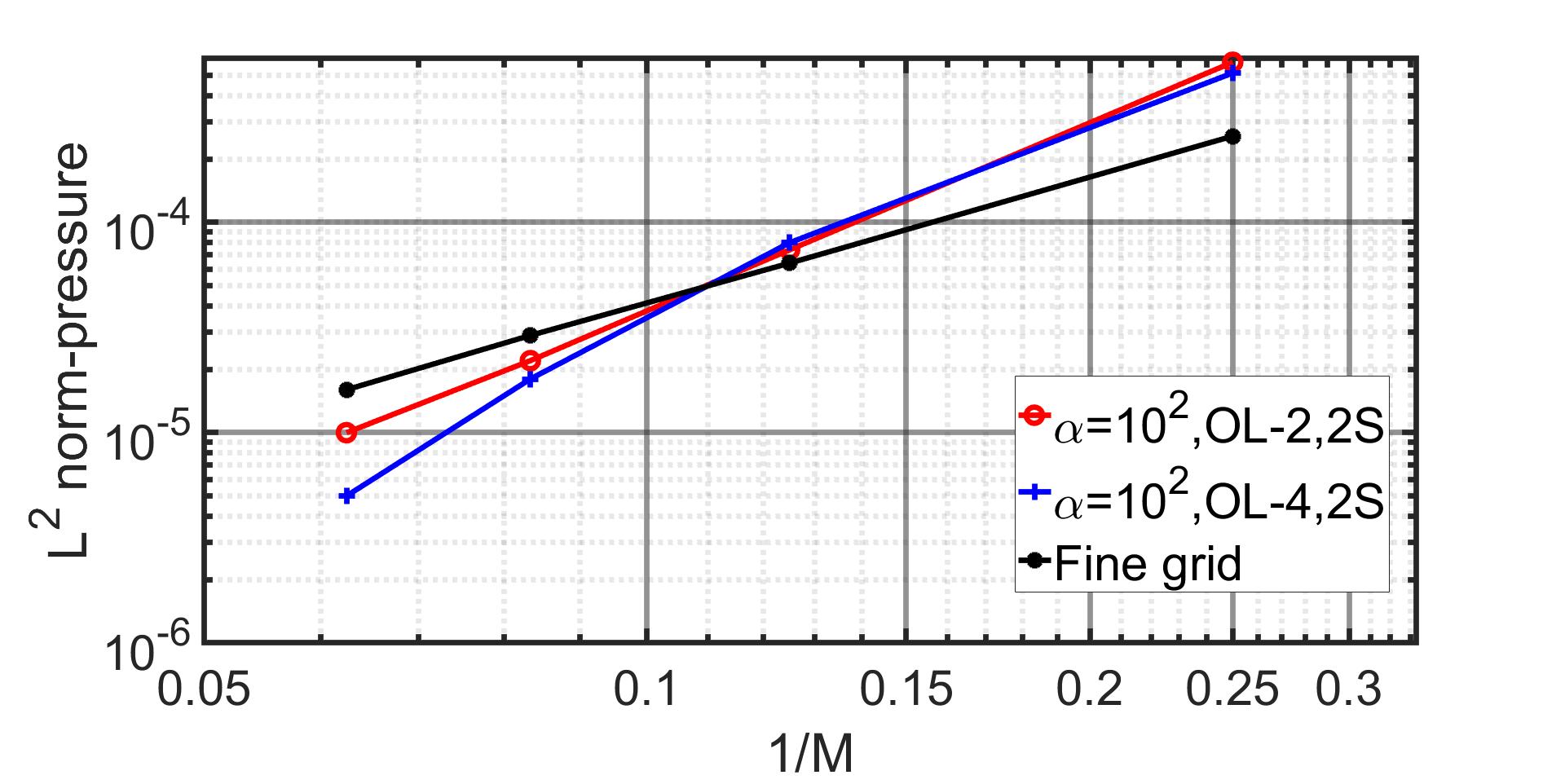}
	\end{minipage}
	\begin{minipage}{0.49\textwidth}
		\includegraphics[width = 1.0\textwidth]{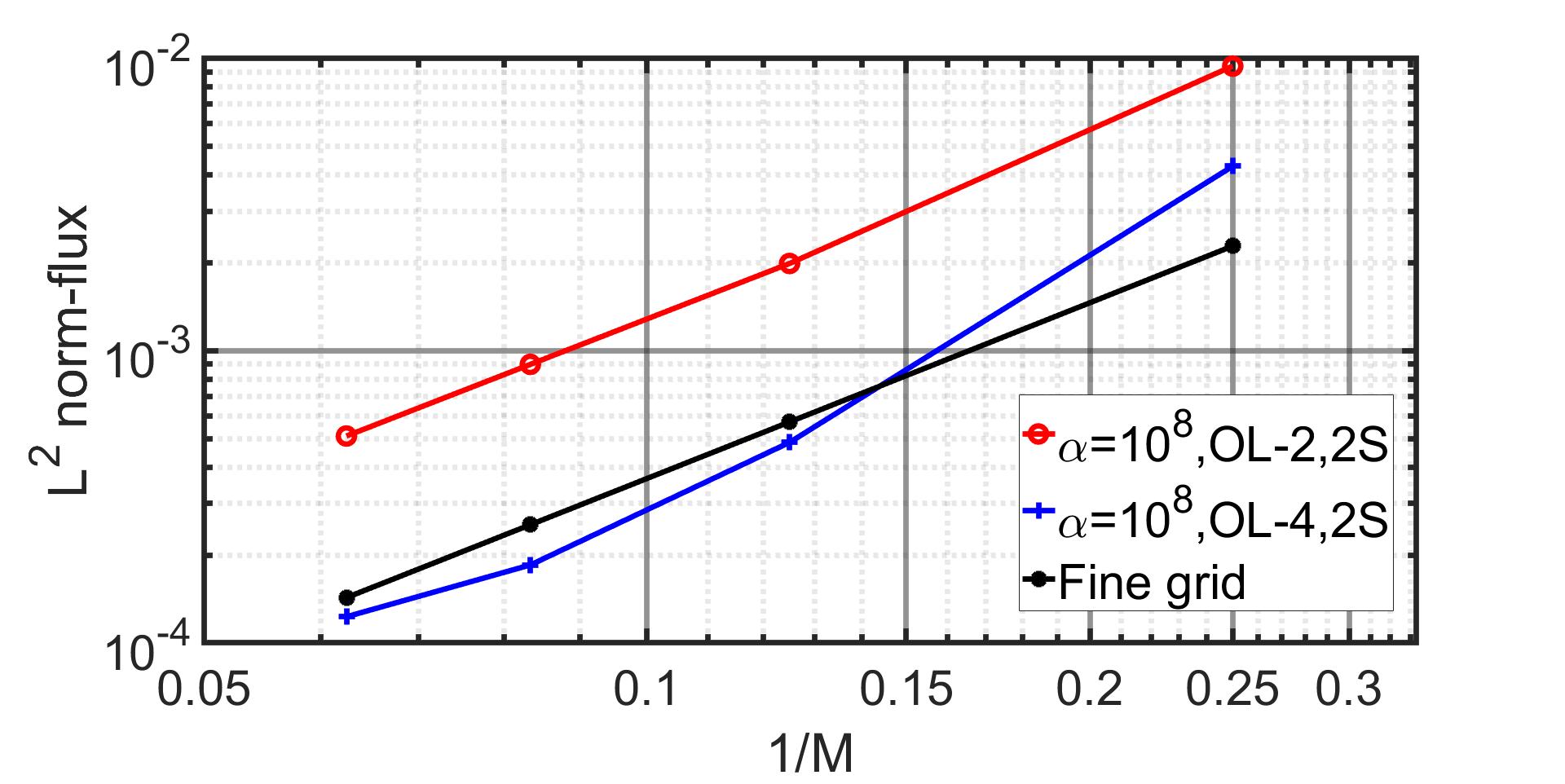}
	\end{minipage}
	\caption{Mesh refinement study for the homogenous problem wherein MRCM-OS exhibits lower errors than the fine grid solution: 
	Pressure absolute error (left) and flux absolute error (right).}
	\label{ana-best}
\end{figure}

Note in Fig. (\ref{ana-best}) that employing 2 smoothing steps yields MRCM-OS with pressure results with lower error rates than those of the fine grid solution when $\alpha=100$. Moreover, comparable accuracy for MRCM-OS flux is observed for an oversampling size of $4h$ when $\alpha=10^8$.

\begin{rmk} 
As noted earlier, the added computational burden of the smoothing steps is minimal, thanks to the reuse of the factorization computed during the initial smoothing step. \end{rmk}

In the following subsections, we will only consider linear polynomial spaces for MRCM-O and MRCM-OS.

\subsection{MRCM-O}

This section illustrates the results obtained from both homogenous and heterogeneous problems using only oversampling techniques.

We first consider the homogenous problem and we present the $L^2(\Omega)$ norm relative errors with respect to the homogeneous solution in Fig. (\ref{analytical-O}) in a variable $\alpha$ study.
These results  illustrate that the convergence rate of our method closely aligns with that of the original MRCM. 
However the errors achieved by our method 
with an oversampling size of $4h$ is nearly half that of the latter for both pressure and flux variables.

\begin{figure}[H]
	\centering
	\begin{minipage}[t]{1\linewidth}
		\includegraphics[width = 0.49\textwidth]{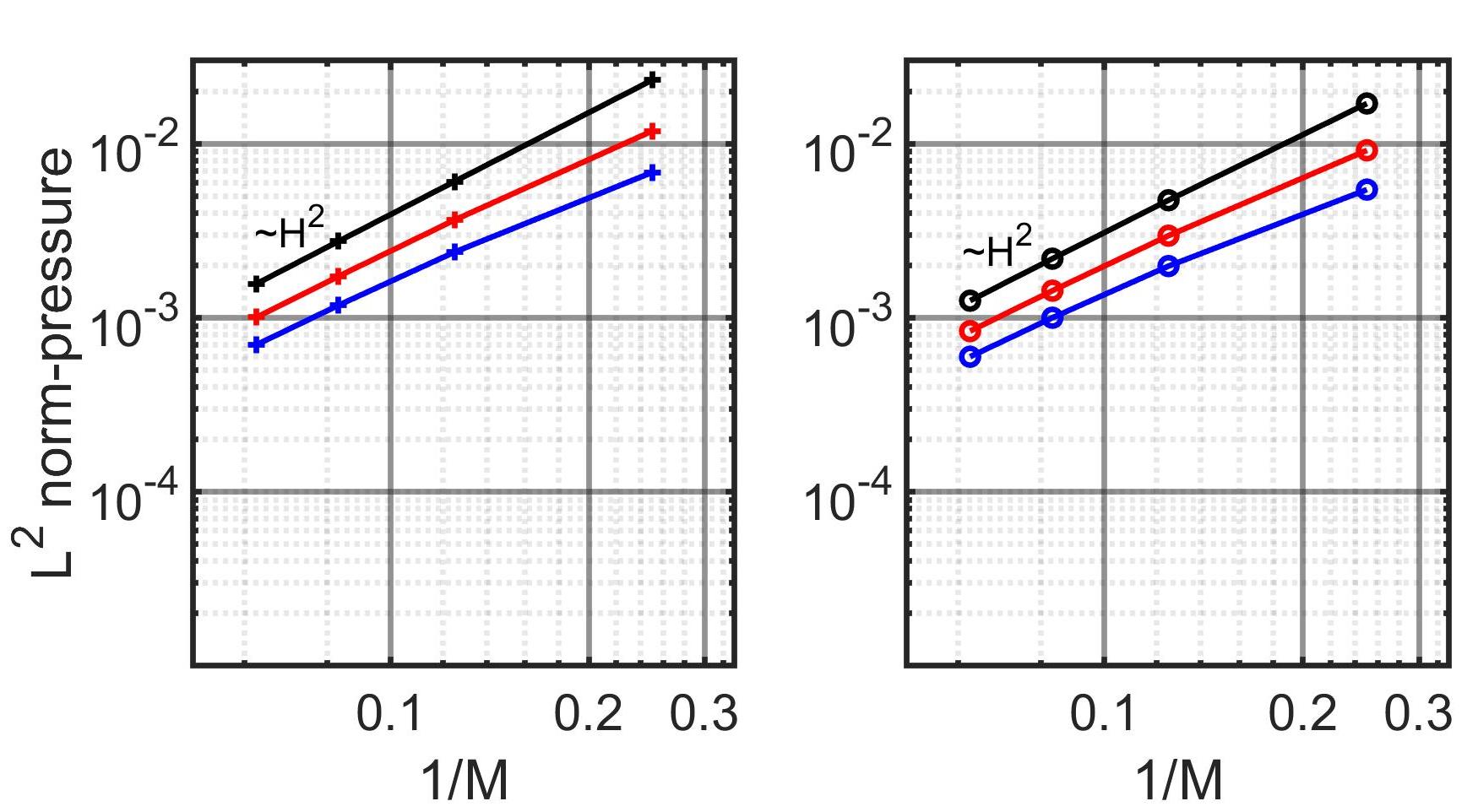}
		\includegraphics[width = 0.49\textwidth]{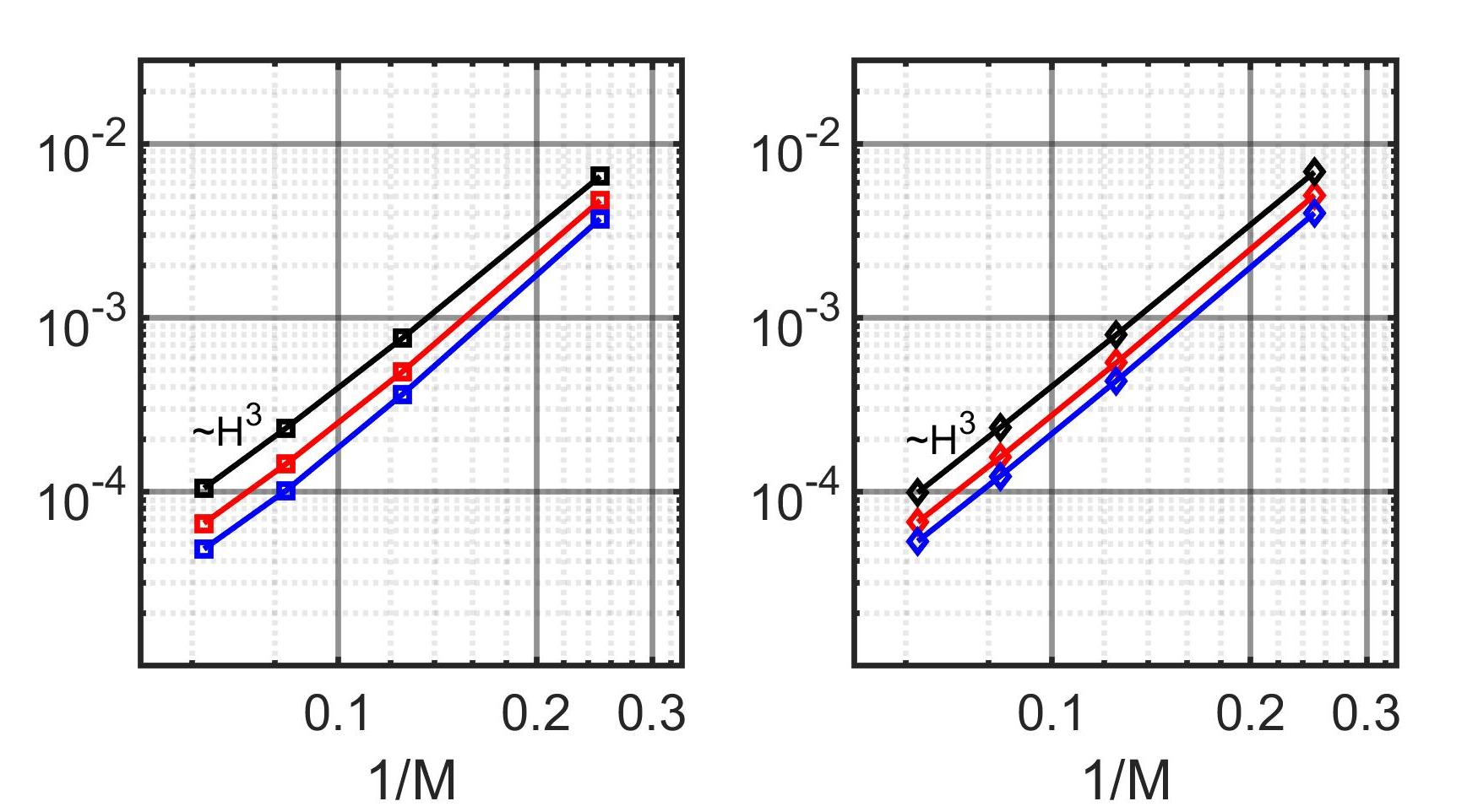}
	\end{minipage}

	\begin{minipage}[t]{1\linewidth}
		\includegraphics[width = 0.49\textwidth]{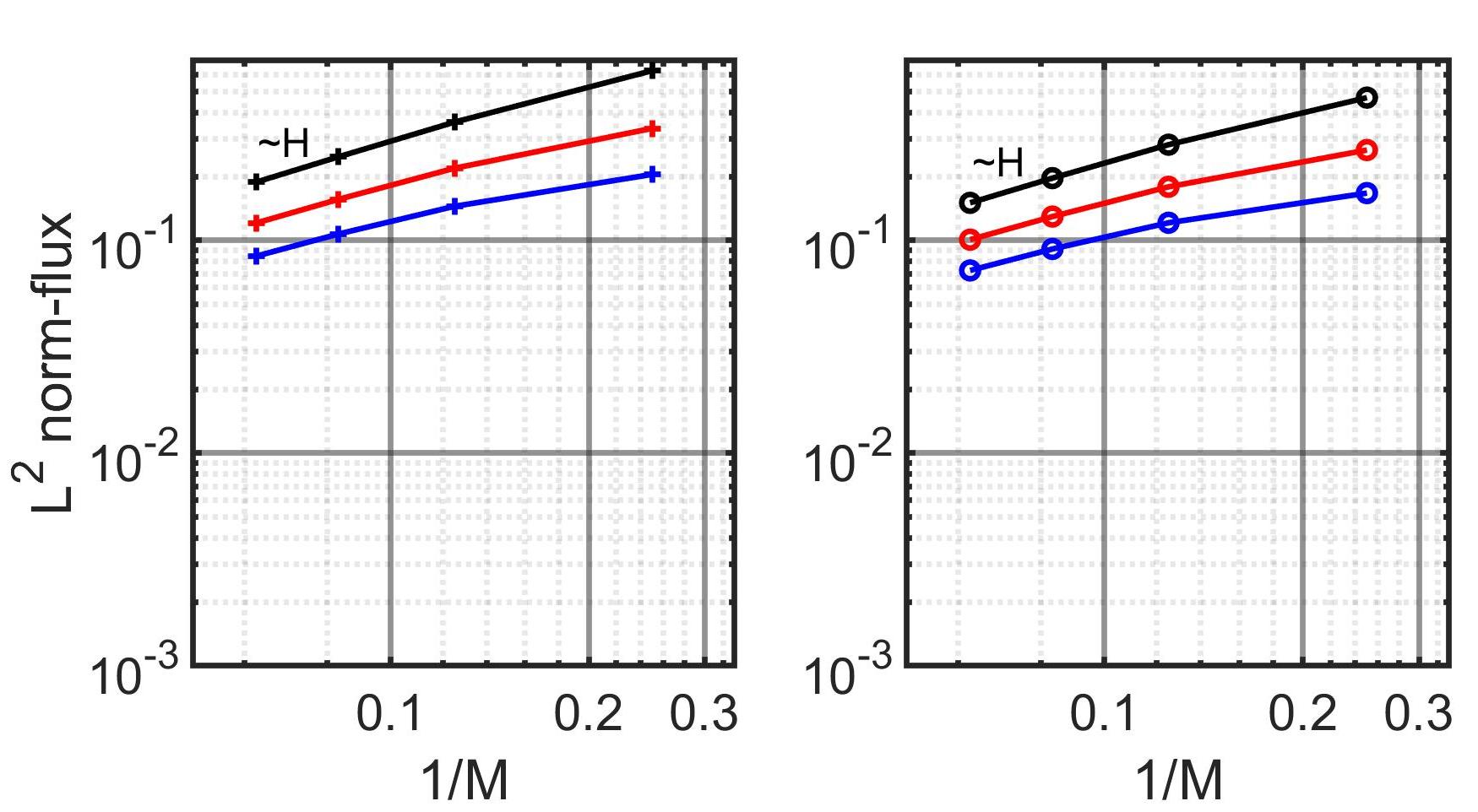}
		\includegraphics[width = 0.49\textwidth]{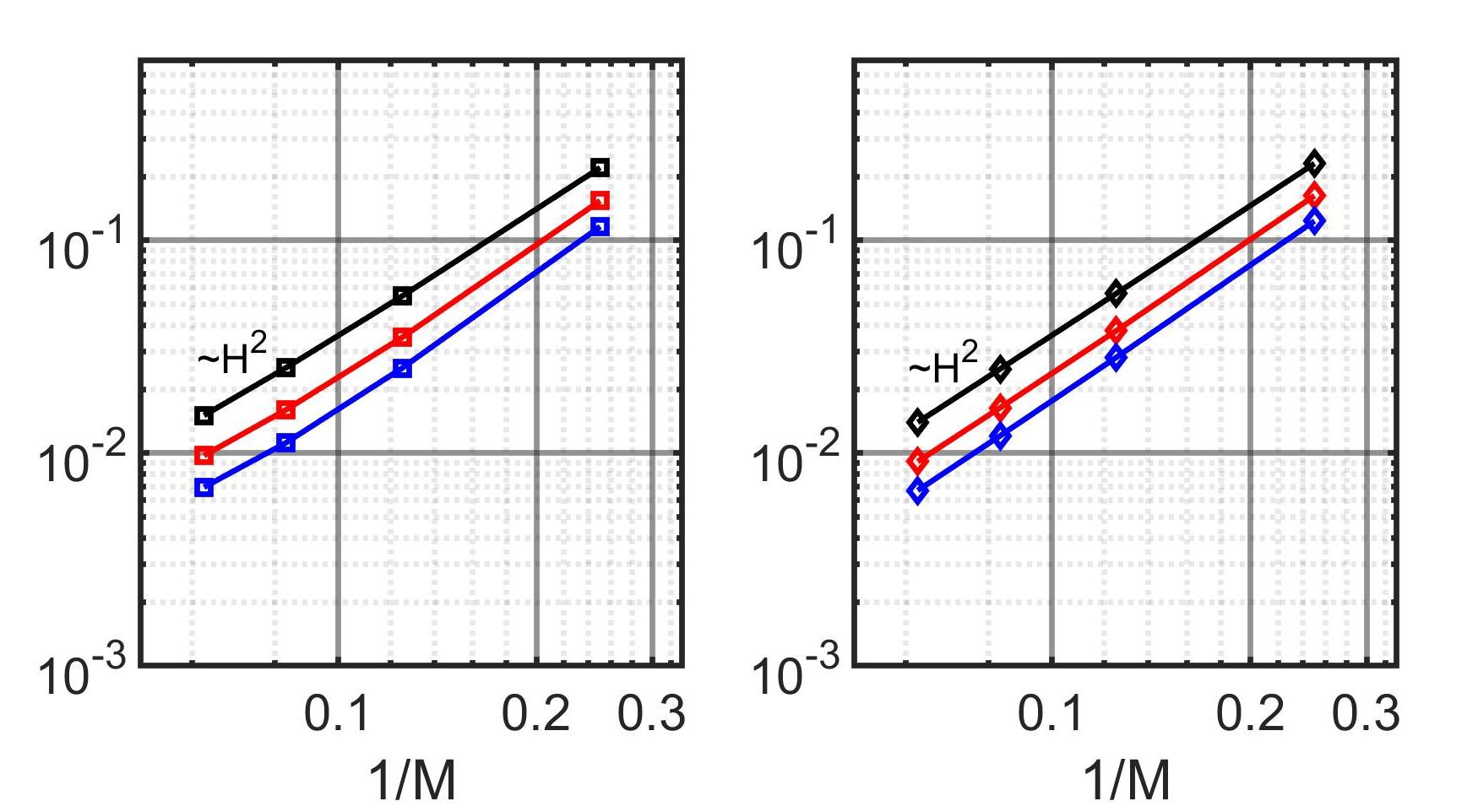}
	\end{minipage}
	
	\begin{minipage}[t]{0.243\linewidth}
	\centering
		\includegraphics[width = 0.5\textwidth]{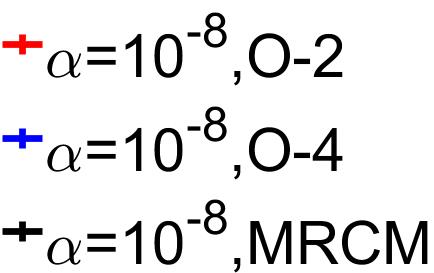}
	\end{minipage}
	\begin{minipage}[t]{0.243\linewidth}
	\centering
		\includegraphics[width = 0.45\textwidth]{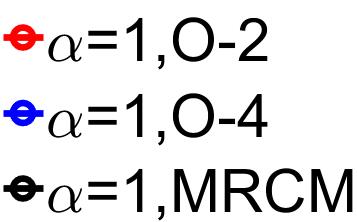}
	\end{minipage}
	\begin{minipage}[t]{0.243\linewidth}
	\centering
		\includegraphics[width = 0.5\textwidth]{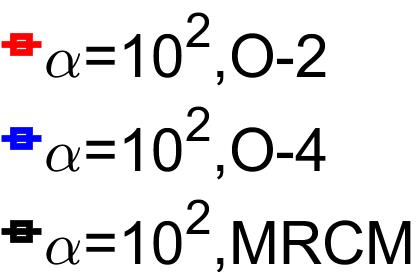}
	\end{minipage}
	\begin{minipage}[t]{0.243\linewidth}
	\centering
		\includegraphics[width = 0.5\textwidth]{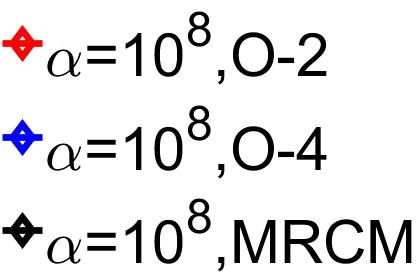}
	\end{minipage}

	\caption{Convergence rate study for the homogeneous problem using only oversampling: Pressure (top) and Flux (bottom).}
	\label{analytical-O}
\end{figure}

Next we present the $L^2(\Omega)$ norm relative error with respect to the fine grid solution for the heterogeneous problem in Fig. (\ref{hete-O}). It is evident that for the heterogeneous problem an oversampling size of $4h$ results in a reduction of error by nearly one order of magnitude.

\begin{figure}[H]
	\centering
	\begin{minipage}[t]{0.49\linewidth}

		\includegraphics[width = 1.0\textwidth]{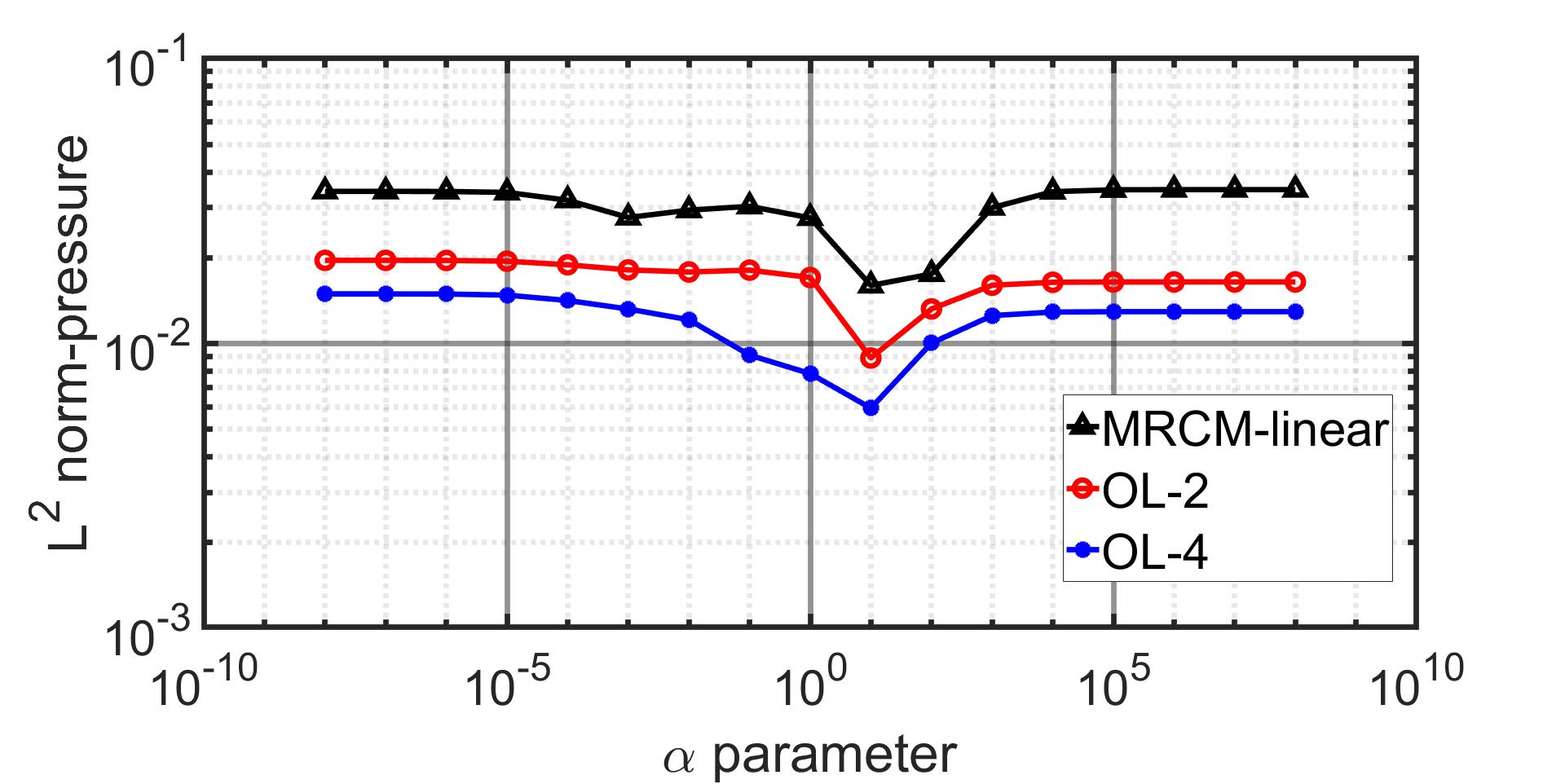}
	\end{minipage}
	\begin{minipage}[t]{0.49\linewidth}
		\includegraphics[width = 1.0\textwidth]{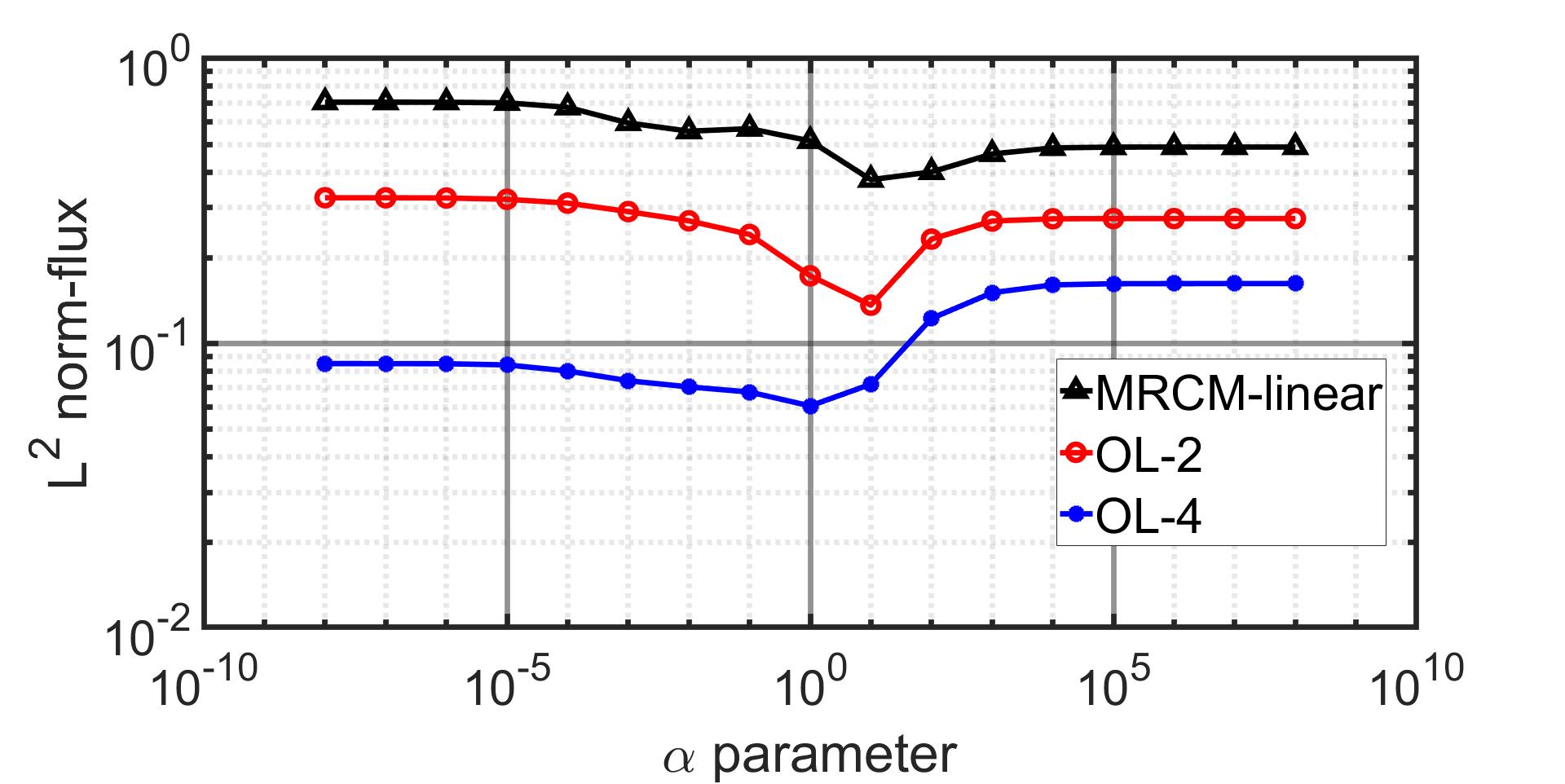}
	\end{minipage}

	\caption{ $\alpha$ parameter study for the heterogeneous problem with only oversampling: pressure (left), flux (right).}
	\label{hete-O}
\end{figure}

In conclusion, oversampling techniques significantly enhance accuracy, with an observed improvement of one order of magnitude. Subsequently, we will demonstrate that additional smoothing steps lead to even higher levels of accuracy in both pressure and flux variables.

\subsection{MRCM-OS}

In this section, we discuss the benefits derived from the application of both oversampling and smoothing techniques.

\subsubsection{Homogenous problem}
Successively applying smoothing steps allows the solution derived from the multiscale method to converge to the fine grid solution. However, an increase in the number of smoothing steps not only incurs higher computational costs but also results in a slower convergence rate. In Fig. (\ref{homo-number}), we employ an $8 \times 8$ domain 
decomposition with an oversampling size of $2h$ to identify an optimal number of smoothing steps. We present the $L^2(\Omega)$ norm absolute error with respect to the fine grid solution. Fig.(\ref{homo-number}) illustrates that the optimal number of smoothing steps lies within the range of 2 to 4, as increasing the number of steps beyond this range results in a diminishing slope.

\begin{figure}[H]
	\centering
	\begin{minipage}{0.49\textwidth}
		\includegraphics[width = 1.0\textwidth]{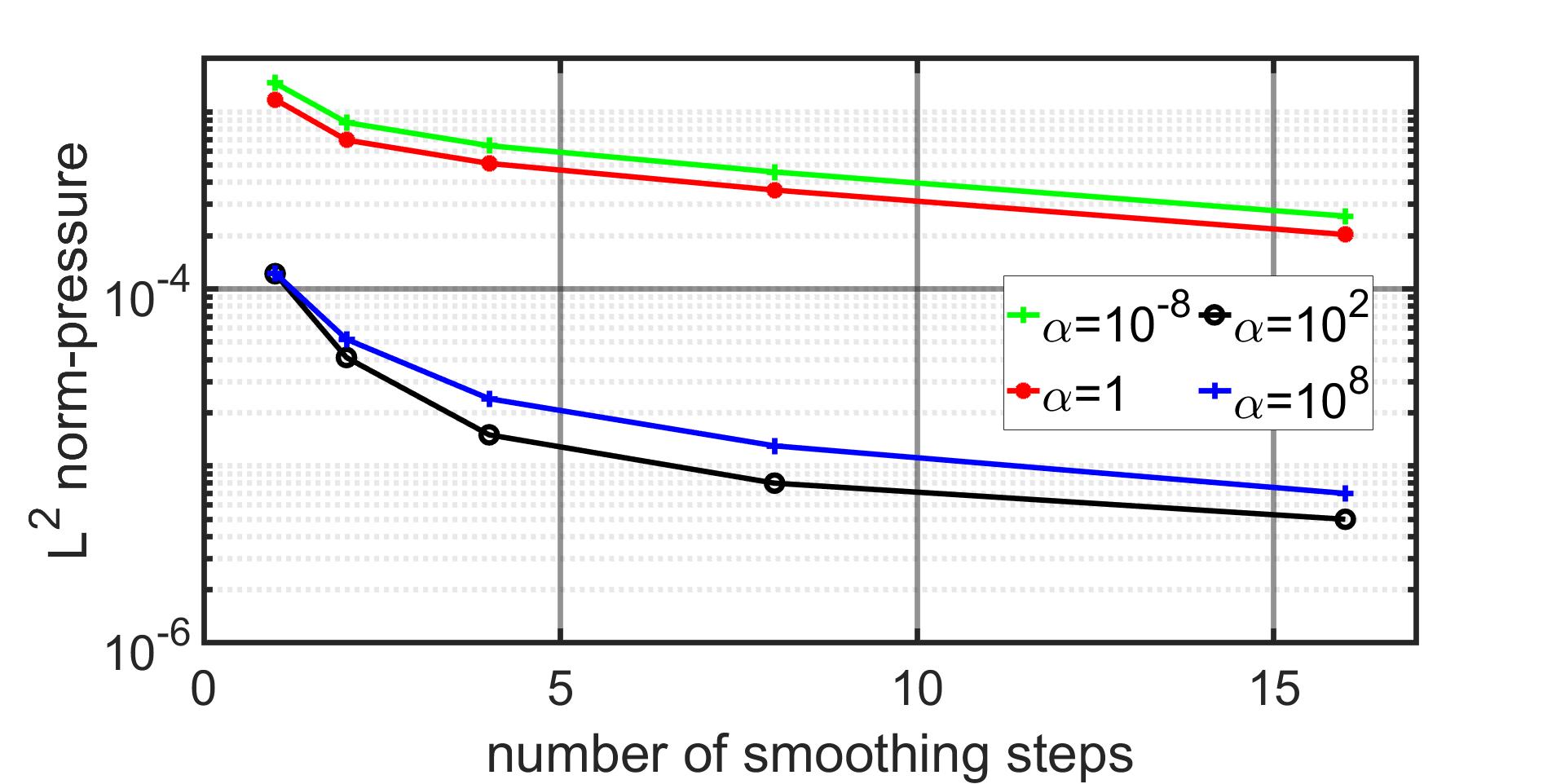}
	\end{minipage}
	\begin{minipage}{0.49\textwidth}
		\includegraphics[width = 1.0\textwidth]{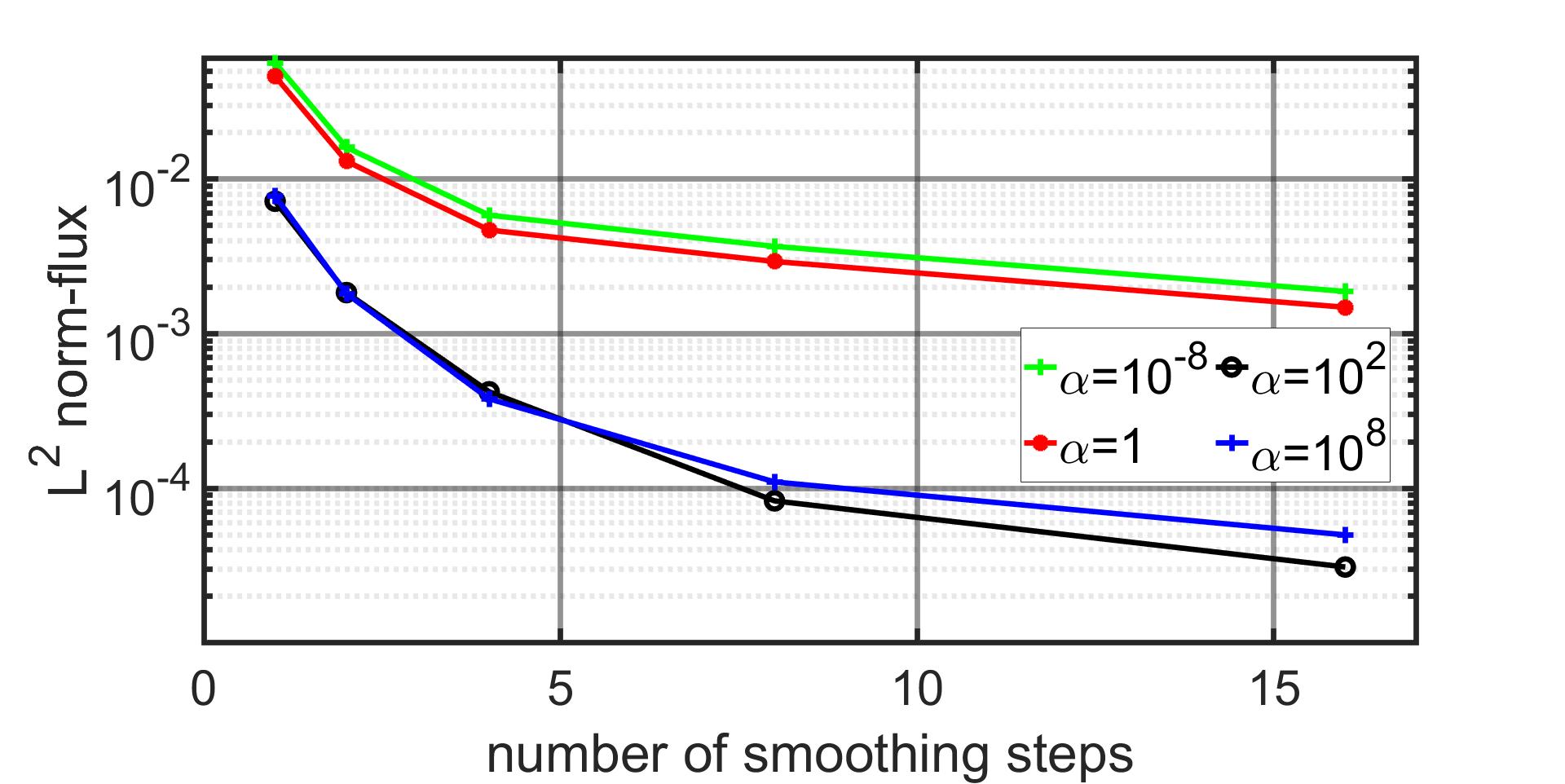}
	\end{minipage}
	\caption{Exploring the optimal number of smoothing steps for the homogeneous problem with oversampling size $2h$: 
	Absolute pressure error (left) and absolute flux error (right).}
	\label{homo-number}
\end{figure}

Next, we consider a study of parameter $\alpha$ concerning the convergence rate with oversampling size set to $2h$, and the number of smoothing steps set to 2 and 4. 
Fig. (\ref{analytical-OS}) displays the $L^2(\Omega)$ norm absolute error relative to the analytical solution. 
Again, in line with the results reported in Fig. (\ref{analytical-O})
these results  illustrate that the convergence rate of our method closely aligns with that of the original MRCM.
Particularly, our method achieves a one-order-of-magnitude improvement in the pressure variable and two orders-of-magnitude improvement in the flux variable.

\begin{figure}[H]
	\centering
	\begin{minipage}[t]{1\linewidth}
		\includegraphics[width = 0.49\textwidth]{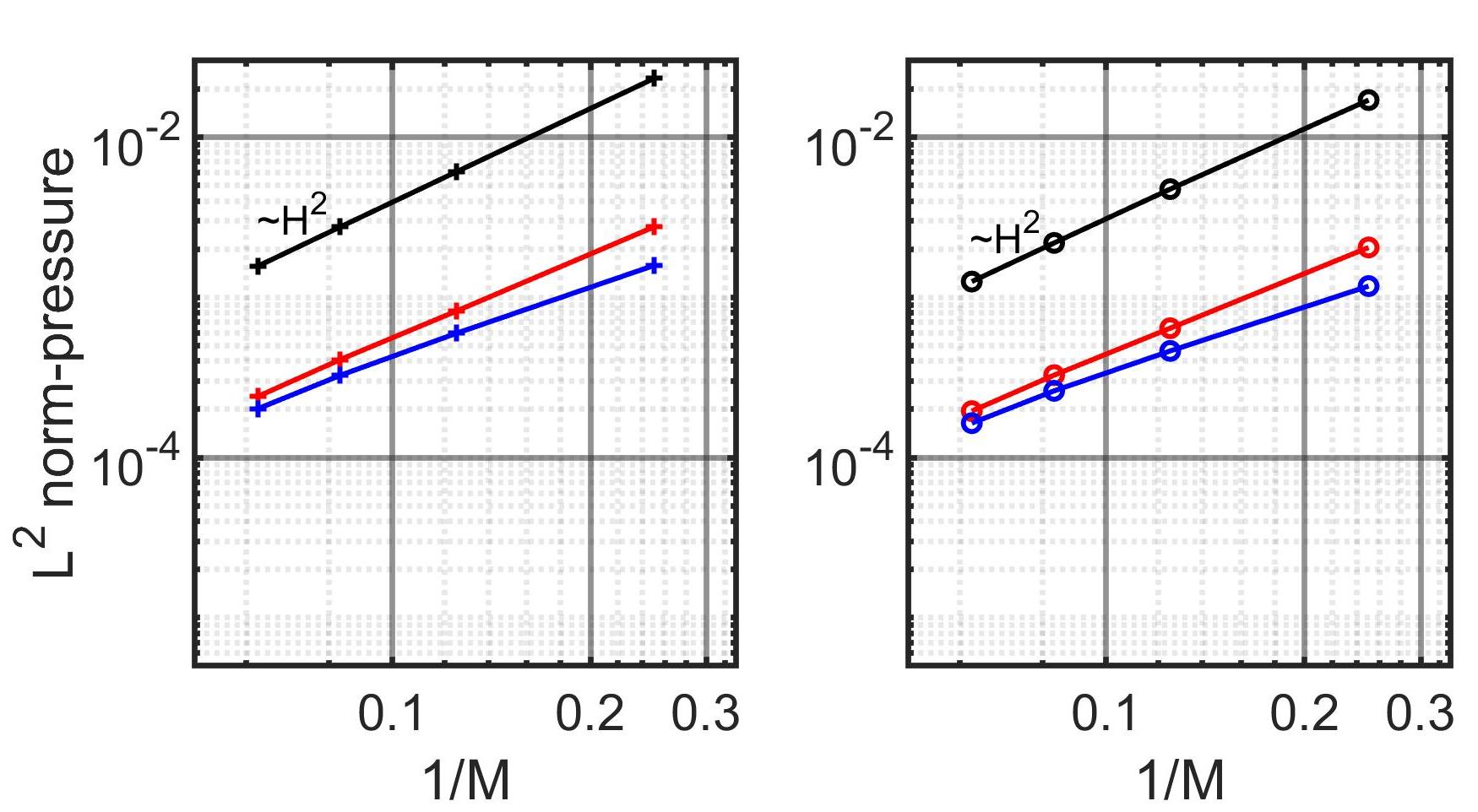}
		\includegraphics[width = 0.49\textwidth]{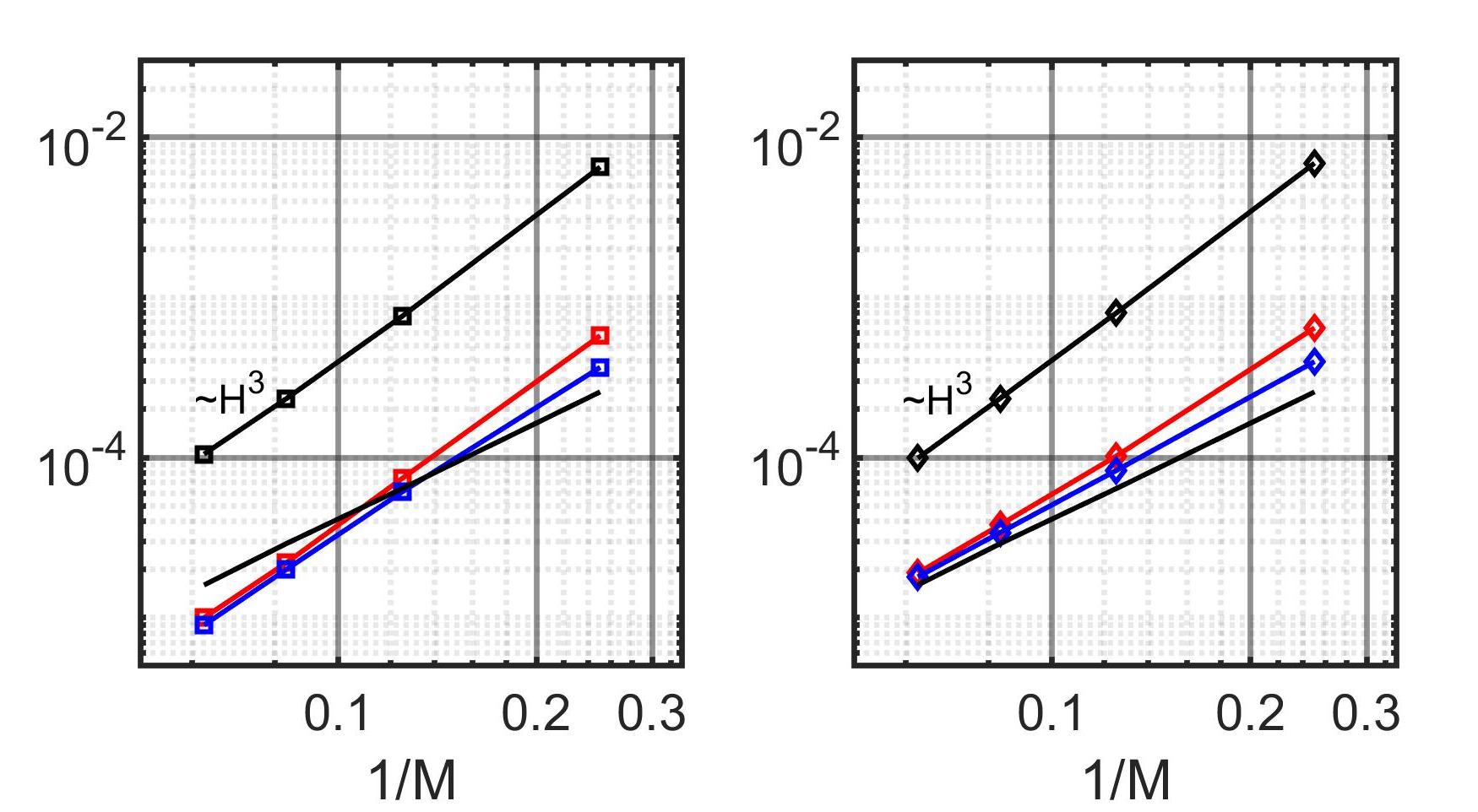}
	\end{minipage}

	\begin{minipage}[t]{1\linewidth}
		\includegraphics[width = 0.49\textwidth]{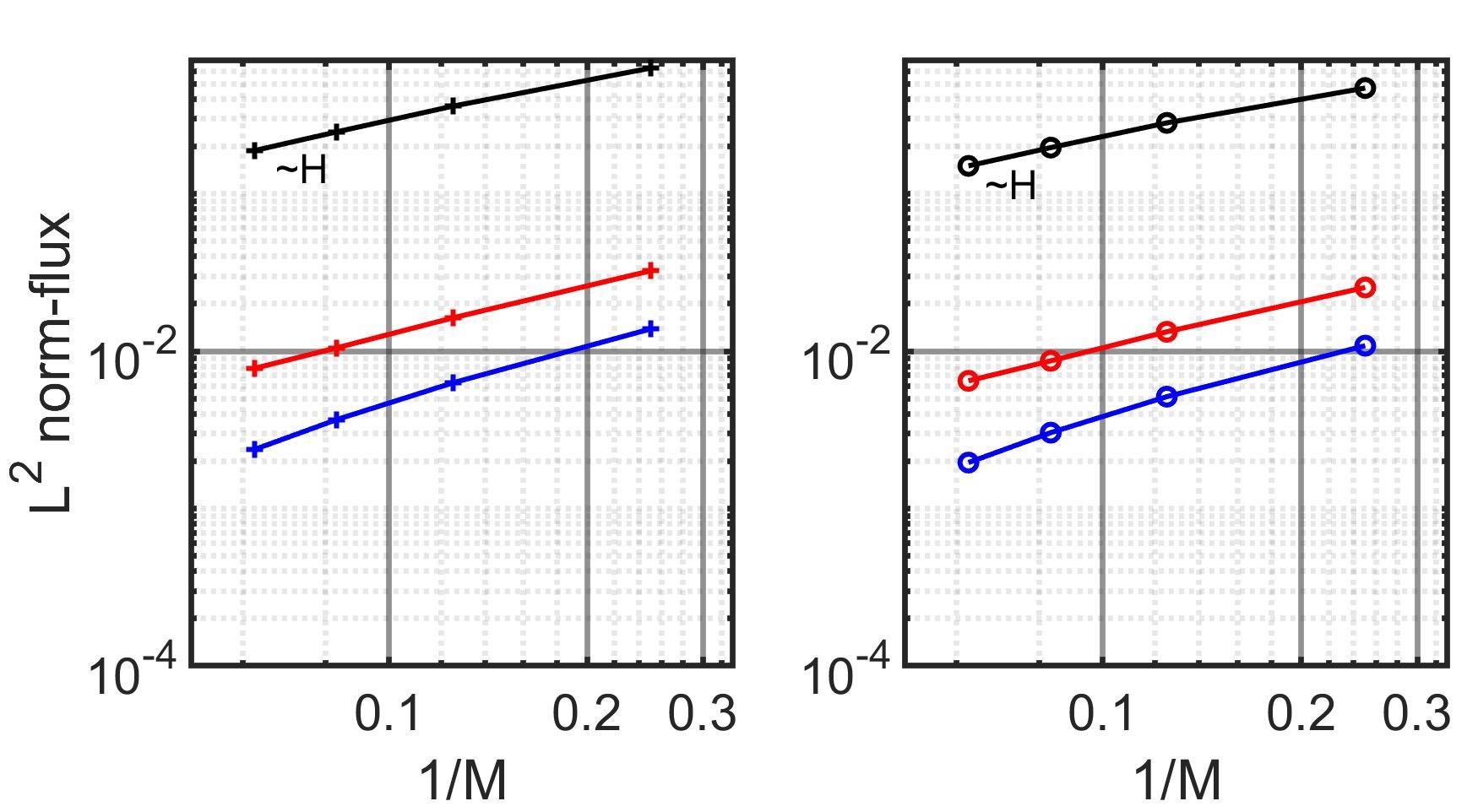}
		\includegraphics[width = 0.49\textwidth]{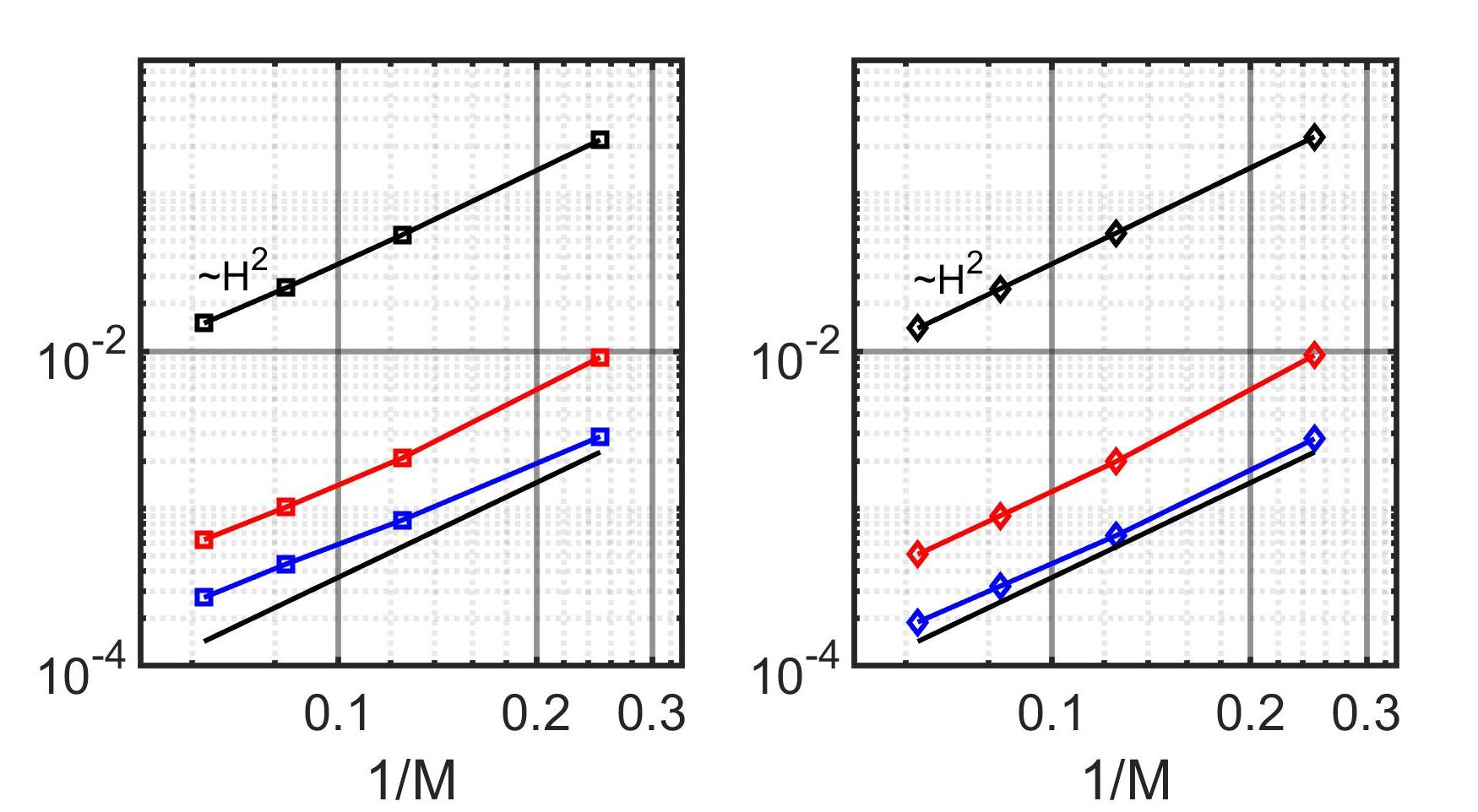}
	\end{minipage}
	
	\begin{minipage}[t]{0.243\linewidth}
	\centering
		\includegraphics[width = 0.5\textwidth]{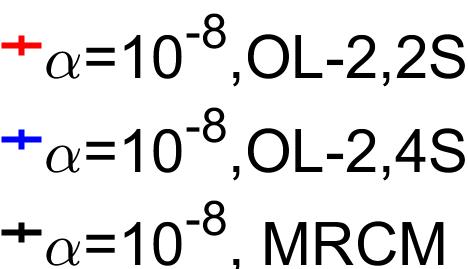}
	\end{minipage}
	\begin{minipage}[t]{0.243\linewidth}
	\centering
		\includegraphics[width = 0.4\textwidth]{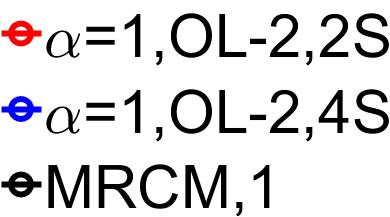}
	\end{minipage}
	\begin{minipage}[t]{0.243\linewidth}
	\centering
		\includegraphics[width = 0.5\textwidth]{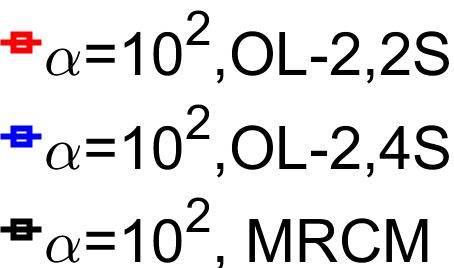}
	\end{minipage}
	\begin{minipage}[t]{0.243\linewidth}
	\centering
		\includegraphics[width = 0.5\textwidth]{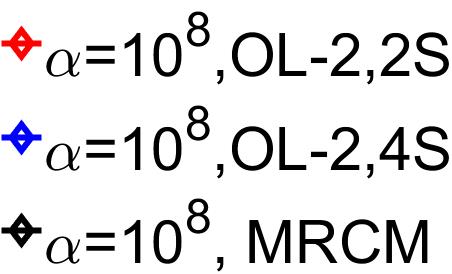}
	\end{minipage}
	\caption{Convergence rate study for the homogeneous problem with oversampling size $2h$: Pressure (top) and Flux (bottom).}
	\label{analytical-OS}
\end{figure}

In order to better understand the improvements introduced by MRCM-OS, next we utilize the multiscale solutions (velocity and pressure) derived from both the original MRCM and our method across four distinct configurations to compute the difference with the analytical solution.  In Fig. (\ref{anacolor-1e-8}), we set $\alpha=10^{-8}$, to get a  MMMFEM-like solution. In Figs. (\ref{anacolor-1})-(\ref{anacolor-100}) we set intermediate values $\alpha=1$ or $100$. In Fig. (\ref{anacolor-1e8}), we set $\alpha=10^8$,  to obtain a MHM-like solution. In these studies the oversampling size is set to $2h$ and two smoothing steps were applied.

\begin{figure}[H]
	\centering

	\begin{minipage}{0.6\textwidth}
		\centering
		\includegraphics[width = 1.0\textwidth]{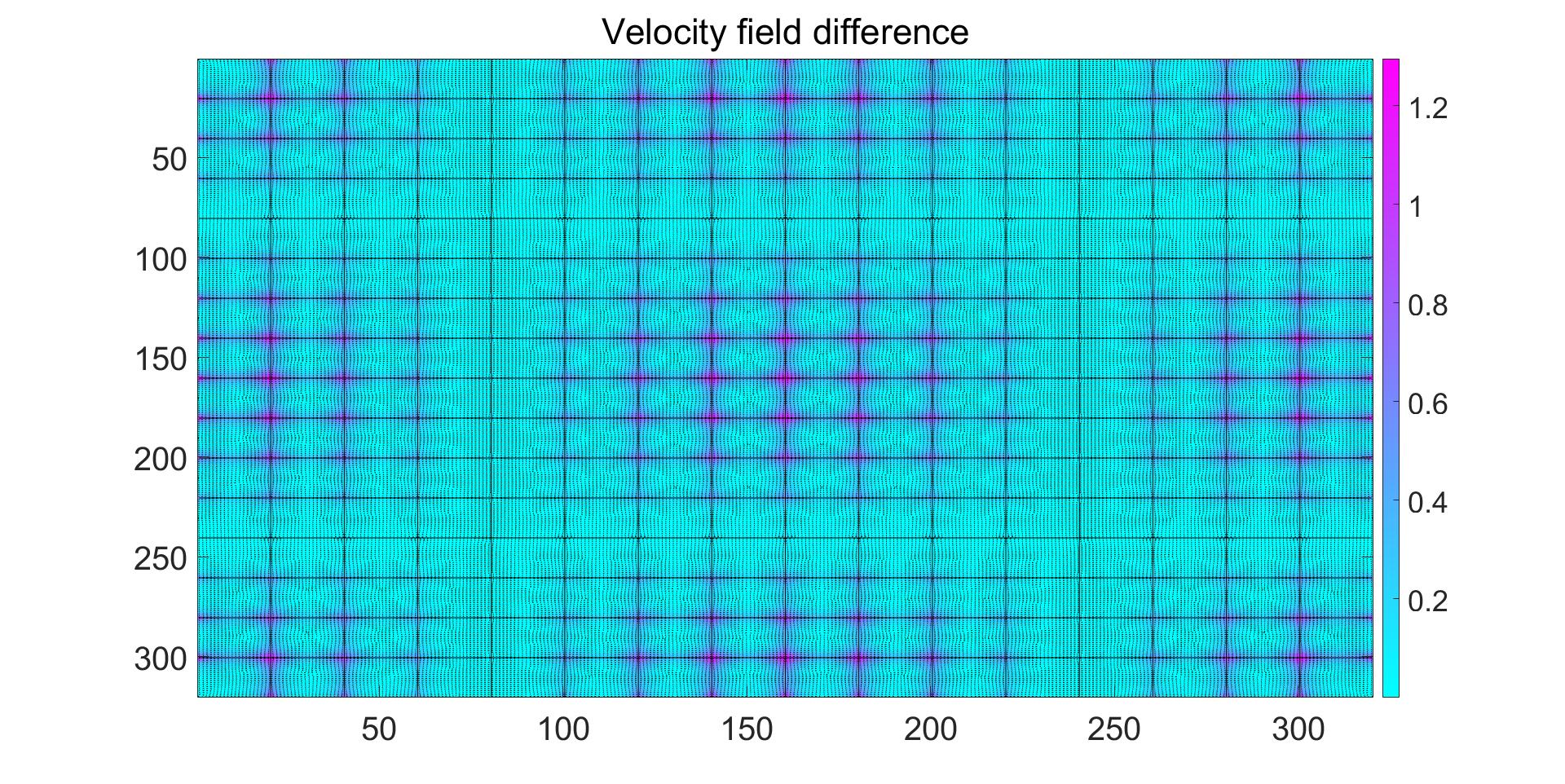}
	\end{minipage}
	\begin{minipage}{0.39\textwidth}\vspace{1cm}
		\centering
		\includegraphics[width = 1.0\textwidth]{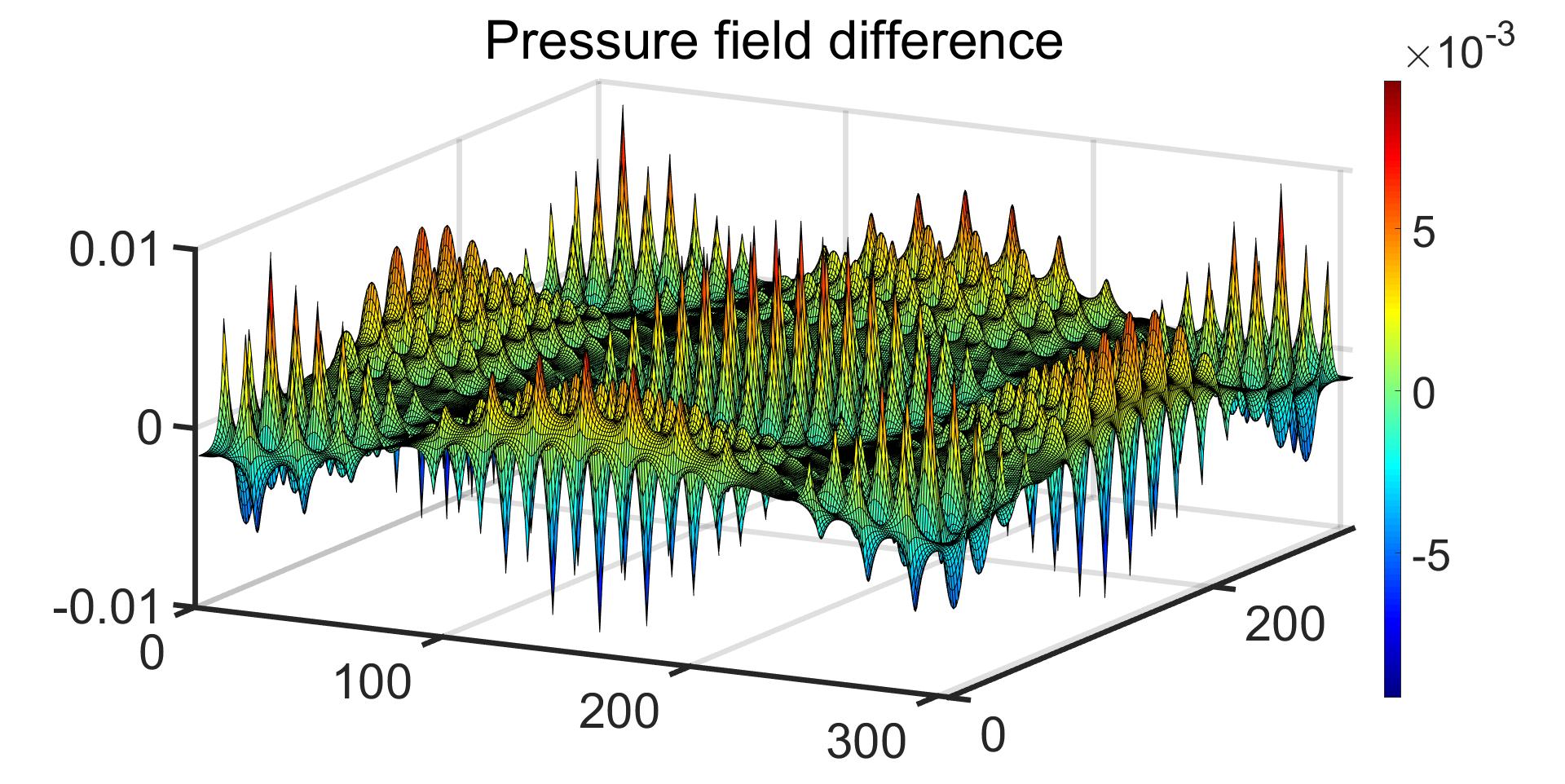}
	\end{minipage}

	\begin{minipage}{0.6\textwidth}
		\centering
		\includegraphics[width = 1.0\textwidth]{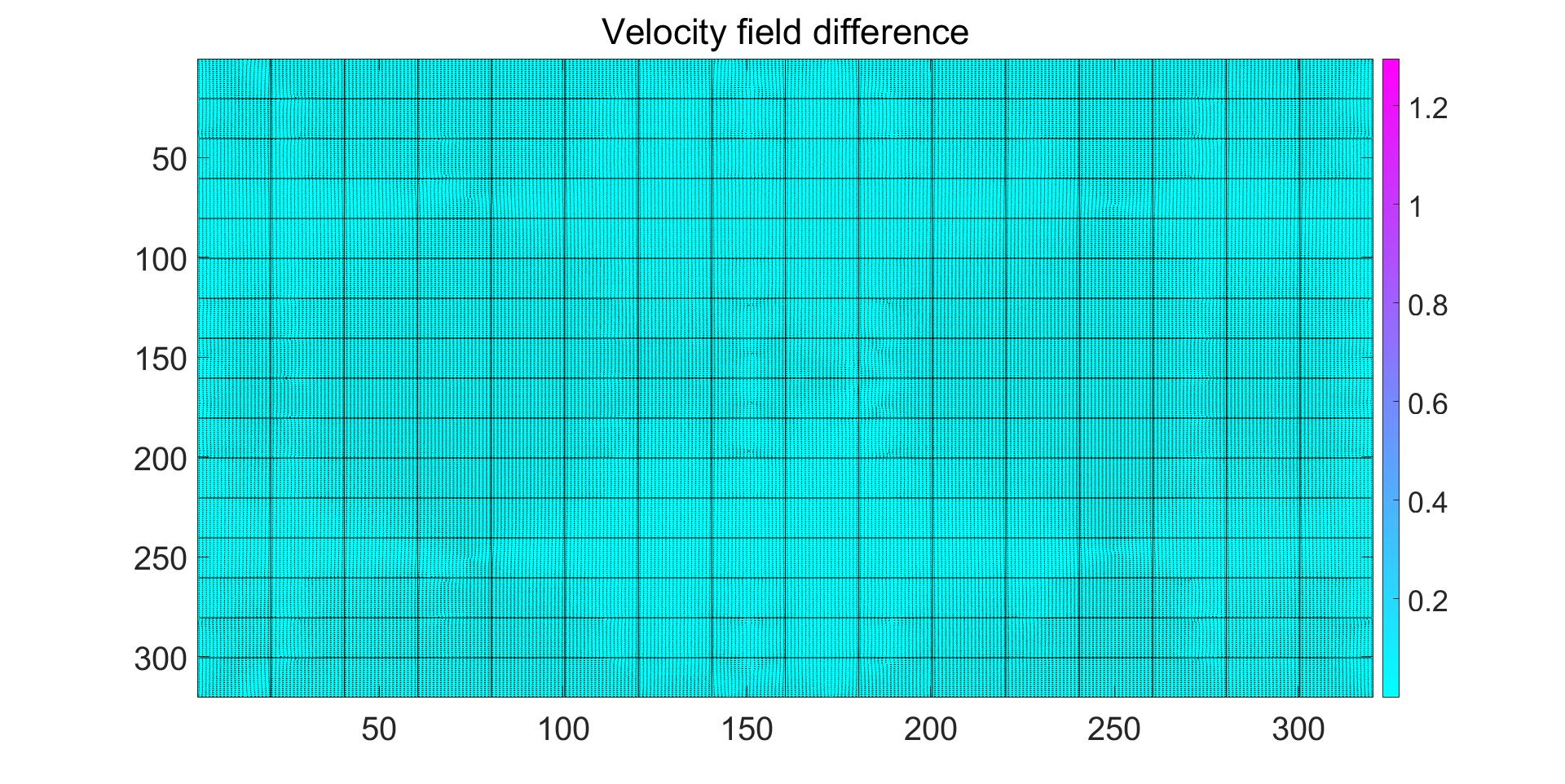}
	\end{minipage}
	\begin{minipage}{0.39\textwidth}\vspace{1cm}
		\centering
		\includegraphics[width = 1.0\textwidth]{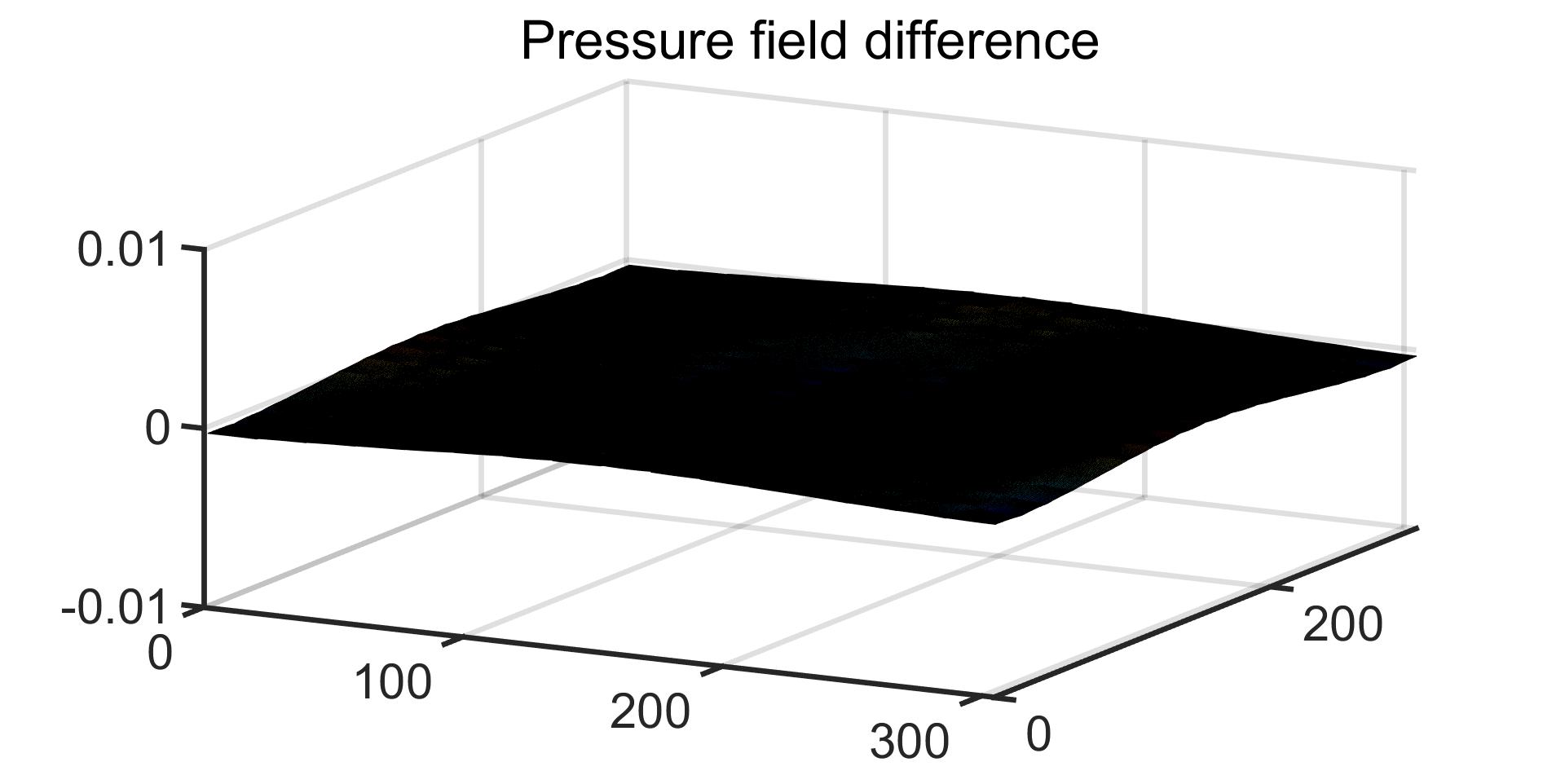}
	\end{minipage}
	\caption{Difference in multiscale solutions for the homogeneous problem with respect to the analytical solution, using different methods for Robin condition parameter $\alpha=10^{-8}$: MRCM method (top), Our method (bottom).}
	\label{anacolor-1e-8}
\end{figure}

\begin{figure}[H]
	\centering

	\begin{minipage}{0.6\textwidth}
		\centering
		\includegraphics[width = 1.0\textwidth]{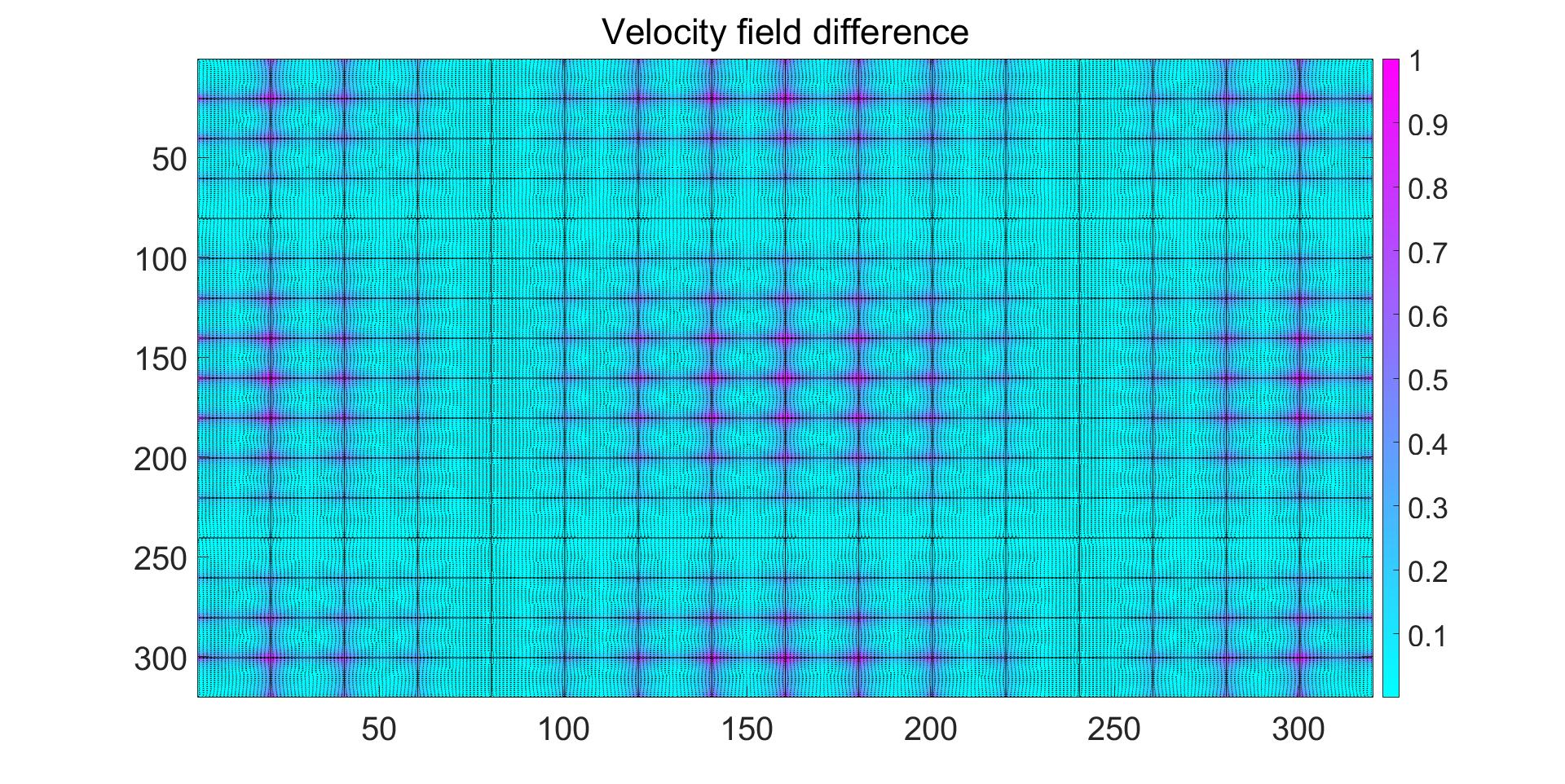}
	\end{minipage}
	\begin{minipage}{0.39\textwidth}\vspace{1cm}
		\centering
		\includegraphics[width = 1.0\textwidth]{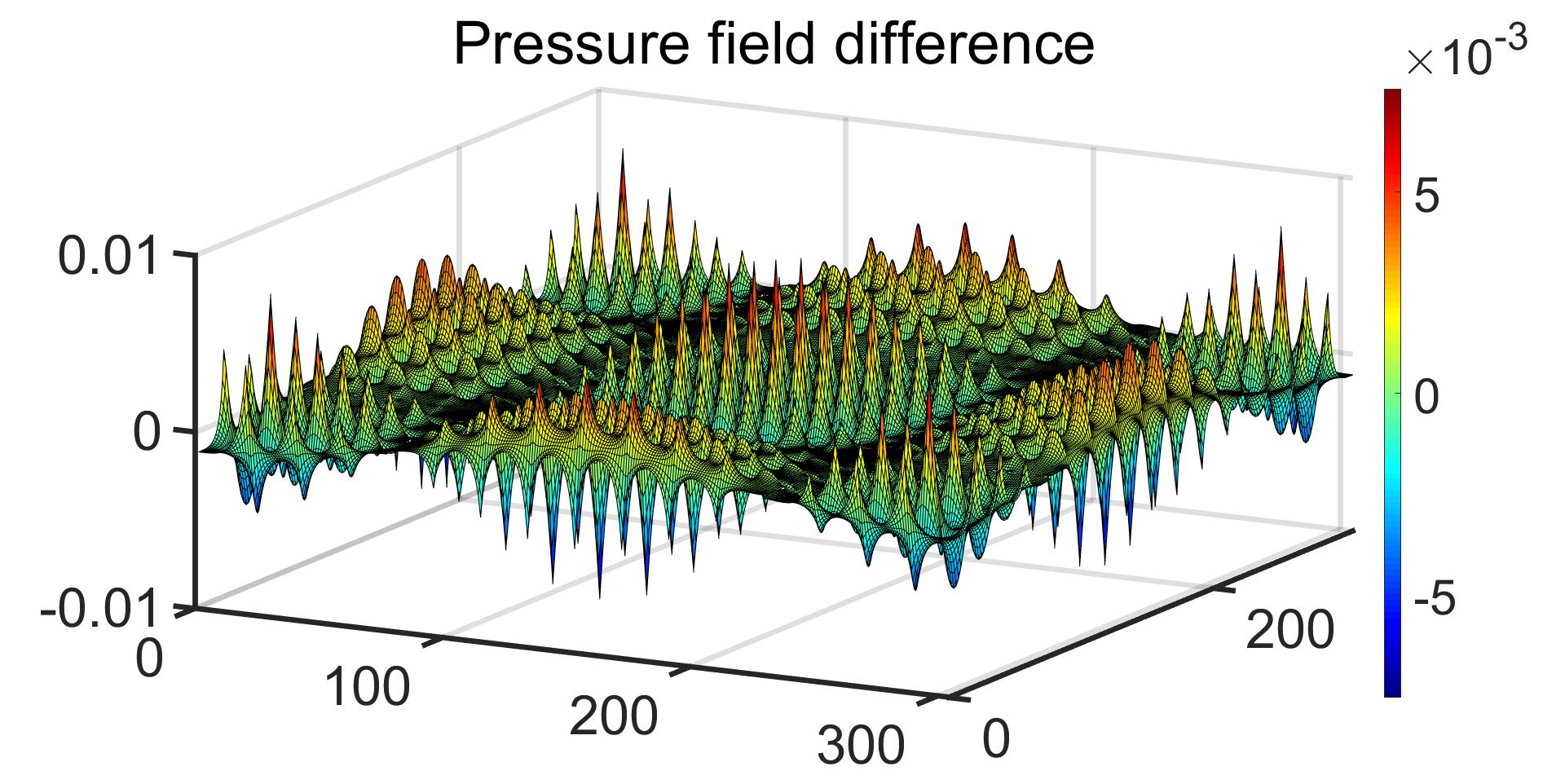}
	\end{minipage}

	\begin{minipage}{0.6\textwidth}
		\centering
		\includegraphics[width = 1.0\textwidth]{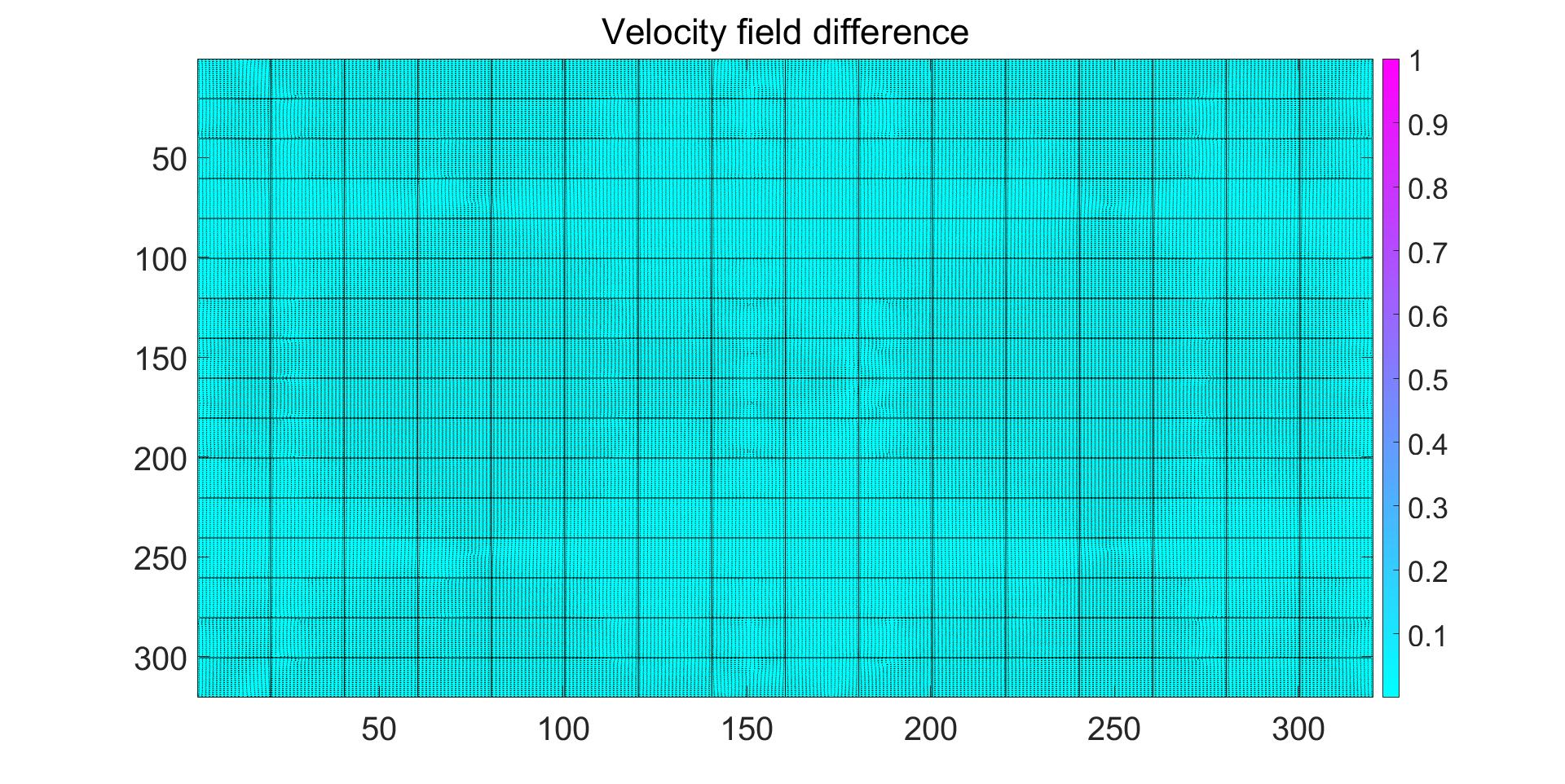}
	\end{minipage}
	\begin{minipage}{0.39\textwidth}\vspace{1cm}
		\centering
		\includegraphics[width = 1.0\textwidth]{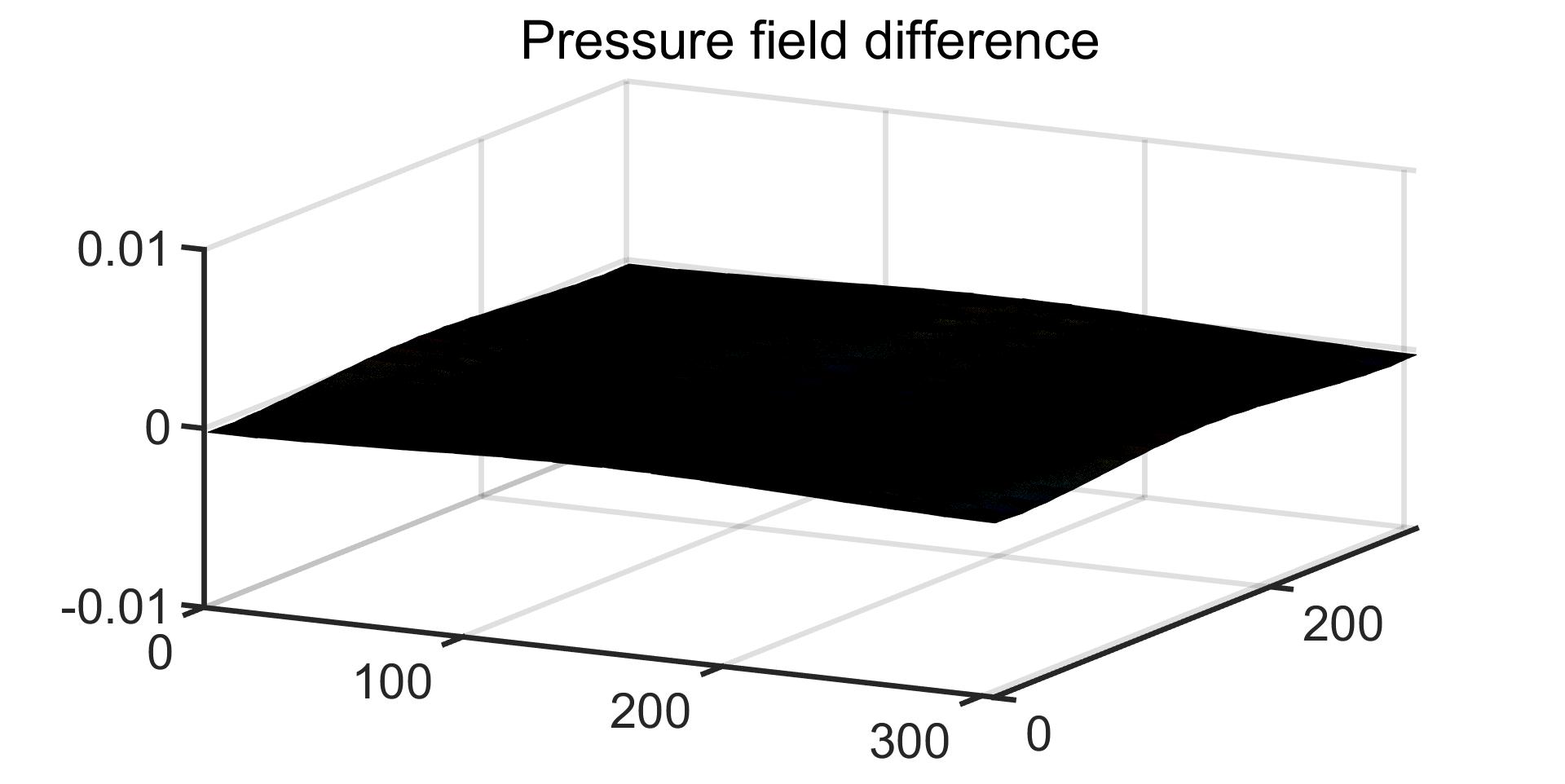}
	\end{minipage}
	\caption{Difference in multiscale solutions for the homogeneous problem with respect to the analytical solution, using different methods for Robin condition parameter $\alpha=1$: MRCM method (top), Our method (bottom).}
	\label{anacolor-1}
\end{figure}

\begin{figure}[H]
	\centering

	\begin{minipage}{0.6\textwidth}
		\centering
		\includegraphics[width = 1.0\textwidth]{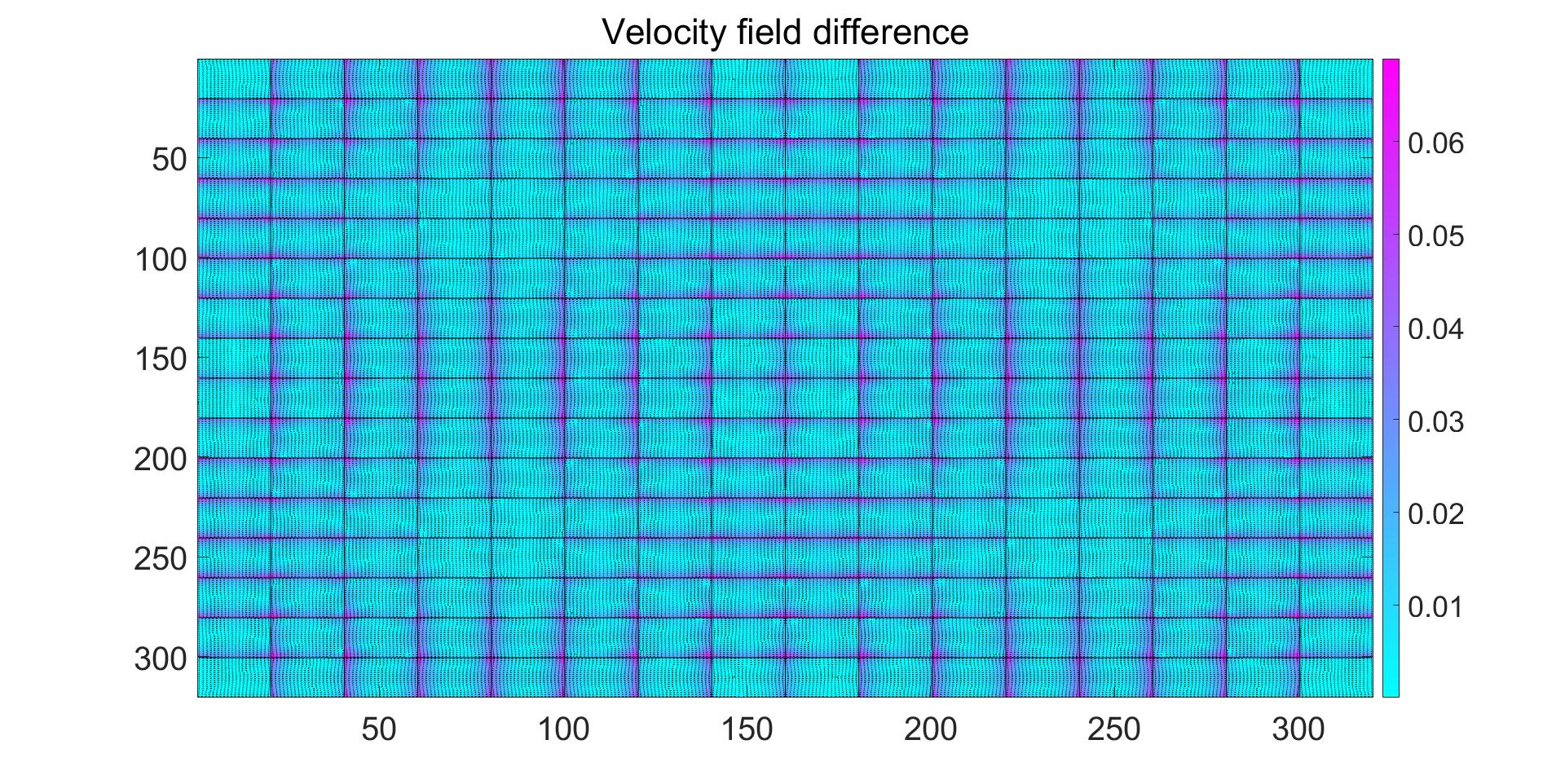}
	\end{minipage}
	\begin{minipage}{0.39\textwidth}\vspace{1cm}
		\centering
		\includegraphics[width = 1.0\textwidth]{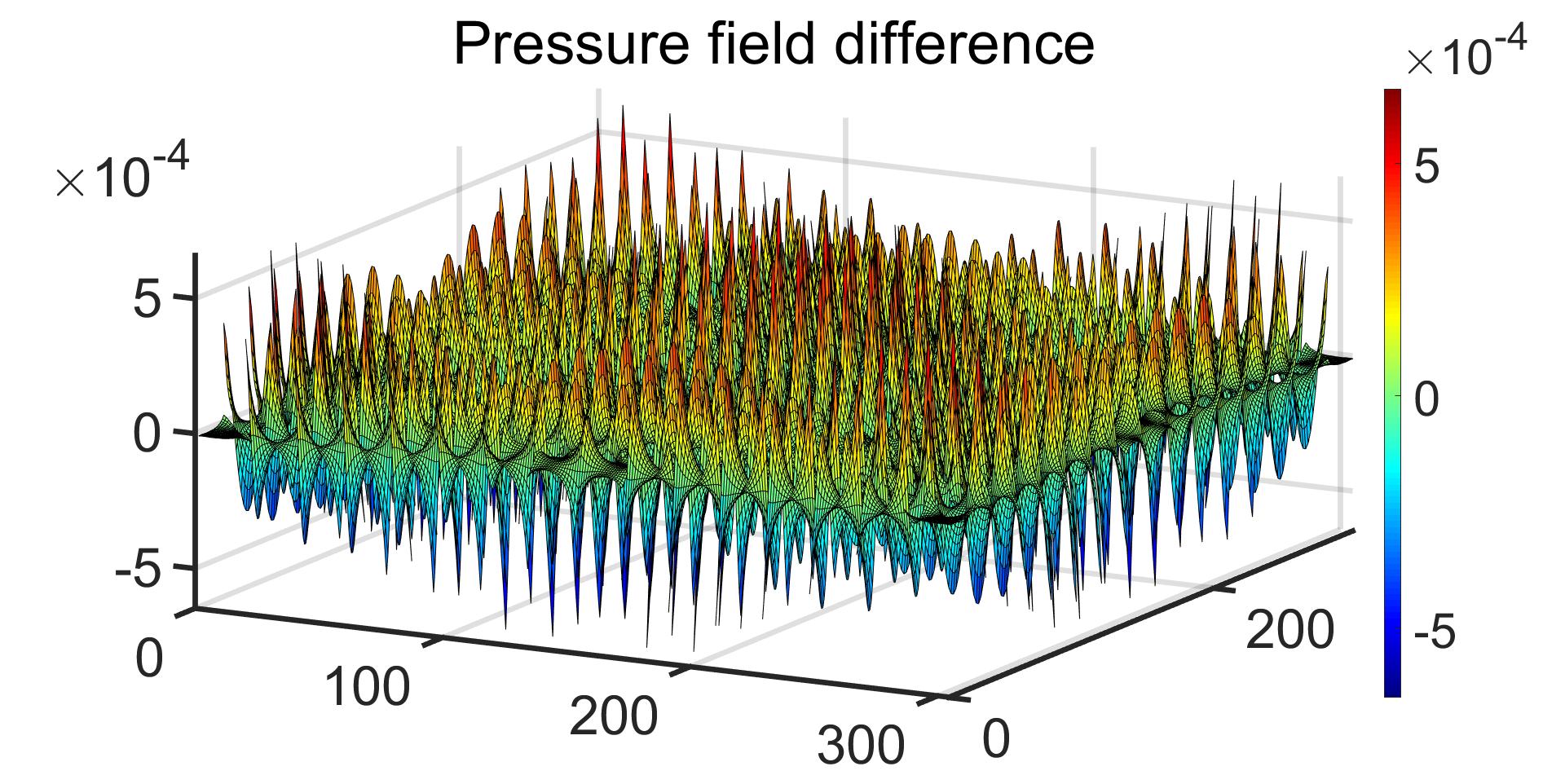}
	\end{minipage}

	\begin{minipage}{0.6\textwidth}
		\centering
		\includegraphics[width = 1.0\textwidth]{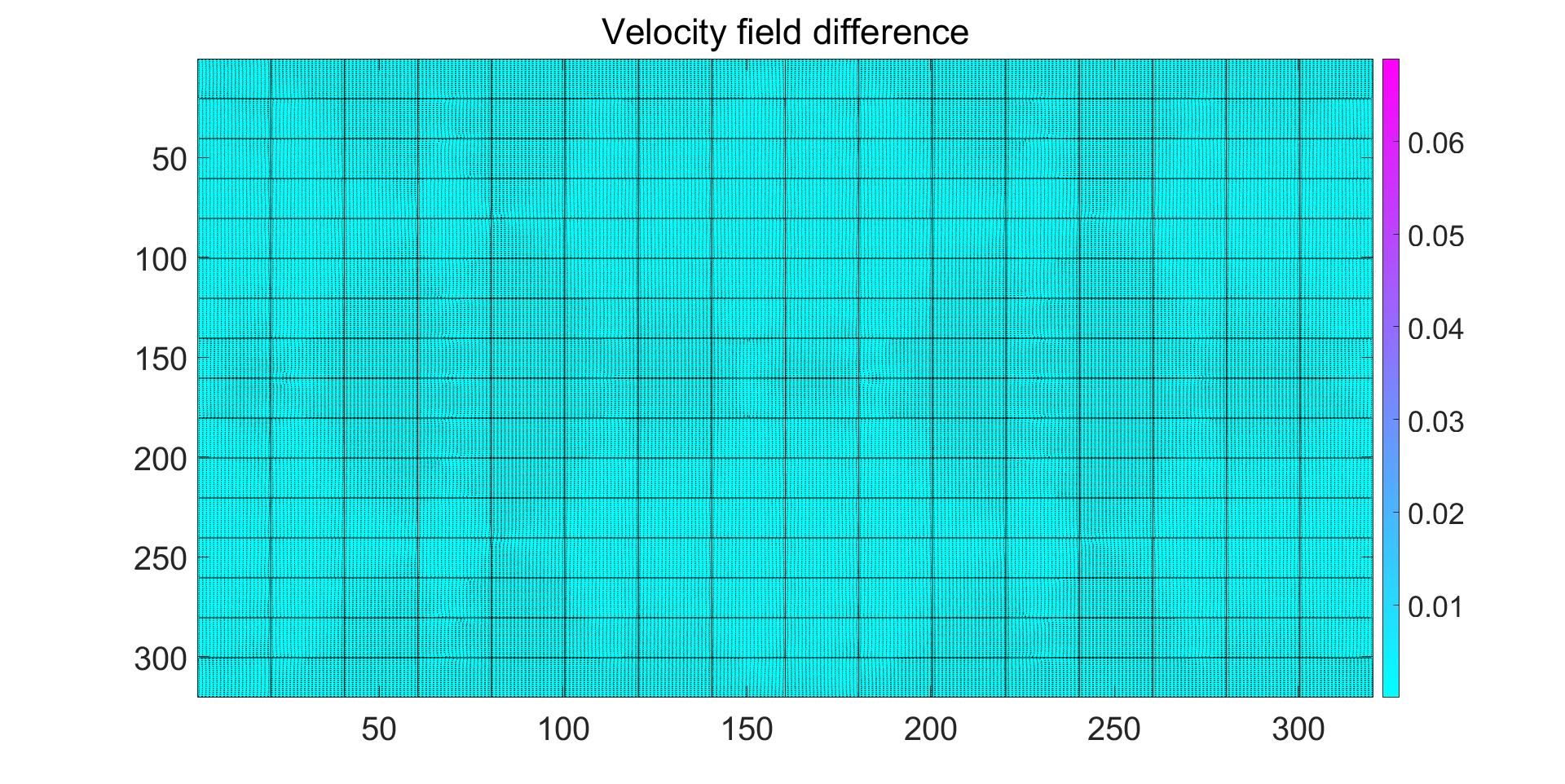}
	\end{minipage}
	\begin{minipage}{0.39\textwidth}\vspace{1cm}
		\centering
		\includegraphics[width = 1.0\textwidth]{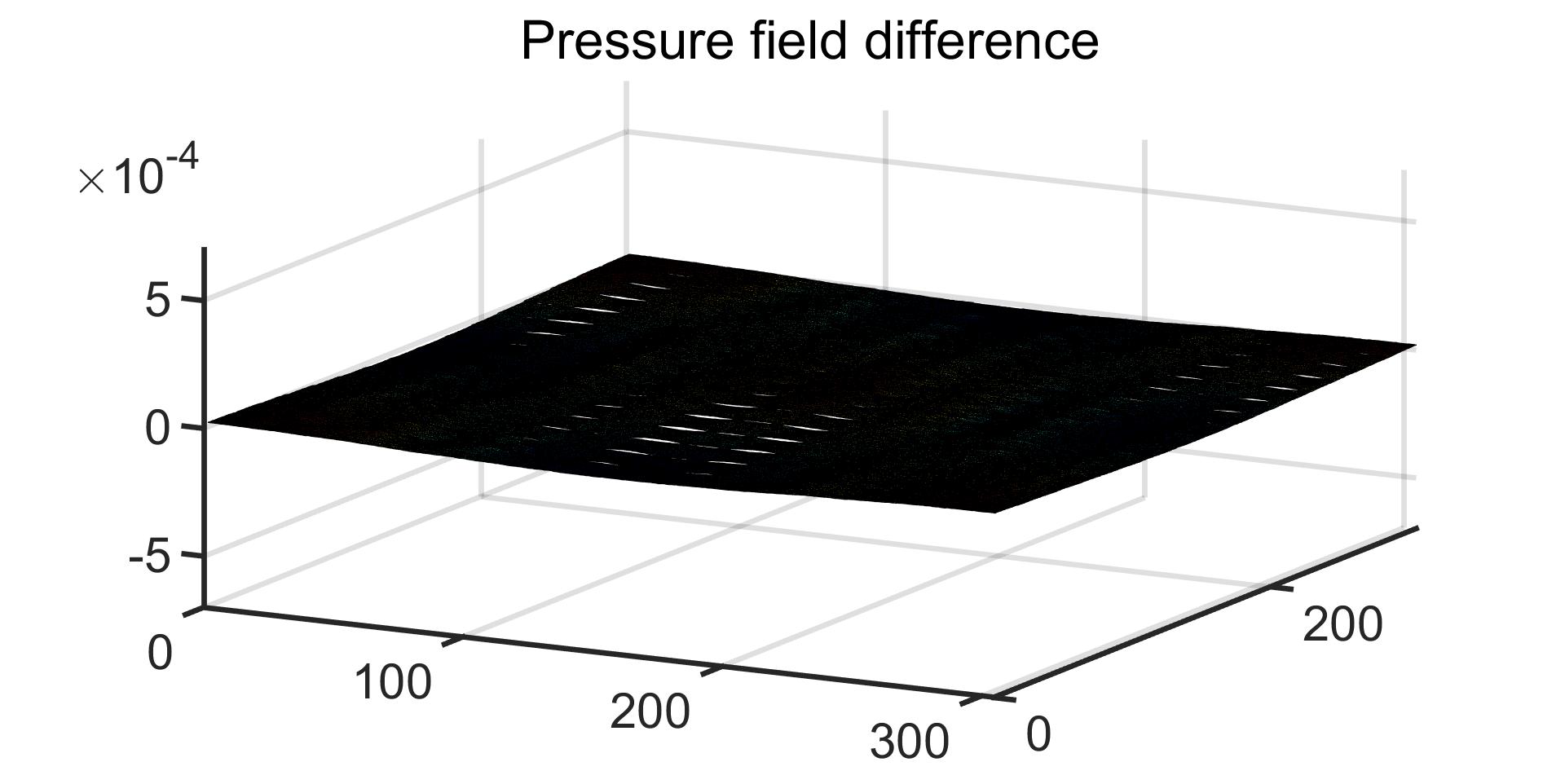}
	\end{minipage}
	\caption{Difference in multiscale solutions for the homogeneous problem with respect to the analytical solution, using different methods for Robin condition parameter $\alpha=100$: MRCM method (top), Our method (bottom).}
	\label{anacolor-100}
\end{figure}

\begin{figure}[H]
	\centering

	\begin{minipage}{0.6\textwidth}
		\centering
		\includegraphics[width = 1.0\textwidth]{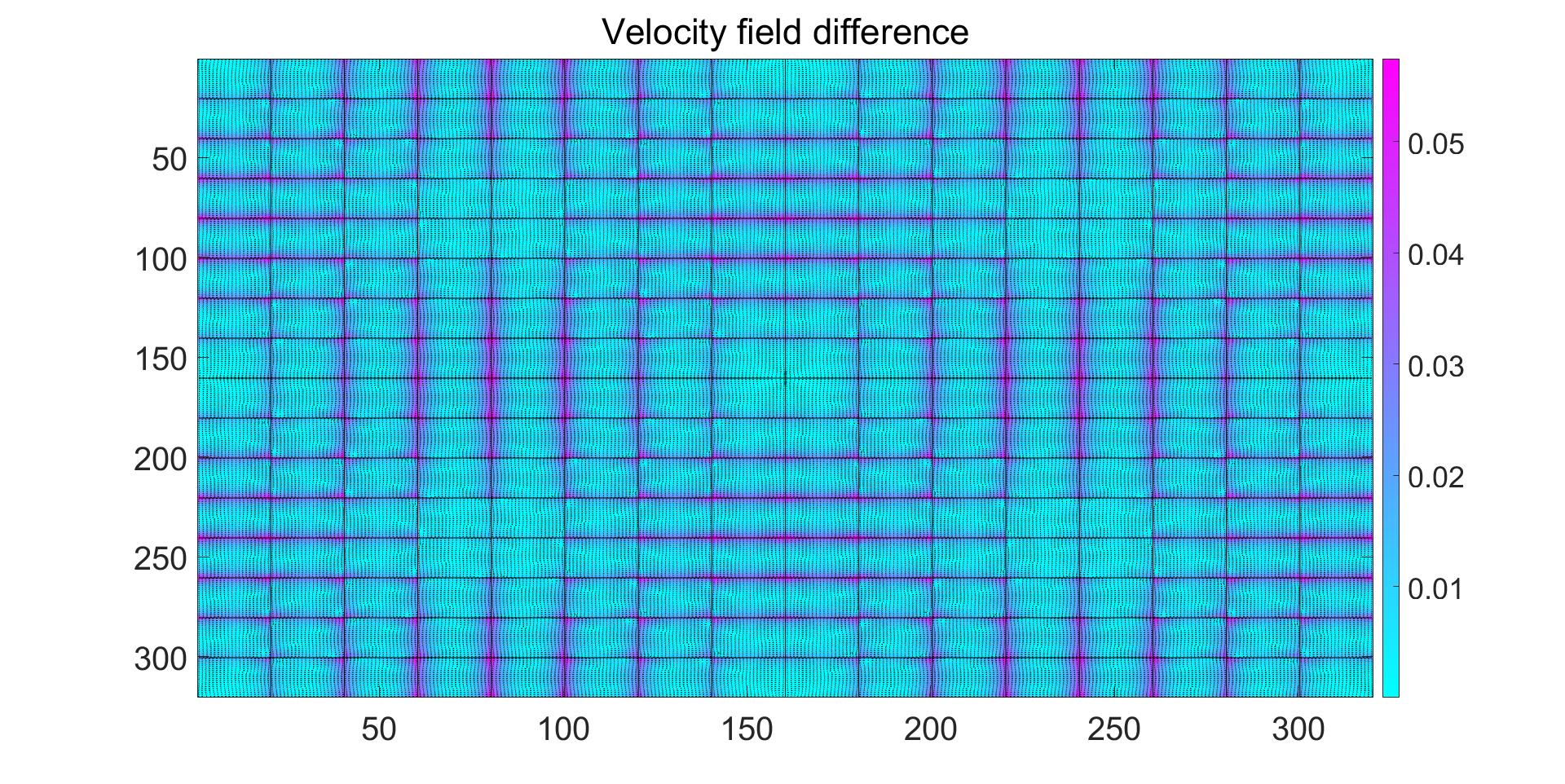}
	\end{minipage}
	\begin{minipage}{0.39\textwidth}\vspace{1cm}
		\centering
		\includegraphics[width = 1.0\textwidth]{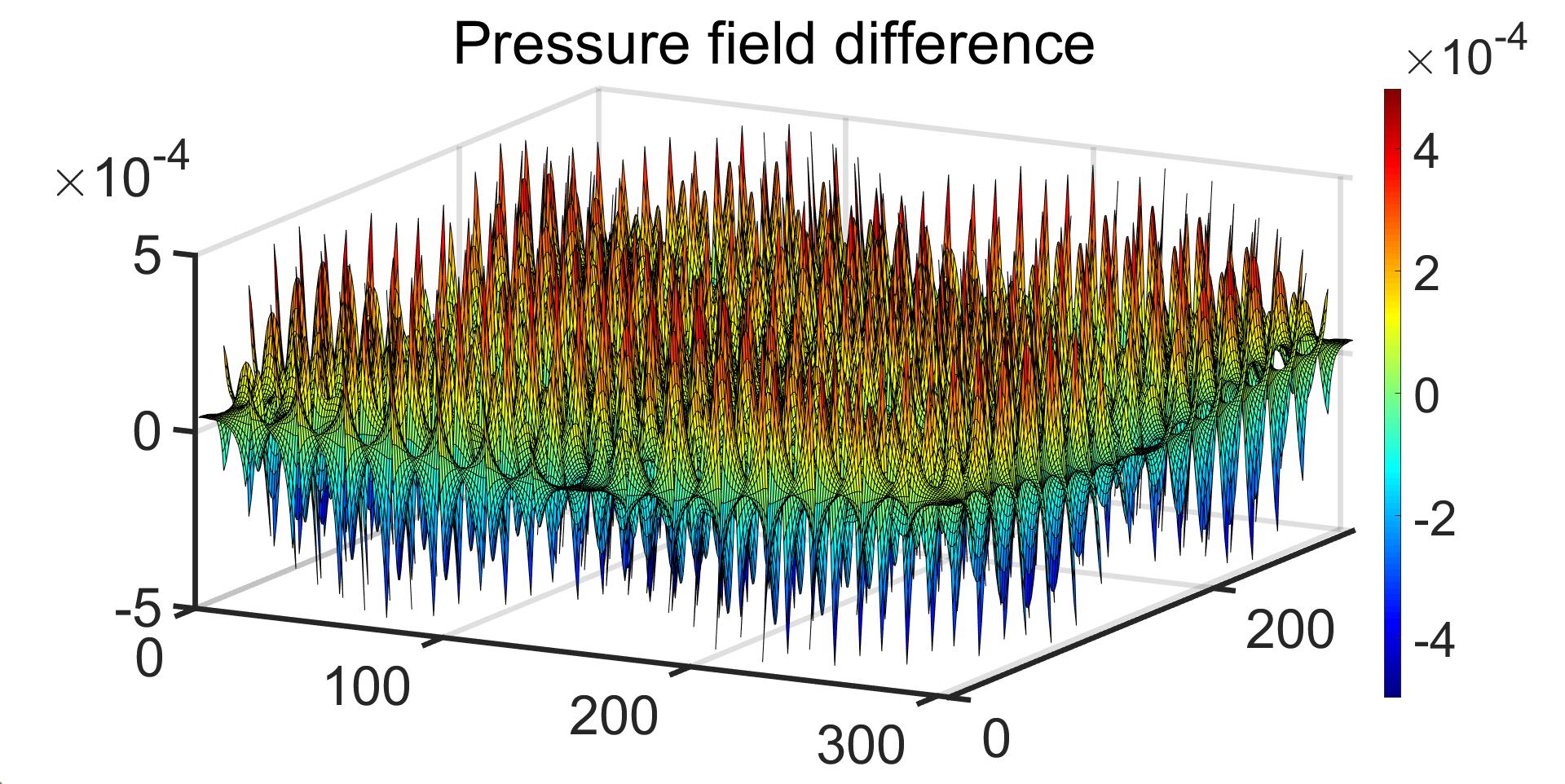}
	\end{minipage}

	\begin{minipage}{0.6\textwidth}
		\centering
		\includegraphics[width = 1.0\textwidth]{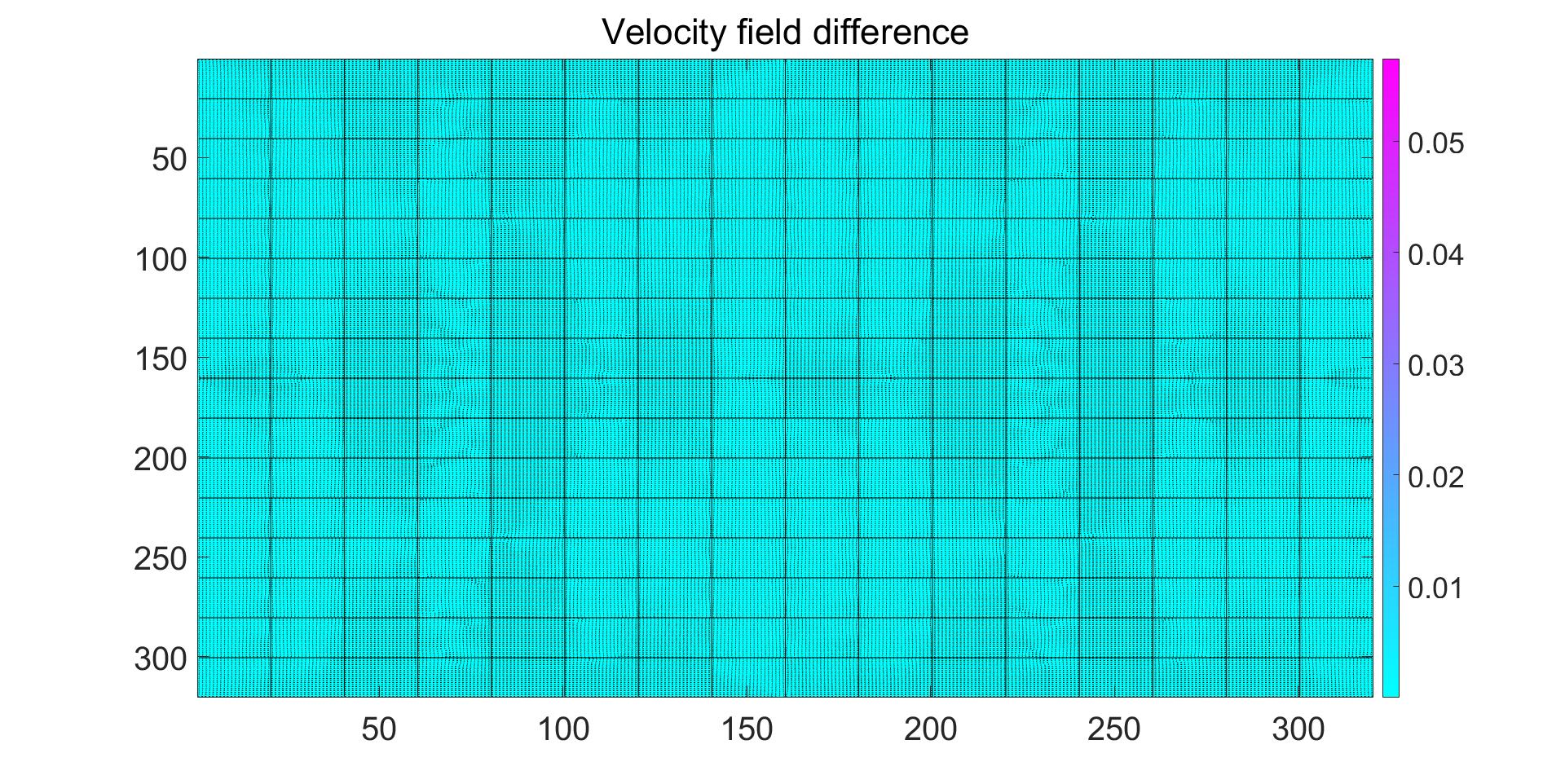}
	\end{minipage}
	\begin{minipage}{0.39\textwidth}\vspace{1cm}
		\centering
		\includegraphics[width = 1.0\textwidth]{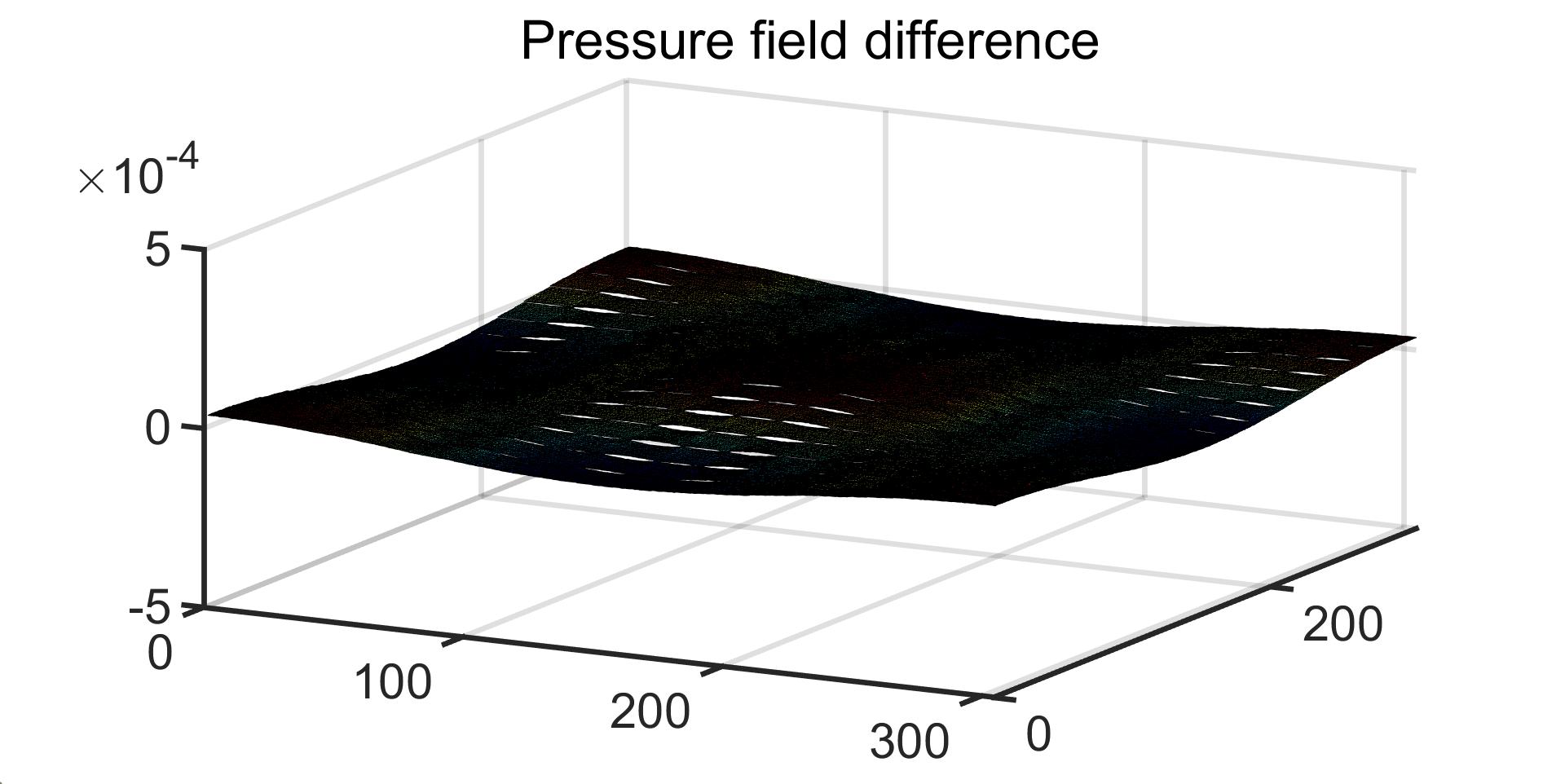}
	\end{minipage}
	\caption{Difference in multiscale solutions for the homogeneous problem with respect to the analytical solution, using different methods for Robin condition parameter $\alpha=10^8$: MRCM method (top), Our method (bottom).}
	\label{anacolor-1e8}
\end{figure}

Figures (\ref{anacolor-1e-8}) through (\ref{anacolor-1e8}) display the errors relative to the analytical solution for various values of parameter $\alpha$. 
Note that our method consistently exhibits less deviation in both pressure and flux variables compared to the original MRCM method. This observation underscores the reason behind the improvements mentioned earlier.

Despite achieving significant improvements, the computational cost of our method is only slightly higher than the of the original MRCM method. 
Table \ref{tab:cost_table} provides a summary of the costs associated with both the original MRCM and our method across various settings, all under the condition of a $16 \times 16$ decomposition of the domain.

%\bigskip

\begin{minipage}{1\textwidth}
	\centering
%	\begin{tabular}{||c|c|c|c|c|c||}
%		\hline
%		  &  $l$ & $kS$ & NLC & size of LP & size of IP \\ [0.5ex]
%		\hline\hline
%		MRCM & 0 & 0 & 9 & 20*20 & 4M(M-1)*4M(M-1)\\[1ex]
%		\hline
%		OC-2 & 2 & 0 & 5 & 24*24 & 2M(M-1)*2M(M-1)\\
%		\hline
%		OC-2,2S & 2 & 2 & 7 & 24*24 & 2M(M-1)*2M(M-1)\\
%		\hline
%		OL-2 & 2 & 0 & 9 & 24*24 & 4M(M-1)*4M(M-1)\\
%		\hline
%		OL-2,2S & 2 & 2 & 11 & 24*24 & 4M(M-1)*4M(M-1)\\
%		\hline
%		OL-4 & 4 & 0 & 9 & 28*28 & 4M(M-1)*4M(M-1)\\
%		\hline
%		OL-4,4S & 4 & 4 & 13 & 28*28 & 4M(M-1)*4M(M-1)\\[1ex]
%		\hline
%	\end{tabular}
\begin{tabular}{lccccc}
  \hline
  \textbf{Method} & \textbf{$l$} & \textbf{$kS$} & \textbf{NLC} & \textbf{Size of LP} & \textbf{Size of IP} \\ 
  \hline
  \textit{MRCM} & 0 & 0 & 9 & $20 \times 20$ & $4M(M-1) \times 4M(M-1)$\\
  \textit{OC-2} & 2 & 0 & 5 & $24 \times 24$ & $2M(M-1) \times 2M(M-1)$\\
  \textit{OC-2,2S} & 2 & 2 & 7 & $24 \times 24$ & $2M(M-1) \times 2M(M-1)$\\
  \textit{OL-2} & 2 & 0 & 9 & $24 \times 24$ & $4M(M-1) \times 4M(M-1)$\\
  \textit{OL-2,2S} & 2 & 2 & 11 & $24 \times 24$ & $4M(M-1) \times 4M(M-1)$\\
  \textit{OL-4} & 4 & 0 & 9 & $28 \times 28$ & $4M(M-1) \times 4M(M-1)$\\
  \textit{OL-4,4S} & 4 & 4 & 13 & $28 \times 28$ & $4M(M-1) \times 4M(M-1)$\\
  \hline
\end{tabular}
	\captionof{table}{A comparison of the complexity of the algorithms: MRCM $\times$ MRCM-OS,  in which 
   $l$: oversampling size; $k$: $\#$ smoothing steps; NLC: $\#$ of local problems per core; LP: local problem IP: interface problem.}	
	\label{tab:cost_table}
	\vspace{0.5cm}
\end{minipage}

In Table \ref{tab:cost_table}, $l$ denotes the oversampling size $lh$, $k$ represents the number of smoothing steps, 
NLC  is the number of local problems computed for each core, LP refers to the local problem, while IP denotes the interface problem.

\begin{rmk} 
We remark that, by reusing the LDL$^T$-factorization, our method, even with an oversampling size of $4h$ and 4 smoothing steps, our computational costs do not significantly exceed those of the original MRCM. 
\end{rmk}

\subsubsection{Heterogeneous problem}
%In this section, we will perform similar procedure as in previous section to present the result in the heterogeneous problem.
In this section, we present the results obtained for the heterogeneous problem for different values of parameter $\alpha$ .

In Fig.(\ref{het-number}), we employ an oversampling size of $4h$ to study the optimal number of smoothing steps and present the $L^2(\Omega)$ 
norm relative error with respect to the fine grid solution. Note that from the results reported in Fig. (\ref{het-number}), the optimal number of smoothing steps lies also within the range of 2 to 4, as increasing the number of steps beyond this range results in a diminishing slope.

\begin{figure}[H]
	\centering
	\begin{minipage}{0.49\textwidth}
		\includegraphics[width = 1.0\textwidth]{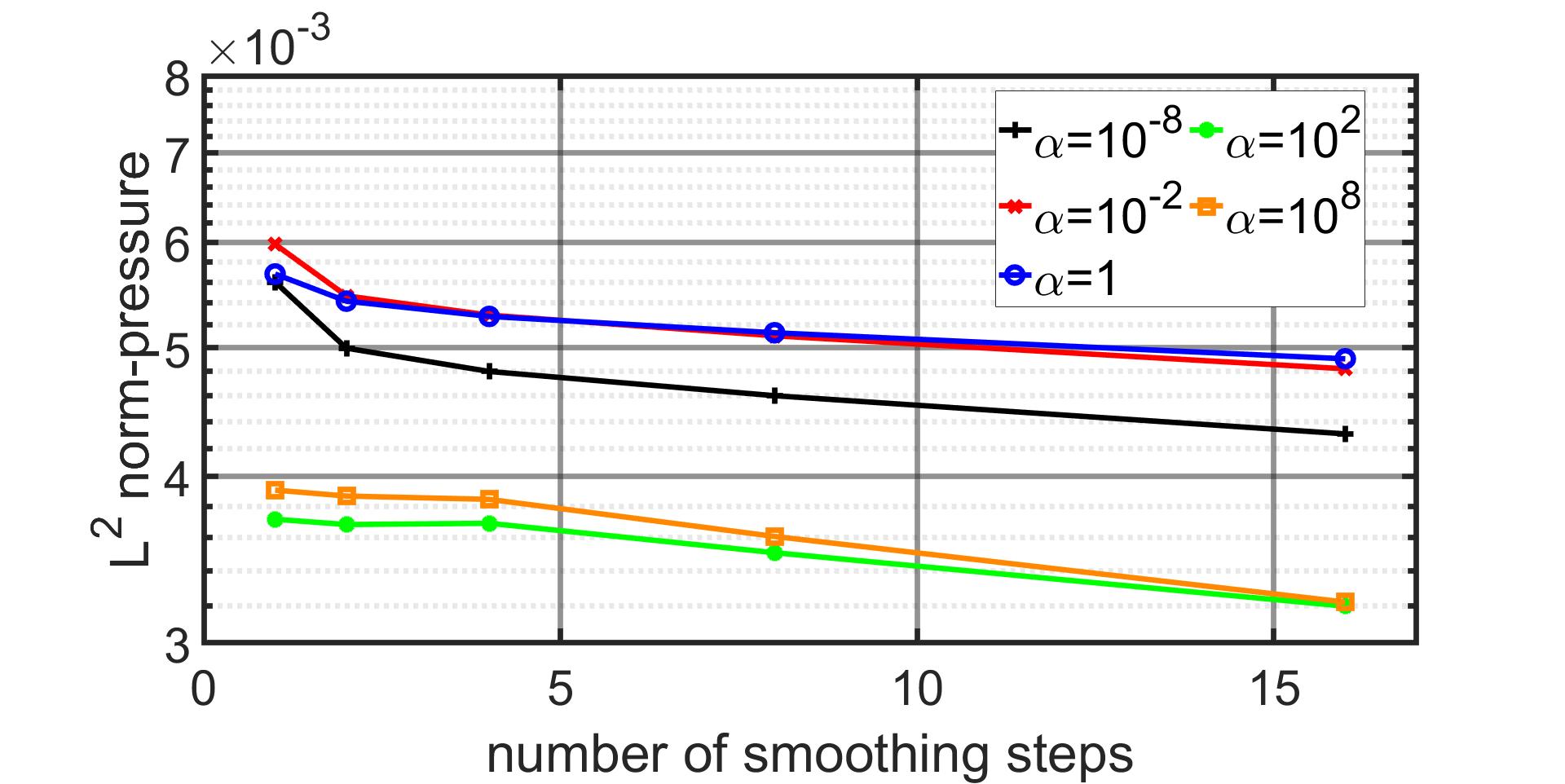}
	\end{minipage}
	\begin{minipage}{0.49\textwidth}
		\includegraphics[width = 1.0\textwidth]{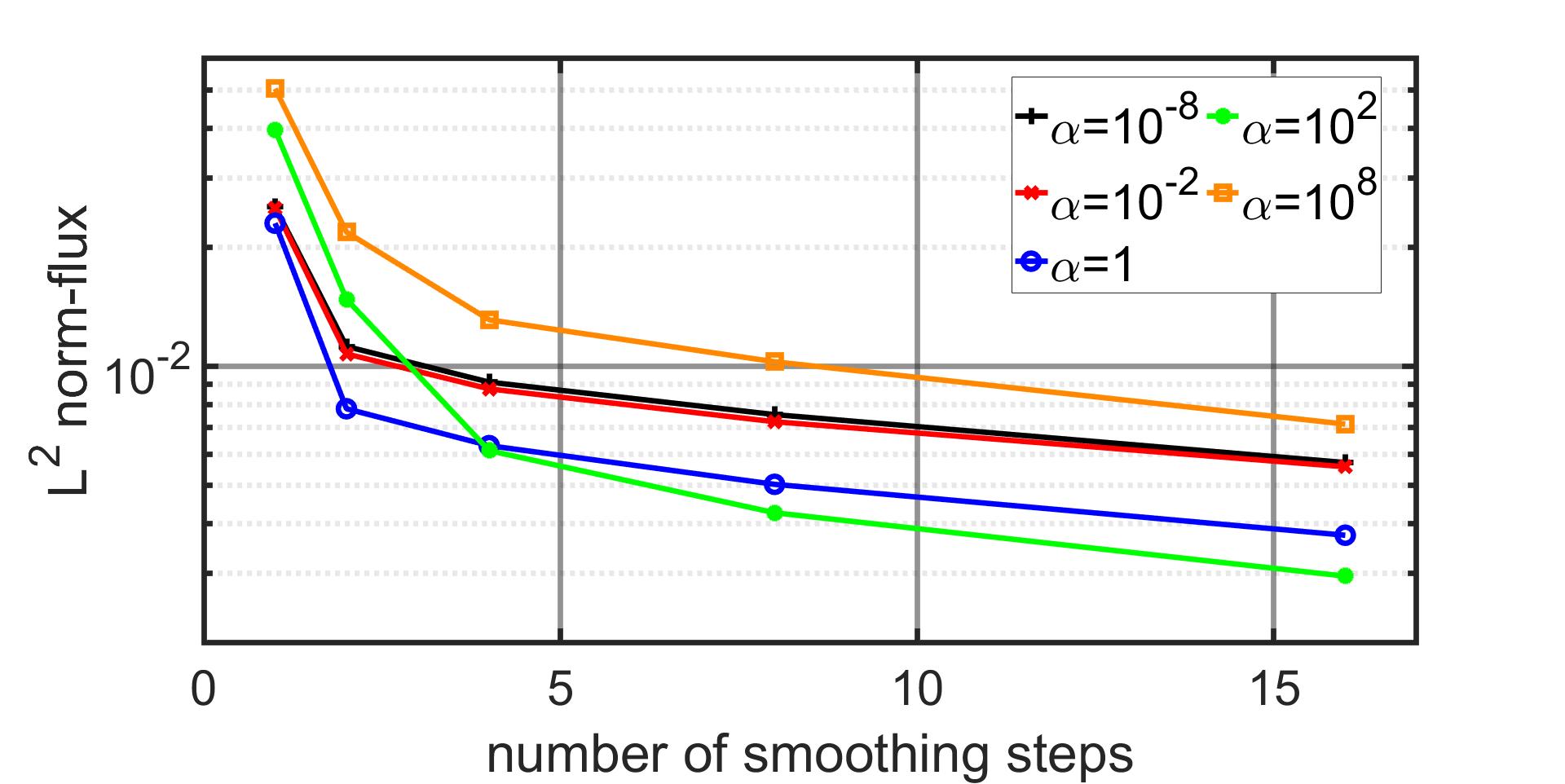}
	\end{minipage}
	\caption{Exploring the optimal number of smoothing steps for the heterogeneous problem with oversampling size $2h$: Absolute pressure error (left) and absolute flux error (right).}
	\label{het-number}
\end{figure}

%Next, we embark on an alpha study concerning under the condition where the oversampling size is $2h$ and $4h$, and the number of smoothing steps is 2 and 4(the best result with $4h$ and 4 smoothing steps has been show above, and here we use as a standard for other outcome). The chart below displays the $L^2(\Omega)$ norm relative error relative to the fine grid solution.

Next, in Fig.(\ref{het-alphastudy}) we present a study on parameter $\alpha$ and set the oversampling size to $2h$ and $4h$. The number of smoothing steps is set to 2 and 4 (the good result obtained with $4h$ for the oversampling size and 4 smoothing steps has been discussed previously, and here it serves as a benchmark for comparison). Fig.(\ref{het-alphastudy}) shows the $L^2(\Omega)$ norm relative error with respect to the fine grid solution.

\begin{figure}[H]
	\centering
	\begin{minipage}{0.49\textwidth}
		\includegraphics[width = 1.0\textwidth]{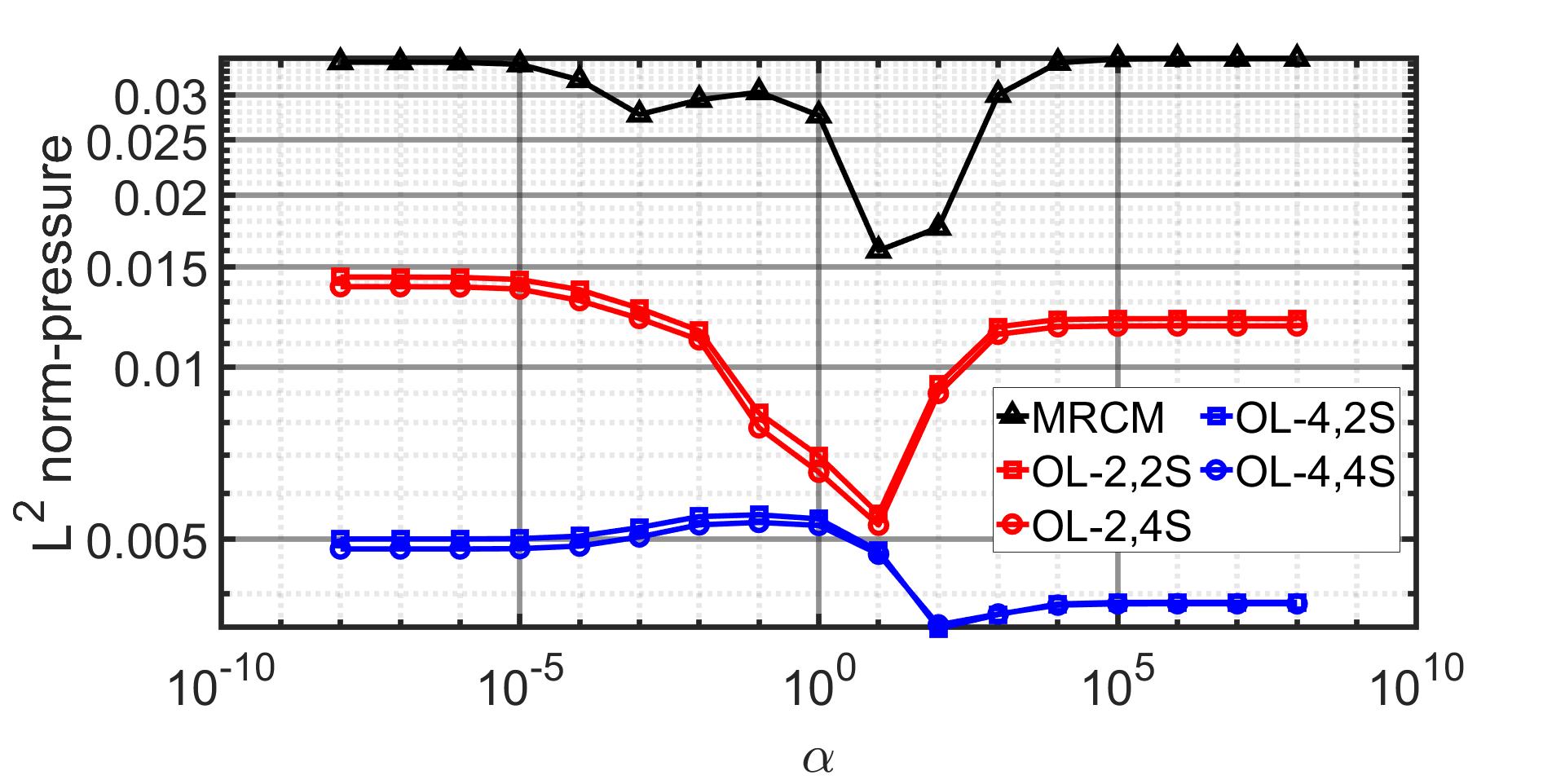}
	\end{minipage}
	\begin{minipage}{0.49\textwidth}
		\includegraphics[width = 1.0\textwidth]{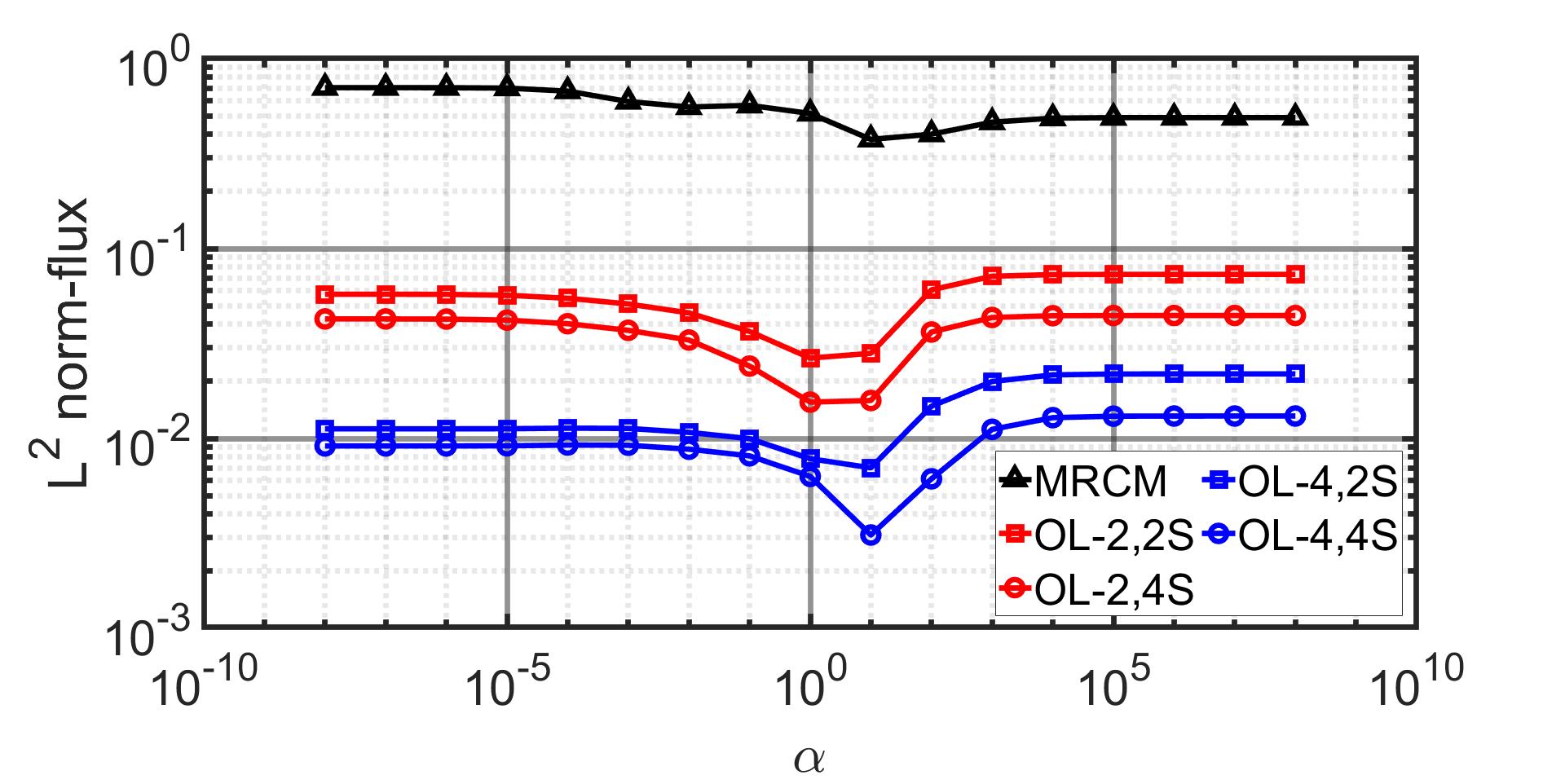}
	\end{minipage}
	\caption{$\alpha$ parameter study for the heterogeneous problem with oversampling sizes $2h$ and $4h$, and 2 or 4 smoothing steps: Pressure (left) and flux (right).}
	\label{het-alphastudy}
\end{figure}

%We've already made a conclusion for the heterogenous problem in the previous section that our method has great improvement, and the detailed reasons for such improvements are elucidated below.
The results reported above for the heterogeneous problem show substantial improvement of MRMC-OS over the original MRCM. 
Next, we discuss the differences between the two procedures taking advantage colored pictures.
Figs. (\ref{hete-OL4-pic}), (\ref{compare1e8}), and (\ref{compare1e-8}) provide visual comparisons of the multiscale solution (velocity and pressure) generated by both the original MRCM method and our method under three distinct scenarios.
In Fig. (\ref{compare1e8}), we set the Robin condition parameter to a very large value ($\alpha=10^8$). In this scenario, the MRCM outcome resembles the MHM solution, with flux continuity dominating along the interface $\Gamma$. However, the pressure continuity along $\Gamma$ is notably poor, featuring significant jumps.
In Fig. (\ref{compare1e-8}), we set the Robin condition parameter to a very small value ($\alpha=10^{-8}$). Consequently, along the interface $\Gamma$, we observe 
improved continuity in the pressure, with significant flux jumps. In this case, the MRCM outcome resembles the MMMFEM solution.
Concerning Fig. (\ref{hete-OL4-pic}), we select an intermediate value for the Robin condition parameter ($\alpha=1$). This choice ensures that neither pressure nor flux continuity dominates along $\Gamma$. As a result, both pressure and flux exhibit only minor jumps within the interface. 
Fig. (\ref{hete-OL4-pic}) depicted that when the value of $\alpha$ is close to 1, the method yields more accurate results compared to cases where $\alpha$ is either very large or very small.
\begin{figure}[H]
	\centering

	\begin{minipage}{0.9\textwidth}
		\centering
		\includegraphics[width = 1.0\textwidth]{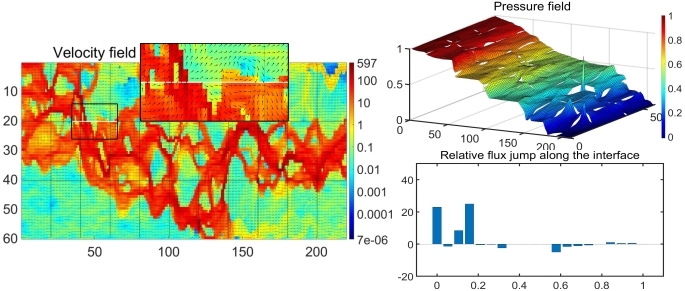}
	\end{minipage}
	
	\begin{minipage}{0.9\textwidth}\vspace{1cm}
		\centering
		\includegraphics[width = 1.0\textwidth]{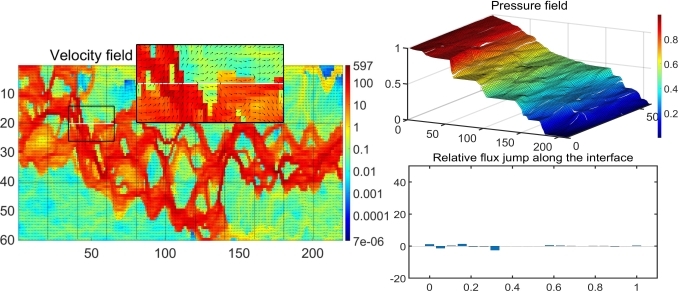}
	\end{minipage}

	\caption{Multiscale solution for heterogeneous problem produced by different methods for Robin condition parameter $\alpha=10^8$ : MRCM method (top), our method (bottom).}
	\label{compare1e8}
\end{figure}

\begin{figure}[H]
	\centering

	\begin{minipage}{0.9\textwidth}
		\centering
		\includegraphics[width = 1.0\textwidth]{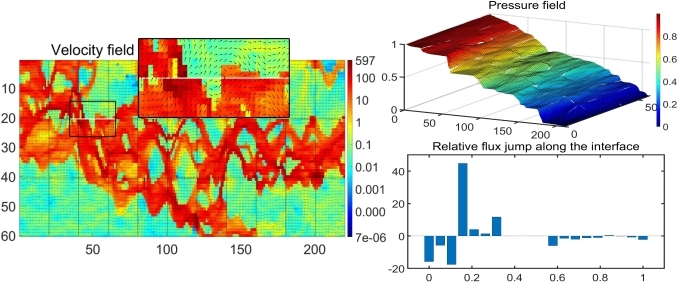}
	\end{minipage}
	\begin{minipage}{0.9\textwidth}\vspace{1cm}
		\centering
		\includegraphics[width = 1.0\textwidth]{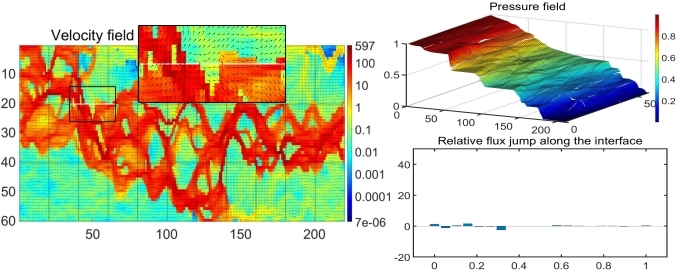}
	\end{minipage}

	\caption{Multiscale solution for heterogeneous problem produced by different methods for Robin condition parameter $\alpha=10^{-8}$ : MRCM method (top), our method (bottom).}
	\label{compare1e-8}
\end{figure}

%Here, we can see that for the pressure field, the original MRCM is less continuous than our method along the interface part, especially when $\alpha=10^8$. And for the velocity field, the original MRCM has a larger flux jump along the interface. That is the reason why our method can achieve a two-order-of-magnitude improvement for pressure and flux variables.
In summary, it is observed that the continuity of the pressure field in the original MRCM is comparatively less pronounced than in our method, particularly evident when $\alpha=10^8$. Additionally, the original MRCM exhibits a more pronounced flux discontinuity along the interface for the velocity field. These findings 
provide an indication of why our methodology achieves a two-order-of-magnitude improvement in the flux variable.

\section{Conclusions}

%We have developed two special techniques, oversampling and smoothing, to enhance the original MRCM \cite{guiraldello2018multiscale}, with the goal of accurately approximating subsurface flows. The oversampling technique calculates the multiscale basis function in an oversampling region, which can reduce the small-scale error accumulated along the interface $\Gamma$ for non-overlapping partition and smoothing techniques can reduce more. numerical results show that our method can achieve two orders of magnitude improvement for flux variables and one order of magnitude for pressure improvement compared to the original MRCM for analytical problems and practical problems like heterogeneous problems, with not too many additional computational costs.
In conclusion, our study introduces two novel methodologies, oversampling and smoothing, tailored to augment the original MRCM \cite{guiraldello2018multiscale}, aimed at refining the approximation of subsurface flows. The oversampling approach strategically computes the multiscale basis function within an extended region, effectively mitigating small-scale errors along the interface $\Gamma$ in non-overlapping partitions. Further refinement is achieved through smoothing techniques. Through numerical experiments, our findings demonstrate that our proposed methodology yield a notable improvement in accuracy. Specifically, our method exhibits a two-orders-of-magnitude improvement in flux variables and a one-order-of-magnitude enhancement in pressure, as compared to the original MRCM, across both analytical and heterogeneous problems. Importantly, these improvements are achieved without significant additional computational overhead.

%I've been developing a new space called the informal space that can further reduce the error, which can be easily constructed for 3d situations compared to higher-order polynomial spaces. And I'll focus on this space in my next paper.
In addition to the aforementioned techniques, our ongoing research has led to the development of  novel informed spaces, which shows promise in further minimizing errors, particularly in three-dimensional scenarios. Unlike higher-order polynomial spaces, the informed space offers simplicity in construction while maintaining effectiveness in error reduction. Our forthcoming paper will delve deeper into the exploration and refinement of these spaces, shedding light on its potential contributions to the field of subsurface flow approximation.

\bigskip
\noindent {\bf Acknowledgments} 

The authors acknowledge the National Laboratory for Scientific Computing (LNCC/MCTI, Brazil) for providing HPC resources of the S. Dumont supercomputer, which have contributed to the research results reported within this paper (URL: http://sdumont.lncc.br). Additionally, the computing resources of the Cyber-Infrastructure Research Services at the University of Texas at Dallas Office of Information Technology were utilized. The authors would like to express their gratitude to Dr. M\'arcio Borges for his assistance in accessing the LNCC cluster. 

\bigskip
\noindent {\bf Declaration of generative AI and AI-assisted technologies in the writing process}

During the preparation of this work the authors used chatGPT in order to improve language and readability. 
After using this tool, the authors reviewed and edited the content as needed and take full responsibility for the content of the publication.

\bibliography{dilongref1,dilongref2,dilongref3}

\begin{thebibliography}{10}

\bibitem{douglas1997numerical}
J.~Douglas, F.~Furtado, and F.~Pereira.
\newblock On the numerical simulation of waterflooding of heterogeneous
  petroleum reservoirs.
\newblock {\em Computational Geosciences}, 1(2):155--190, 1997.

\bibitem{operator2009}
E.~Abreu, J.~Douglas, F.~Furtado, and F.~Pereira.
\newblock Operator splitting for three-phase flow in heterogeneous porous
  media.
\newblock {\em Communications in Computational Physics}, 6(1):72--84, 2009.

\bibitem{HouMulti}
T.Y. Hou and X-H. Wu.
\newblock {A Multiscale Finite Element Method for Elliptic Problems in
  Composite Materials and Porous Media}.
\newblock {\em Journal of Computational Physics}, 134(1):169 -- 189, 1997.

\bibitem{HouMultivol}
P.~Jenny, S.H. Lee, and H.A. Tchelepi.
\newblock Adaptive multiscale finite-volume method for multiphase flow and
  transport in porous media.
\newblock {\em Multiscale Model. Simul.}, 3(1):50--64, 2005.

\bibitem{jenny2003multi}
P.~Jenny, S.H. Lee, and H.A. Tchelepi.
\newblock Multi-scale finite-volume method for elliptic problems in subsurface
  flow simulation.
\newblock {\em Journal of Computational Physics}, 187(1):47--67, 2003.

\bibitem{chen_hou}
Z.~Chen and T.Y. Hou.
\newblock {A Mixed Multiscale Finite Element Method for Elliptic Problems with
  Oscillating Coefficients}.
\newblock {\em Math. Comp}, 72:541--576, 2003.

\bibitem{arbogast}
T.~Arbogast, G.~Pencheva, M.F. Wheeler, and I.~Yotov.
\newblock {A Multiscale Mortar Mixed Finite Element Method}.
\newblock {\em {Multiscale Modeling \& Simulation}}, 6(1):319, 2007.

\bibitem{araya2013multiscale}
R.~Araya, C.~Harder, D.~Paredes, and F.~Valentin.
\newblock {Multiscale Hybrid-Mixed Method}.
\newblock {\em SIAM Journal on Numerical Analysis}, 51(6):3505--3531, 2013.

\bibitem{guiraldello2018multiscale}
R.T. Guiraldello, R.F. Ausas, F.S. Sousa, F.~Pereira, and G.C. Buscaglia.
\newblock The {Multiscale} {Robin} {Coupled} {Method} for flows in porous
  media.
\newblock {\em Journal of Computational Physics}, 355:1--21, 2018.

\bibitem{pereira}
A.~Francisco, V.~Ginting, F.~Pereira, and J.~Rigelo.
\newblock {Design and implementation of a multiscale mixed method based on a
  nonoverlapping domain decomposition procedure}.
\newblock {\em Mathematics and Computers in Simulation}, 99:125 -- 138, 2014.

\bibitem{Recursive2023}
E.~Abreu, P.~Ferraz, A.M.E. Santo, et~al.
\newblock Recursive formulation and parallel implementation of multiscale mixed
  methods.
\newblock {\em Journal of Computational Physics}, 473:111681, 2023.

\bibitem{HPC2022}
A.~Jaramillo, R.T. Guiraldello, S.~Paz, et~al.
\newblock Towards hpc simulations of billion-cell reservoirs by multiscale
  mixed methods.
\newblock {\em Computational Geosciences}, 26(3):481--501, 2022.

\bibitem{interface2019}
R.T. Guiraldello, R.F. Ausas, F.S. Sousa, et~al.
\newblock Interface spaces for the multiscale robin coupled method in reservoir
  simulation.
\newblock {\em Mathematics and Computers in Simulation}, 164:103--119, 2019.

\bibitem{interface2021}
F.F. Rocha, F.S. Sousa, R.F. Ausas, et~al.
\newblock Interface spaces based on physics for multiscale mixed methods
  applied to flows in fractured-like porous media.
\newblock {\em Computer Methods in Applied Mechanics and Engineering},
  385:114035, 2021.

\bibitem{velocityMM2020}
R.T. Guiraldello, R.F. Ausas, F.S. Sousa, et~al.
\newblock Velocity postprocessing schemes for multiscale mixed methods applied
  to contaminant transport in subsurface flows.
\newblock {\em Computational Geosciences}, 24:1141--1161, 2020.

\bibitem{MMtwo2020}
F.F. Rocha, F.S. Sousa, R.F. Ausas, et~al.
\newblock Multiscale mixed methods for two-phase flows in high-contrast porous
  media.
\newblock {\em Journal of Computational Physics}, 409:109316, 2020.

\bibitem{MRCM2022}
F.F. Rocha, F.S. Sousa, R.F. Ausas, et~al.
\newblock A multiscale robin-coupled implicit method for two-phase flows in
  high-contrast formations.
\newblock {\em Journal of Computational Science}, 60:101592, 2022.

\bibitem{MPM1}
A.~Ali, H.~Mankad, F.~Pereira, and F.S. Sousa.
\newblock The multiscale perturbation method for second order elliptic
  equations.
\newblock {\em Applied Mathematics and Computation}, 387:125023, 2020.
\newblock Computational and Industrial Mathematics.

\bibitem{MPM2}
F.F. Rocha, H.~Mankad, F.S. Sousa, and F.~Pereira.
\newblock The multiscale perturbation method for two-phase reservoir flow
  problems.
\newblock {\em Applied Mathematics and Computation}, 421:126908, 2022.

\bibitem{Douglas1993}
J.~Douglas, P.J. Paes-Leme, J.E. Roberts, and J.P. Wang.
\newblock A parallel iterative procedure aplicable to the approximate solution
  of second order partial differential equations by mixed finite element
  methods.
\newblock {\em Numer. Math.}, 65(1):95--108, 1993.

\bibitem{first-over}
T.Y. Hou and X.H. Wu.
\newblock A multiscale finite element method for elliptic problems in composite
  materials and porous media.
\newblock {\em JOURNAL OF COMPUTATIONAL PHYSICS}, 134:169--189, 1997.

\bibitem{oversampleprove1}
T.Y. Hou, X.H. Wu, and Z.Q. Cai.
\newblock Convergence of a multiscale finite element method for elliptic
  problems with rapidly oscillating coefficients.
\newblock {\em Math. Comp.}, 68:913--943, 1999.

\bibitem{oversampleprove2}
Y.R. Efendiev, T.Y. Hou, and X.H. Wu.
\newblock The convergence of nonconforming multiscale finite element methods.
\newblock {\em Math. Comp.}, 37:888--910, 2000.

\bibitem{2004overfv1}
V.~Ginting.
\newblock Analysis of two-scale finite volume element method for elliptic
  problem.
\newblock {\em Journal of Numerical Mathematics}, 12(2):119, 2004.

\bibitem{2007overfv}
L.J. Durlofsky, Y.~Efendiev, and V.~Ginting.
\newblock An adaptive local–global multiscale finite volume element method
  for two-phase flow simulations.
\newblock {\em Advances in Water Resources}, 30(3):576--588, 2007.

\bibitem{2004overfe1}
Y.~Efendiev, T.Y. Hou, and V.~Ginting.
\newblock Multiscale finite element methods for nonlinear problems and their
  application.
\newblock {\em Communications in Mathematical Sciences}, 2(4):553--589, 2004.

\bibitem{2004overfe2}
T.Y. Hou, X.H. Wu, and Y.~Zhang.
\newblock Removing the cell resonance error in the multiscale finite element
  method via a petrov-galerkin formulation.
\newblock {\em Communications in Mathematical Sciences}, 2(2):185--205, 2004.

\bibitem{2005overfe1}
G.~Allaire and R.~Brizzi.
\newblock A multiscale finite element method for numerical homogenization.
\newblock {\em Multiscale Modeling \& Simulation}, 4(3):790 -- 812, 2005.

\bibitem{2008over1}
J.~Nolen, G.~Papanicolaou, and O.~Pironneau.
\newblock A framework for adaptive multiscale methods for elliptic problems.
\newblock {\em Multiscale Modeling \& Simulation}, 7(1):171--196, 2008.

\bibitem{2008overfe2}
Z.~Chen, M.~Cui, T.Y. Savchuk, and X.~Yu.
\newblock The multiscale finite element method with nonconforming elements for
  elliptic homogenization problems.
\newblock {\em Multiscale Modeling \& Simulation}, 7(2):517--538, 2008.

\bibitem{2011over}
Y.~Efendiev and J.~Galvis.
\newblock Coarse-grid multiscale model reduction techniques for flows in
  heterogeneous media and applications.
\newblock {\em Numerical analysis of multiscale problems}, 83, 2011.

\bibitem{2013overfe2}
P.~Henning and D.~Peterseim.
\newblock Oversampling for the multiscale finite element method.
\newblock {\em Multiscale Modeling \& Simulation}, 11(4):1149--1175, 2013.

\bibitem{2017over1}
D.~Paredes, F.~Valentin, and H.~Versieux.
\newblock On the robustness of multiscale hybrid-mixed methods.
\newblock {\em Mathematics of Computation}, 86(304):525--548, 2017.

\bibitem{2005over}
T.Y. Hou.
\newblock Multiscale modelling and computation of fluid flow.
\newblock {\em International journal for numerical methods in fluids},
  47(8-9):707--719, 2005.

\bibitem{2006overfe1}
X.~He and L.~Ren.
\newblock A multiscale finite element linearization scheme for the unsaturated
  flow problems in heterogeneous porous media.
\newblock {\em Water Resources Research}, 42(8), 2006.

\bibitem{2008overfe1}
J.~Chu, Y.~Efendiev, V.~Ginting, and T.Y. Hou.
\newblock Flow based oversampling technique for multiscale finite element
  methods.
\newblock {\em Advances in Water Resources}, 31(4):599--608, 2008.

\bibitem{2010overfe1}
H.W. Zhang, J.K. Wu, J.~Lü, and Z.D. Fu.
\newblock Extended multiscale finite element method for mechanical analysis of
  heterogeneous materials.
\newblock {\em Acta mechanica sinica}, 26:899--920, 2010.

\bibitem{2012over}
H.W. Zhang, J.K. Wu, and J.~Lv.
\newblock A new multiscale computational method for elasto-plastic analysis of
  heterogeneous materials.
\newblock {\em Computational mechanics}, 49:149--169, 2012.

\bibitem{2013overfe1}
S.~Zhang, D.S. Yang, H.W. Zhang, and Y.G. Zheng.
\newblock Coupling extended multiscale finite element method for thermoelastic
  analysis of heterogeneous multiphase materials.
\newblock {\em Computers \& Structures}, 121:32--49, 2013.

\bibitem{2021over}
Y.~Wang, E.~Chung, and L.~Zhao.
\newblock Constraint energy minimization generalized multiscale finite element
  method in mixed formulation for parabolic equations.
\newblock {\em Mathematics and Computers in Simulation}, 188:455--475, 2021.

\bibitem{2005book}
B.~Engquist, P.~Lötstedt, and O.~Runborg.
\newblock {\em Multiscale methods in science and engineering}.
\newblock {}. Springer-Verlag Berlin Heidelberg, 2005.

\bibitem{ExpandedMixed}
L.~Jiang, D.~Copeland, and J.D. Moulton.
\newblock Expanded mixed multiscale finite element methods and their
  applications for flows in porous media.
\newblock {\em Multiscale Modeling \& Simulation}, 10(2):418--450, 2012.

\bibitem{over2017}
Q.F. Zhang, Z.Q. Huang, J.~Yao, et~al.
\newblock A multiscale mixed finite element method with oversampling for
  modeling flow in fractured reservoirs using discrete fracture model.
\newblock {\em Journal of Computational and Applied Mathematics}, 323:95--110,
  2017.

\bibitem{GMFE2015}
E.~Chung, Y.~Efendiev, and C.S. Lee.
\newblock Mixed generalized multiscale finite element methods and applications.
\newblock {\em Multiscale Modeling \& Simulation}, 13(1):338--366, 2015.

\bibitem{GMFE2017}
Y.~Yang, S.~Fu, and E.T. Chung.
\newblock Online mixed multiscale finite element method with oversampling and
  its applications.
\newblock {\em Journal of Scientific Computing}, 82(31), 2017.

\bibitem{DDintro}
V.~Dolean, P.~Jolivet, and F.~Nataf.
\newblock {\em An Introduction to Domain Decomposition Method: Algorithms,
  theory, and Parallel Implementation}.
\newblock {Computational Science and Engineering}. Siam, 2015.

\bibitem{DomainD}
B.F. Smith, P.E. Bjørstad, and W.D. Gropp.
\newblock {\em Domain Decomposition: Parallel Multilevel Methods for Elliptic
  Partial Differential Equations}.
\newblock {}. Cambridge University Press, 1996.

\bibitem{guiraldello2018interface}
R.T. Guiraldello, R.F. Ausas, F.S. Sousa, F.~Pereira, and G.C. Buscaglia.
\newblock {Interface spaces for the Multiscale {Robin} Coupled Method in
  reservoir simulation}.
\newblock {\em Mathematics and Computers in Simulation}, 164:103 -- 119, 2019.

\bibitem{RaviartThomas::1977}
P.A Raviart and J.M Thomas.
\newblock A mixed finite element method for 2nd order elliptic problems.
\newblock In {\em Mathematical Aspects of the Finite Elements Method}, Lecture
  Notes in Mathematics, 606, pages 292--315. Springer, Berlin, 1977.

\bibitem{douglas}
J.~Douglas, P.J. Paes-Leme, J.E. Roberts, and J.P. Wang.
\newblock {A Parallel Iterative Procedure Applicable to the Approximate
  Solution of Second Order Partial Differential Equations by Mixed Finite
  Element Methods.}
\newblock {\em Numer. Math.}, 65(1):95--108, 1993.

\end{thebibliography}
\bibliographystyle{unsrt}

\end{document}